\documentclass[12pt, a4paper]{amsart}
\usepackage{amscd, amssymb, pb-diagram}
\usepackage{mathtools}
\usepackage{mathrsfs}
\usepackage{enumerate}
\usepackage[shortlabels]{enumitem}
\usepackage{amsmath}
\usepackage[margin=2.9cm]{geometry}
\usepackage{amscd, amssymb, pb-diagram}
\usepackage{tikz-cd}
\usepackage{amsthm} 
\usepackage[nospace,noadjust]{cite}
\usepackage{lmodern}
\usepackage{graphicx}
\graphicspath{{images/}{../images/}}
\usepackage[toc,page]{appendix}
\usepackage[normalem]{ulem}
\usepackage{emptypage}
\usepackage{hyperref}
\usepackage{subfiles}
\usepackage{datetime}

\usepackage{tensor}
\usepackage{wrapfig}
\usepackage{tikz}
\usepackage{multicol}
\usepackage[autostyle]{csquotes}

\allowdisplaybreaks

\def\supp{\operatorname{supp}}

\def\dim{\operatorname{dim}}

\def\id{\operatorname{id}}

\def\sup{\operatorname{sup}}
\def\max{\operatorname{max}}
\def\min{\operatorname{min}}

\def\id{\operatorname{id}}
\def\supp{\operatorname{supp}}

\def\Ext{\operatorname{Ext}}

\def\ad{\operatorname{Ad}}

\def\supp{\operatorname{supp}}
\def\dim{\operatorname{dim}}
\def\id{\operatorname{id}}

\def\sup{\operatorname{sup}}

\def\max{\operatorname{max}}
\def\min{\operatorname{min}}

\def\mul{\operatorname{m}}
\def\Leb{\operatorname{Leb}}

\def\ent{\operatorname{h}}

\def\cl{\operatorname{cl}}
\def\diam{\operatorname{diam}}

\def\supp{\operatorname{supp}}
\def\Lip{\operatorname{Lip}}

\def\Kt{\operatorname{K}}
\def\KKt{\operatorname{KK}}
\def\Ext{\operatorname{Ext}}
\def\Li{\operatorname{Li}}
\def\rank{\operatorname{rank}}

\newcommand{\gs}{G^s(Q)}
\newcommand{\gu}{G^u(P)}

\makeatletter
\def\Ddots{\mathinner{\mkern1mu\raise\p@
\vbox{\kern7\p@\hbox{.}}\mkern2mu
\raise4\p@\hbox{.}\mkern2mu\raise7\p@\hbox{.}\mkern1mu}}
\makeatother

\newcommand{\ep}{\varepsilon}

\tikzstyle{vertex}=[circle]
\tikzstyle{goto}=[->,shorten >=1pt,>=stealth,semithick]

\newtheorem{thm}{Theorem}[section]
\newtheorem{cor}[thm]{Corollary}

\newtheorem{lemma}[thm]{Lemma}
\newtheorem{prop}[thm]{Proposition}

\newtheorem*{thm2}{Main Theorem}

\theoremstyle{definition}
\newtheorem{definition}[thm]{Definition}

\newtheorem{notation}[thm]{Notation}

\theoremstyle{remark}
\newtheorem{remark}[thm]{Remark}

\newtheorem*{Acknowledgements}{Acknowledgements}

\makeatletter
\@namedef{subjclassname@2020}{%
  \textup{2020} Mathematics Subject Classification}
\makeatother

\begin{document}

\begin{abstract}
A fundamental ingredient in the noncommutative geometry program is the notion of $\KKt$-duality, often called $\Kt$-theoretic Poincar\'{e} duality, that generalises Spanier-Whitehead duality. In this paper we construct a $\theta$-summable Fredholm module that represents the fundamental class in $\KKt$-duality between the stable and unstable Ruelle algebras of a Smale space. To find such a representative, we construct dynamical partitions of unity on the Smale space with highly controlled Lipschitz constants. This requires a generalisation of Bowen's Markov partitions. Along with an aperiodic point-sampling technique we produce a noncommutative analogue of Whitney's embedding theorem, leading to the Fredholm module.
\end{abstract}

\title[A geometric fundamental class for Smale spaces]{A geometric representative for the fundamental class in KK-duality of Smale spaces}

\author[D.M. Gerontogiannis]{Dimitris Michail Gerontogiannis}
\address{Leiden University, Niels Bohrweg 1, 2333 CA
Leiden, The Netherlands}
\email{d.m.gerontogiannis@math.leidenuniv.nl}

\author[M.F. Whittaker]{Michael F. Whittaker}
\address{University of Glasgow, Glasgow Q12 8QQ, United Kingdom}
\email{Mike.Whittaker@glasgow.ac.uk\\Joachim.Zacharias@glasgow.ac.uk}

\author[J. Zacharias]{Joachim Zacharias}
\email{Joachim.Zacharias@glasgow.ac.uk}

\maketitle

\section{Introduction}\label{intro}
In \cite[Chapter 6]{Connes}, Connes asserts that the notion of manifold in noncommutative geometry will be reached only after an understanding of duality in $\KKt$-theory. This generalises the situation of a compact spin$^c$-manifold $M$ where its fundamental class in $\Kt$-homology, that implements Poincar{\'e} duality in $\KKt$-theory, has a concrete Fredholm module representation. Specifically, it is given by the phase of a Dirac operator acting on the Hilbert space of $L^2$-spinors on $M$. Spanier-Whitehead duality in $\KKt$-theory, here referred to as $\KKt$-duality, provides a general framework to this end.

With this in mind, we bring the stable and unstable Ruelle algebras of a Smale space a step closer to being realised as noncommutative manifolds. Specifically, we build a $\theta$-summable Fredholm module representation of their $\KKt$-duality fundamental class in \cite{KPW}. The heart of our method lies in approximating Smale spaces through dynamical open covers, introduced in \cite{Gero}, that are built from Markov partitions. This representation opens the window to Lefschetz fixed point theorems for Ruelle algebras, as in \cite{EmersonLef,EEK}.

Smale spaces, defined by Ruelle \cite{Ruelle}, extend Smale's non-wandering sets of Axiom A systems \cite{Smale} to the topological setting. For a Smale space $(X,\varphi)$, Putnam \cite{Putnam_algebras} defined simple, nuclear, purely infinite groupoid $C^*$-algebras; the stable and unstable Ruelle algebras $\mathcal{R}^s(X,\varphi)$ and $\mathcal{R}^u(X,\varphi)$. The zero dimensional Smale spaces are the subshifts of finite type and their Ruelle algebras are stabilised Cuntz-Krieger algebras, thus we think of the Ruelle algebras as higher dimensional Cuntz-Krieger algebras. 

Further, the stable and unstable Ruelle algebras were shown to be $\KKt$-dual \cite{KPW}. The fundamental class $\Delta$ was given by an extension $\tau_{\Delta}$ of $\mathcal{R}^s(X,\varphi) \otimes \mathcal{R}^u(X,\varphi)$ by the compact operators. The class $\Delta$ abstractly has Fredholm module representatives due to nuclearity, Choi-Effros lifting, and Stinespring dilation. However, concrete representatives are difficult to define, particularly those that arise geometrically and have summability properties. In this paper we find such a representative.

\begin{thm2}[c.f. Theorem \ref{thm:KK_1_lift_Ruelle}]\label{introthm:KK_1_lift_Ruelle}
Let $(X,\varphi)$ be an irreducible Smale space. Then there is an odd Fredholm module over $\mathcal{R}^s(X,\varphi)\otimes \mathcal{R}^u(X,\varphi)$ that is explicitly associated with the underlying dynamics and represents $\Delta$. Moreover, there are dense Lipschitz groupoid algebras $\Lambda_s\subset \mathcal{R}^s(X,\varphi),\, \Lambda_u\subset \mathcal{R}^u(X,\varphi)$ such that the Fredholm module is $\theta$-summable on $\Lambda_s\otimes_{\text{alg}}\Lambda_u$.


\end{thm2}

Before explaining the theorem we present some context around $\KKt$-duality. This duality is a noncommutative generalisation of the classical Spanier-Whitehead duality which relates the homology of a complex with the cohomology of a dual complex. Following \cite{KS}, the separable $C^*$-algebras $A$ and $DA$ are $\KKt$-dual if there are duality classes $\mu \in \KKt_i(A \otimes DA,\mathbb C),\, \nu \in \KKt_i(\mathbb C,A \otimes DA)$ with Kasparov products 
\begin{equation}\label{Spanier-Whitehead_def}
\nu \otimes_{A} \mu =(- 1)^i1_{DA} \in KK_0(DA,DA)\quad \text{and}\quad \nu \otimes_{DA} \mu =1_{A} \in KK_0(A,A).
\end{equation}
Kasparov slant products with the duality classes produce isomorphisms between the $\Kt$-theory of $A$ and the $\Kt$-homology of the dual algebra $DA$, and vice versa. Further details appear in Subsection \ref{sec:SW-duality}. $\KKt$-duality was first introduced by Connes \cite{Connes} and Kasparov \cite{Kasparov2} as $\Kt$-theoretic Poincar{\'e} duality. As discussed in \cite{KP, KPW, KS} though, Poincar{\'e} duality relates the homology and cohomology of the same manifold given it is orientable, while the duality we consider here extends Spanier-Whitehead duality to the noncommutative framework.

The key to finding a Fredholm module representative of the fundamental class $\Delta$ is to use the underlying geometry of the Smale space. Bowen's Markov partitions \cite{Bowen2,Bowen3,Bowen4} provide a combinatorial model of arbitrary precision for Smale spaces and are essential in their dimension theory \cite{Barreira} and homology \cite{Putnam_Book}. However, Markov partitions do not fit directly in the geometric aspect of Smale spaces as they are closed covers with zero Lebesgue covering numbers, unless the Smale space is zero-dimensional. Nevertheless, one can inflate them to open covers realising the same combinatorial data, as in \cite[Theorem 6.2]{Gero}, that approximate the metric through refinement. These covers admit Lipschitz partitions of unity with controlled Lipschitz constants and yield a dense aperiodic sample of homoclinic points. Then, inspired by Whitney's Embedding Theorem \cite{Whitney} we produce an isometry. Remarkably, when the metric exhibits self-similarity, the isometry $V$ compresses a representation $\rho$ of $\mathcal{R}^s(X,\varphi)\otimes \mathcal{R}^u(X,\varphi)$ back to the extension $\tau_{\Delta}$, up to compact operators and unitary equivalence. The pair $(\rho,V)$ then gives a Fredholm module representative of $\Delta$ that is $\theta$-summable on an algebraic tensor product of Lipschitz groupoid algebras. The general problem reduces to the self-similar case as the metric is self-similar up to topological conjugacy, see Subsection \ref{Sec:Smale}. Interestingly, the first author has used \cite[Theorem 6.2]{Gero} to prove Ahlfor's regularity of Smale spaces and to study $\Kt$-homological finiteness of the Ruelle algebras. However, our main theorem is the raison d'\^etre for this approximation tool. We believe the first author's innovation has many more applications still to come.

A Fredholm module representation has also been given by Goffeng and Mesland \cite{GM} for the fundamental class \cite{KP} of Cuntz-Krieger algebras, by considering Fock spaces. Echterhoff, Emerson and Kim \cite{EEK2} obtained such a result for crossed products of complete Riemannian manifolds by countable groups acting isometrically, co-compactly and properly. However, Ruelle algebras are not associated to Fock spaces in general, and Smale spaces are typically fractal rather than smooth manifolds. A concrete representation of the $\Kt$-theory duality class of an irrational rotation algebra is given by Duwenig and Emerson in \cite{DE}. Further, the work of Rennie, Robertson and Sims \cite{RRS} on crossed products by $\mathbb Z$ with coefficient algebras satisfying $\KKt$-duality does not apply to our situation, since the coefficient algebras of $\mathcal{R}^s(X,\varphi)$ and $\mathcal{R}^u(X,\varphi)$ do not have $\KKt$-duals in general, see \cite[Subsection 5.1.2]{Gero_Thesis}. Finally, it would be interesting if our result was related to a Fredholm module representation of the fundamental extension class of Emerson \cite{EmersonDuality} for crossed products of hyperbolic group boundary actions, as in some cases these are strongly Morita equivalent to Ruelle algebras \cite{LacaSp}.

The structure of this paper is as follows. Section \ref{prelims} has several parts outlining the objects and tools required in the paper. Specifically, Subsection \ref{Sec:Smale} defines Smale spaces along with standard properties. This includes a key observation of Artigue \cite{Artigue} that all Smale spaces are topologically conjugate to self-similar Smale spaces. Subsection \ref{sec:Markovpartitions} outlines the first author's $\delta$-enlarged Markov partitions, which are key to our construction. We then introduce Smale space $C^*$-algebras in Subsection \ref{sec:Smale_C_alg}, including the stable and unstable algebras as well as their crossed products by $\mathbb Z$, called Ruelle algebras. Subsection \ref{sec:K-theoretic_prelim} briefly introduces $\KKt$-theory along with the notion of smoothness for an extension and of summability for a Fredholm module. In the final two sections of the preliminaries we introduce $\KKt$-duality and the specific duality classes for Ruelle algebras, focussing on the fundamental class.

The remainder of the paper describes the lift of the $\KKt$-duality fundamental class to the $\theta$-summable Fredholm module picture. Section \ref{sec:Towards_a_KK_1_lift} defines an untwisted representation of the tensor product $\mathcal{R}^s(X,\varphi)\otimes \mathcal{R}^u(X,\varphi)$ on an enlarged Hilbert space. Moreover, we symmetrise the dynamics of the representation in order to facilitate the construction of our compressing isometry later on. To aid the reader's understanding we include Subsection \ref{sec:intuition_isometry} that outlines the ideas behind our compressing isometry. In this section we explain how we aim to prove the main theorem and give intuition towards our analogue of Whitney's embedding.

Section \ref{sec:pou_Markov} uses a refining sequence of $\delta$-enlarged Markov partitions to construct a sequence of dynamical Lipschitz partitions of unity. Each has highly controlled Lipschitz constants. In combination with a particular choice function from the refining sequence, with values in a dense set of homoclinic points in the Smale space, we are able to construct the compressing isometry $V$ in Section \ref{sec:averaging_isom}. The isometry $V$ is the direct sum of averaging isometries $V_n$ that correspond to certain levels of the refining sequence. Also these isometries satisfy quasi-invariance properties relevant to the $\mathbb Z$-actions in the Ruelle algebras.

In the final Section \ref{sec:DilateKPW}, we define our Fredholm module representative of the fundamental class defining the $\KKt$-duality from \cite{KPW}. Here we prove the main theorem of the paper, Theorem \ref{thm:KK_1_lift_Ruelle}, through a sequence of highly technical lemmas.

\section{Preliminaries}\label{prelims}

In this section we briefly introduce Smale spaces and their $C^*$-algebras. We also recall key results of the first author's extension of Bowen's Markov partitions to refining sequences of well-behaved open covers for Smale spaces \cite{Gero}. Then we outline aspects of $\KKt$-theory needed in this paper and define $\KKt$-duality. Finally, we state the $\KKt$-duality theorem for Ruelle algebras of an irreducible Smale space and quickly recap the construction of the fundamental class from \cite{KPW}.

\subsection{Smale spaces}\label{Sec:Smale}
In this paper we focus on dynamical systems with standard topological recurrence conditions -- specifically, {\em non-wandering}, {\em irreducible} as well as {\em mixing} dynamics, see \cite[Definitions 2.1.3, 2.1.4, and 2.1.5]{Putnam_Book}. Heuristically, a Smale space is a dynamical system with a hyperbolic local product structure. Specifically, we have the following definition, originally due to Ruelle \cite{Ruelle}.

\begin{definition}[{\cite[p.19]{Putnam_Book}, \cite{Ruelle}}]\label{defSmaSpa}
A \emph{Smale space} $(X,\varphi)$ consists of an infinite compact metric space $X$ with metric $d$ along with a homeomorphism $\varphi: X \to X$ such that, there exist constants $ \ep_{X} > 0, \lambda > 1$ and a locally defined bi-continuous bracket map
\[ [\cdot, \cdot]:\{(x,y) \in X \times X : d(x,y) \leq \ep_{X}\} \mapsto X \]
satisfying the bracket axioms:
\begin{itemize}
\item[B1] $\left[ x, x \right] = x$,
\item[B2] $\left[ x, [ y, z] \right] = [ x, z]$,
\item[B3] $\left[ [ x, y], z \right] = [ x,z ]$, and
\item[B4] $\varphi[x, y] = [ \varphi(x), \varphi(y)]$,
\end{itemize}
for any $x, y, z$ in $X$ when both sides are defined. In addition, for $x \in X$ and $ 0 < \ep \leq \ep_{X}$ we can define the local stable and unstable sets
\begin{align*}
X^{s}(x, \ep) & = \{ y \in X : d(x,y) < \ep, [y,x] =x \}, \\ 
X^{u}(x, \ep) & = \{ y \in X : d(x,y) < \ep, [x,y] =x \}
\end{align*}
on which we have the contraction axioms:
\begin{itemize}
\item[C1] For $y,z \in X^s(x,\ep)$ we have $d(\varphi(y),\varphi(z)) \leq \lambda^{-1} d(y,z)$,
\item[C2] For $y,z \in X^u(x,\ep)$ we have $d(\varphi^{-1}(y),\varphi^{-1}(z)) \leq \lambda^{-1} d(y,z)$.
\end{itemize}
\end{definition}

The local product structure is encoded in the bracket axioms. Specifically, for $x\in X$ and $0<\ep \leq \ep_X/2$, it follows that the bracket map 
\begin{equation}\label{eq:bracketmap}
[\cdot, \cdot]:X^u(x,\varepsilon)\times X^s(x,\varepsilon)\to X
\end{equation}
is a homeomorphism onto its image. Moreover, there is some $0<\ep_X'\leq \ep_X/2$ that for $x,y \in X$ with $d(x,y)\leq \ep_X'$ one has $d(x,[x,y]),d(y,[x,y])< \ep_X/2$, and $[x,y]$ is the unique intersection of the local stable set of $x$ and the local unstable set of $y$, as depicted in Figure \ref{Bracket intersection}.
\begin{figure}[ht]
\begin{center}
\begin{tikzpicture}
\tikzstyle{axes}=[]
\begin{scope}[style=axes]
	\draw[<->] (-3,-1) node[left] {$X^s(x,\ep_X/2)$} -- (1,-1);
	\draw[<->] (-1,-3) -- (-1,1) node[above] {$X^u(x,\ep_X/2)$};
	\node at (-1.2,-1.4) {$x$};
	\node at (1.1,-1.4) {$[x,y]$};
	\pgfpathcircle{\pgfpoint{-1cm}{-1cm}} {2pt};
	\pgfpathcircle{\pgfpoint{0.5cm}{-1cm}} {2pt};
	\pgfusepath{fill}
\end{scope}
\begin{scope}[style=axes]
	\draw[<->] (-1.5,0.5) -- (2.5,0.5) node[right] {$X^s(y,\ep_X/2)$};
	\draw[<->] (0.5,-1.5) -- (0.5,2.5) node[above] {$X^u(y,\ep_X/2)$};
	\node at (0.7,0.2) {$y$};
	\node at (-1.6,0.2) {$[y,x]$};
	\pgfpathcircle{\pgfpoint{0.5cm}{0.5cm}} {2pt};
	\pgfpathcircle{\pgfpoint{-1cm}{0.5cm}} {2pt};
	\pgfusepath{fill}
\end{scope}
\end{tikzpicture}
\caption{The bracket map as the intersection of local stable and unstable sets}
\label{Bracket intersection}
\end{center}
\end{figure}
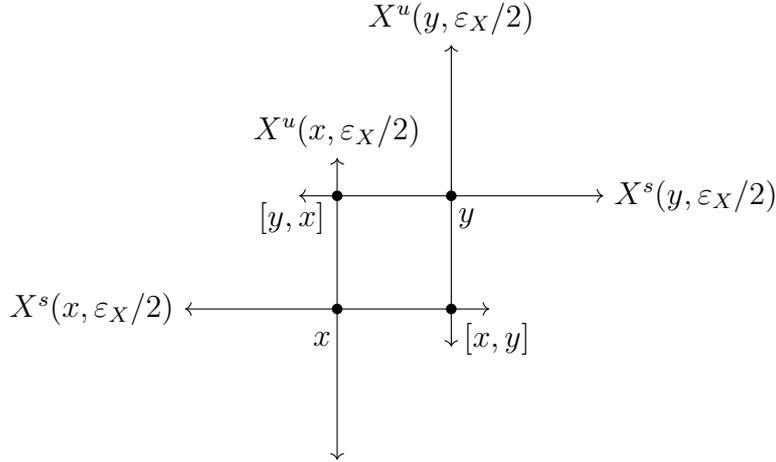

Given a point $x \in X$, global stable and unstable sets $X^s(x)$ and $X^u(x)$ are defined by
\begin{align*}
X^s(x) & = \{y \in X : \lim_{n \rightarrow +\infty} d(\varphi^n(x),\varphi^n(y)) = 0\} \text{ and } \\
X^u(x) & = \{y \in X : \lim_{n \rightarrow +\infty} d(\varphi^{-n}(x),\varphi^{-n}(y)) = 0\}.
\end{align*}
It is straightforward to prove that $X^s(x,\ep) \subset X^s(x)$ and $X^u(x,\ep) \subset X^u(x)$. Further, the open sets $\{X^s(y,\ep) : y \in X^s(x), 0 <\ep \leq \ep_X\}$ form a neighbourhood base for a locally compact and Hausdorff topology on $X^s(x)$. Similarly for $X^u(x)$. Also, the stable and unstable equivalence on $X$ will be denoted by $x \sim_s y$ and $x \sim_u y$. 

Smale spaces were defined by Ruelle to give a topological description of the non-wandering sets of differentiable dynamical systems satisfying Smale's \textit{Axiom A} \cite{Ruelle}. The zero dimensional Smale spaces are exactly the \emph{subshifts of finite type} (SFT) \cite{Putnam_Book}. More recently, Wieler characterised the Smale spaces with zero dimensional stable or unstable sets, as what are now called \textit{Wieler solenoids} \cite{Wieler}. These include Smale spaces associated with William's solenoids, aperiodic substitution tilings, and the limit space of self-similar groups. There are many higher dimensional examples as well, notably the hyperbolic toral automorphisms.

Smale spaces are expansive dynamical systems \cite[Proposition~2.1.9]{Putnam_Book} and hence have finite topological entropy \cite[Theorem~3.2]{Walters}, which we denote by $\ent(\varphi)$. Further, they always have finite Hausdorff dimension \cite[Chapter 7]{Ruelle}. Fractal dimensions of a large class of Smale spaces were calculated in \cite{Gero}.

The following theorem shows that recurrence properties of Smale spaces decompose the space in a rather nice manner.

\begin{thm}[Smale's Decomposition Theorem {\cite[Theorem 2.1.13]{Putnam_Book}}] \label{thm: Smale decomposition}
Assume that the Smale space $(X,\varphi)$ is non-wandering. Then $X$ can be decomposed into a finite disjoint union of clopen, $\varphi$-invariant, irreducible sets $X_0,\ldots , X_{N-1}$. Each of these sets can then be decomposed into a finite disjoint union of clopen sets $X_{i0},\ldots ,X_{iN_i}$ that are cyclically permuted by $\varphi$, and where $\varphi^{N_i+1}|_{X_{ij}}$ is mixing, for every $0\leq j\leq N_i$.
\end{thm}

To finish this subsection we define the notion of a self-similar Smale space along with Artigue's Theorem that all Smale spaces admit a self-similar metric. For additional details see \cite[Section 4.5]{Gero}. 

\begin{definition}\label{def:selfsimilarSmalespace}
A Smale space $(X,\varphi)$ with metric $d$ will be called \textit{self-similar} if both $\varphi,\,\varphi^{-1}$ are $\lambda$-Lipschitz with respect to the metric.
\end{definition}

As we can see from the next theorem, combining results of Fathi \cite{Fathi} and Fried \cite{Fried}, Artigue showed there is a plethora of self-similar Smale spaces.

\begin{thm}[\cite{Artigue}]\label{lem:Artiguelemma}
Every Smale space $(X,d,\varphi)$ admits a compatible (self-similar) metric $d'$ so that $(X,d',\varphi)$ is a self-similar Smale space.
\end{thm}

\subsection{Approximating Smale spaces}\label{sec:Markovpartitions}
Bowen's seminal Theorem \cite[Theorem 12]{Bowen2} asserts that every irreducible Smale space $(X,\varphi)$ admits Markov partitions of arbitrarily small diameter; that is, the space $X$ can be partitioned into closed with non-empty interior subsets called rectangles, that respect the dynamical structure and overlap only on their boundaries.

\begin{definition}
A non-empty subset $R\subset X$ is called a \textit{rectangle} if $\text{diam}(R)\leq \varepsilon_X'$ and $[x,y]\in R$, for any $x,y \in R$.
\end{definition}

Bowen's Theorem leads to the construction of a SFT for which $(X,\varphi)$ is a factor. The factor map is finite-to-one and one-to-one on a dense $G_\delta$-subset of the shift space. Although, Markov partitions provide a combinatorial framework for understanding the topological structure of Smale spaces, a more delicate approach is required for studying their metric structure. In \cite[Theorem 6.2]{Gero}, the first author extended Markov partitions to special open covers that can be used to approximate the metric structure of Smale spaces. Before being more specific, we introduce notation used in \cite[Theorem 6.2]{Gero} that we require here. 

Let $\#S$ denote the cardinality of a finite set $S$. Suppose $(Z,r)$ is a compact metric space and $\mathcal{U}$ is a finite cover of $Z$ such that $\varnothing, Z\not \in \mathcal{U}$. Define
\begin{align*}
\text{diam}(\mathcal{U})&=\max\limits_{U\in \, \mathcal{U}}\text{diam}(U);\\
\mul(\mathcal{U})&= \max \{n: U_{i_1}\cap \ldots \cap U_{i_n}\neq \varnothing \text{ such that } U_{i_j}\neq U_{i_k} \text{ for } j \neq k\};\\
\text{Leb}(\mathcal{U})&= \min\limits_{z\in Z} \max\limits_{U\in \mathcal{U}}r(z,Z\setminus U).
\end{align*}
The last quantity is the Lebesgue covering number of $\mathcal{U}$ and it holds that for every $z\in Z$ there is some $U\in \mathcal{U}$ so that the ball $B(z,\ell)\subset U$, where $\ell=\text{Leb}(\mathcal{U})$.

In order to study dynamical properties on a compact metric space by means of finite approximations we use refining sequences of open covers, first introduced in \cite[Corollary p. 314]{AKM} to study topological entropy.

\begin{definition}\label{def:refining sequence}
Let $(Z,r)$ be a compact metric space. A sequence $(\mathcal{V}_n)_{n\geq 0}$ of finite open covers of $Z$ is called \textit{refining} if, $\mathcal{V}_0=\{Z\}$ and for every $n\geq 0$ any element of $\mathcal{V}_{n+1}$ lies inside some element of $\mathcal{V}_n$ such that
\[
\lim\limits_{n\to \infty}\text{diam}(\mathcal{V}_n)=0.
\]
\end{definition}

A refining sequence of open covers naturally gives rise to an approximation graph \cite[Section 2.4]{Gero}; that is, a rooted graph, with vertices given by the sets in the covers and edges defined by inclusion of the sets in preceding refinements. Its infinite path space approximates the underlying space. In the sequel, it will often be useful to visualise refining sequences as rooted graphs.

Resuming our discussion, \cite[Theorem 6.2]{Gero} is the key tool for understanding the Ahlfors regularity of Smale spaces. The starting point is a Markov partition $\mathcal{R}_1$ with sufficiently small diameter. Inductively, $\mathcal{R}_1$ is then refined to Markov partitions $\mathcal{R}_n$, where each $\mathcal{R}_{n}$ refines $\mathcal{R}_{n-1}$ and $\diam(\mathcal{R}_n)$ goes to zero. The sequence $(\mathcal{R}_n)_{n\in \mathbb N}$ is the main ingredient for deriving Bowen's factor map. However, for our purposes it has one drawback. Specifically, the Lebesgue numbers of $\mathcal{R}_n$ are zero, unless the Smale space is zero-dimensional. To remedy this, we $\delta$-enlarge $\mathcal{R}_1$ to a cover of open rectangles $\mathcal{R}_1^{\delta}$. We note that $\delta>0$ is carefully chosen and is independent of $n\in \mathbb N$. Using induction we then define a refining sequence of covers of open rectangles $\mathcal{R}_n^{\delta}$ that are $\delta$-enlarged versions of $\mathcal{R}_n.$ While we do not go into the specifics of the construction here, \cite[Theorem 6.2]{Gero} is also key to defining a Fredholm module representative of the fundamental class for Smale spaces. We encourage the reader to delve into specifics of the first author's construction in \cite{Gero}. Here we require the following result.

\begin{thm}[{\cite[Theorem 6.2]{Gero}}]\label{thm:theoremgraphSmalespaces}
Suppose $(X,\varphi)$ is an irreducible Smale space and $\mathcal{R}_1$ is a Markov partition with sufficiently small diameter (see \cite[Section 6]{Gero}). Then, there are positive constants $C,c$ and $\theta \leq \varepsilon_X$, as well as a sufficiently small $\delta>0$, so that for every $n\in \mathbb N$,
\begin{enumerate}[(1)]
\item the cover $\mathcal{R}_n^{\delta}$ consists of open rectangles and $\mathcal{R}_{n+1}^{\delta}$ refines $\mathcal{R}_n^{\delta}$;
\item if $R^{\delta}\in \mathcal{R}_n^{\delta}$ and $x,y\in R^{\delta}$, then $d(\varphi^r(x),\varphi^r(y))< \varepsilon_X'$ when $|r|\leq n-1$;
\item  if $|r|\leq n-1$, then the open cover $\varphi^r(\mathcal{R}_n^{\delta})$ refines $\mathcal{R}_{n-|r|}^{\delta}$;
\item $\diam(\mathcal{R}_n^{\delta})\leq \lambda^{-n+1}\theta;$
\item $\mul(\mathcal{R}_n^{\delta})\leq (\#\mathcal{R}_1^{\delta})^2.$
\end{enumerate}
Moreover, for every $\varepsilon \in (0,1)$, there is some $n_0\in \mathbb N$ such that, for $n\geq n_0$, we have
\begin{enumerate}[(6)]
\item $ce^{2(\ent(\varphi)-\varepsilon) n}< \#\mathcal{R}_n^{\delta} <Ce^{2(\ent(\varphi)+\varepsilon) n}.$
\end{enumerate} 
If, in addition, $\varphi, \varphi^{-1}$ are Lipschitz and $\Lambda\coloneqq \max \{\Lip (\varphi),\Lip (\varphi^{-1})\}$, then there is some $\eta\in (0,\theta]$ so that, for every $n\in \mathbb N$, we have
\begin{enumerate}[(7)]
\item $\Leb(\mathcal{R}_n^{\delta})\geq \Lambda^{-n+1}\eta$.
\end{enumerate}
The sequence $(\mathcal{R}_n^{\delta})_{n\geq 0}$ is called a refining sequence of $\delta$-enlarged Markov partitions. 
\end{thm}

The following corollary shows that, up to conjugacy, condition (7) of Theorem \ref{thm:theoremgraphSmalespaces} always holds in the best possible way.

\begin{cor}[{\cite[Corollary 6.3]{Gero}}]\label{cor:topconjgraph}
An irreducible Smale space $(X,\varphi)$ is topologically conjugate to a self-similar Smale space $(Y,\psi)$ which admits a refining sequence of $\delta$-enlarged Markov partitions, with contraction constant $\lambda_{Y}>1$ equal to $\Lambda_{Y}.$ 
\end{cor}

\subsection{Smale space C*-algebras}\label{sec:Smale_C_alg}

In this section, we describe various $C^\ast$-algebras arising from an irreducible Smale space. These $C^*$-algebras arise from the stable and unstable equivalence relations on Smale spaces, closely following the approach of Putnam and Spielberg \cite{PS}. While our treatment here is brief, it is self-contained. For further details see any of \cite{KPW,  Ruelle_algebras, Putnam_algebras, PS}, with the first being specific to the situation at hand.

Suppose $(X,\varphi)$ is an irreducible Smale space and choose periodic orbits $P$ and $Q$. Define
\[ X^{s}(P) = \bigcup_{p \in P} X^{s}(p), \quad X^{u}(Q) = \bigcup_{q \in Q} X^{u}(q).  \]
Give $X^s(P)$ a locally compact and Hausdorff topology by forming a neighbourhood base of open sets  $X^s(x,\ep)$, as $x$ varies over $X^s(P)$ and $0 < \ep \leq \ep_X$. For $X^u(Q)$ we use the analogous approach. We obtain the stable and unstable groupoids 
\begin{align*}
\gs & = \{(v,w) : v \sim_s w \textrm{ and } v,w \in X^u(Q) \} \text{ and } \\
\gu & = \{(v,w) : v \sim_u w \textrm{ and } v,w \in X^s(P) \}.
\end{align*}
To topologise the stable groupoid, let $(v,w)$ be in $\gs$. Then for some $N \geq 0$, $\varphi^N(w)$ is in $X^s(\varphi^N(v),\ep_X^\prime/2)$. Continuity of $\varphi$ implies the existence of $\eta > 0$ such that
$\varphi^N(X^u(w,\eta)) \subset X^u(\varphi^N(w),\ep_X^\prime/2)$. Thus we can define the stable holonomy map $h^s: X^u(w,\eta) \rightarrow X^u(v, \ep_X/2)$ by
\[ 
h^s(x) = \varphi^{-N}[\varphi^N(x),\varphi^N(v)],
\]
which is a homeomorphism onto its image. Also, let 
\begin{equation}\label{V_basis}
V^s(v,w,h^s,\eta,N) = \{ (h^s(x),x) : x \in X^u(w,\eta)\}.
\end{equation}
A picture of the map $h^s$ appears in Figure \ref{fig:local_homeo}.

\begin{lemma}\label{stable nbhd}
The sets $V^s(v,w,h^s,\eta,N)$ from \eqref{V_basis} form a neighbourhood base for a second countable, locally compact and Hausdorff topology on $\gs$, for which $\gs$ is an \'etale groupoid. The analogous result holds for $\gu$ by considering the sets $V^u(v,w,h^u,\eta,N)$.
\end{lemma}

Note that the stable and unstable groupoids depend on $P$ and $Q$ up to the groupoid equivalence given in \cite{MRW}. These groupoids are proved to be amenable in \cite{PS}. Further, $X^h(P,Q)=X^s(P)\cap X^u(Q)$ is countable and dense in $X$.

Now consider the continuous functions of compact support $C_c(\gs)$ as a complex vector space. The convolution and involution on $C_c(\gs)$, for $f,g \in C_c(\gs)$ and $(v,w) \in \gs$, are given by
\newpage
\begin{align*}
f \cdot g(v,w) & = \sum_{z \sim_s v} f(v,z)g(z,w), \\
f^\ast(v,w) & = \overline{f(w,v)},
\end{align*}
making $C_c(\gs)$ into a complex $\ast$-algebra. A partition of unity argument shows that any function in $C_c(\gs)$ can be written as a sum of functions with support in an element of the neighbourhood base described in Lemma \ref{stable nbhd}.

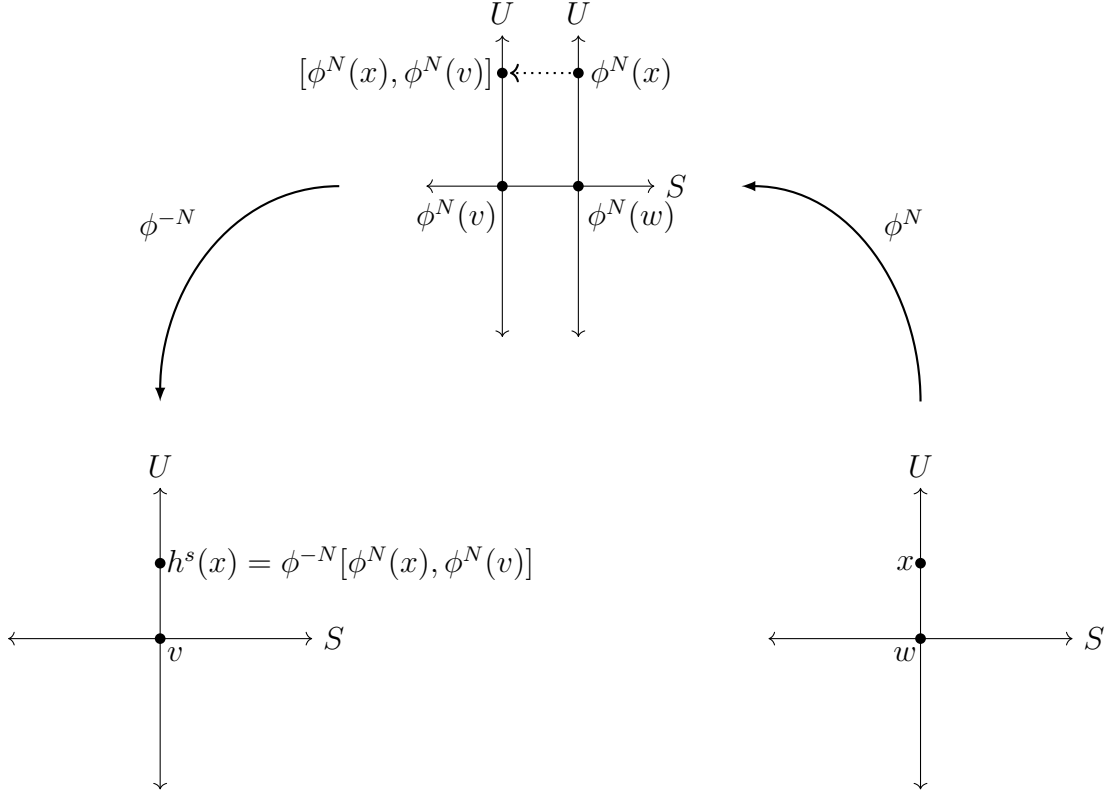
\begin{figure}[htb]
\begin{center}
\begin{tikzpicture}
\tikzstyle{axes}=[]
\begin{scope}[style=axes]
	\draw[<->] (3,0) -- (7,0) node[right] {$S$};
	\draw[<->] (5,-2) -- (5,2) node[above] {$U$};
	\node at (4.8,-0.2) {$w$};
	\node at (4.8,1) {$x$};
	\pgfpathcircle{\pgfpoint{5cm}{0cm}} {2pt};
	\pgfpathcircle{\pgfpoint{5cm}{1cm}} {2pt};
	\pgfusepath{fill}
\end{scope}
\begin{scope}[style=axes]
	\draw[<->] (-7,0) -- (-3,0) node[right] {$S$};
	\draw[<->] (-5,-2) -- (-5,2) node[above] {$U$};
	\node at (-4.8,-0.2) {$v$};
	\node at (-2.5,1) {$h^s(x)=\phi^{-N}[\phi^N(x),\phi^N(v)]$};
	\pgfpathcircle{\pgfpoint{-5cm}{0cm}} {2pt};
	\pgfpathcircle{\pgfpoint{-5cm}{1cm}} {2pt};
	\pgfusepath{fill}
\end{scope}
\begin{scope}[style=axes]
	\draw[<->] (-1.5,6) -- (1.5,6) node[right] {$S$};
	\draw[<->] (-.5,4) -- (-.5,8) node[above] {$U$};
	\draw[<->] (.5,4) -- (.5,8) node[above] {$U$};
	\node at (1.2,5.6) {$\phi^N(w)$};
	\node at (1.2,7.5) {$\phi^N(x)$};
	\node at (-1.1,5.6) {$\phi^N(v)$};
	\node at (-1.9,7.5) {$[\phi^N(x),\phi^N(v)]$};
	\pgfpathcircle{\pgfpoint{-.5cm}{6cm}} {2pt};
	\pgfpathcircle{\pgfpoint{.5cm}{6cm}} {2pt};
	\pgfpathcircle{\pgfpoint{-.5cm}{7.5cm}} {2pt};
	\pgfpathcircle{\pgfpoint{.5cm}{7.5cm}} {2pt};
	\pgfusepath{fill}
	\draw[->,thick,dotted] (0.4,7.5) -- (-0.4,7.5);
\end{scope}
\node (anchor1) at ( 2.5,6) {};
\node (anchor2) at ( 5,3) {}
	edge [->,>=latex,out=90,in=0,thick] node[auto,swap]{$\phi^N$} (anchor1);
\node (anchor3) at ( -2.5,6) {};
\node (anchor4) at ( -5,3) {}
	edge [<-,>=latex,out=90,in=180,thick] node[auto]{$\phi^{-N}$} (anchor3);
\end{tikzpicture}
\caption{The stable holonomy map onto its image $h^s: X^u(w,\eta) \rightarrow X^u(v,\varepsilon_X/2)$}
\label{fig:local_homeo}
\end{center}
\end{figure}

To complete $C_c(G^s(Q))$ into a $C^*$-algebra, define a faithful representation $\rho_s$ of $C_c(G^s(Q))$ on $\mathscr{H}=\ell^2(X^h(P,Q))$, for $f\in C_c(G^s(Q)), \, \xi \in \mathscr{H}$, by 
\begin{equation}\label{eq:regular representation_fund}
\rho_s(f)\xi(v)= \sum_{z\sim_s v}f(v,z) \xi (z).
\end{equation}
The \textit{stable algebra} $\mathcal{S}(Q)$ is defined as the completion of $\rho_s(C_c(G^s(Q)))$ in $\mathcal{B}(\mathscr{H})$ and the \textit{unstable algebra} $\mathcal{U}(P)$ is constructed similarly. See \cite{KPW,PS} for further details.

It is often convenient to choose $P=Q$. However, it will be important later on to choose $P\neq Q$ so that $X^h(P,Q)$ contains no periodic points. We also often suppress the notation of $\rho_s$ and write $a\xi$ instead of $\rho_s(a)\xi$, for $a\in C_c(G^s(Q)), \, \xi \in \mathscr{H}$.

\begin{lemma}[{\cite[Lemma 3.3]{KPW}}]
Suppose that $V^s(v,w,h^s,\eta,N)$ is an open set as in Lemma \ref{stable nbhd}, and $a\in C_c(G^s(Q))$ with $\text{supp}(a) \subset V^s(v,w,h^s,\eta,N)$. Then for every $x\in X^h(P,Q)$ we have $$a\delta_x= a(h^s(x),x)\delta_{h^s(x)},$$ if $x\in X^u(w,\eta)$, and $a\delta_x=0$, otherwise.
\end{lemma}

The primary $C^*$-algebras in this paper are the stable and unstable Ruelle algebras, which are crossed products of $\mathcal{S}(Q)$ and $\mathcal{U}(P)$ by the integers. Indeed, define an automorphism $\alpha_s$ of $C_c(G^s(Q))$ by 
\begin{equation}
\alpha_s(f)=f\circ (\varphi^{-1} \times \varphi^{-1}).
\end{equation}
By continuity, the automorphism $\alpha_s$ extends to $\mathcal{S}(Q)$. For $a\in C_c(G^s(Q))$ with $\supp(a)\subset V^s(v,w,h^s,\eta,N)$ and $x\in X^h(P,Q)$  such that $h^s(\varphi^{-1}(x))$ is defined, we have 
\begin{equation}\label{eq:groupoid_autom}
\alpha_s(a)\delta_x=a(h^s\circ \varphi^{-1}(x), \varphi^{-1}(x))\delta_{\varphi \circ h^s \circ \varphi^{-1}(x)}.
\end{equation}

Moreover, the homeomorphism $\varphi$ induces a unitary $u$ on $\mathscr{H}$ given by $u\delta_x=\delta_{\varphi (x)}$, such that $\alpha_s(a)=uau^*$ for every $a\in \mathcal{S}(Q)$. Thus we obtain a crossed product 
$\mathcal{S}(Q)\rtimes_{\alpha_s} \mathbb Z$ called the \emph{stable Ruelle algebra}. The unstable Ruelle algebra for $G^u(P)$ is constructed analogously. 

\begin{notation}\label{def:SmaleCalg}
The stable and unstable algebras as defined above are denoted by $\mathcal{S}(Q)$ and $\mathcal{U}(P)$, respectively. Also, the stable and unstable Ruelle algebras are denoted by
$\mathcal{R}^s(Q)=\mathcal{S}(Q)\rtimes_{\alpha_s} \mathbb Z$ and $\mathcal{R}^u(P)=\mathcal{U}(P)\rtimes_{\alpha_u} \mathbb Z$, respectively.
\end{notation}

For any choice of $P$ and $Q$, the stable and unstable (Ruelle) algebras are strongly Morita equivalent to the ones in \cite{Putnam_algebras}. Further, since the algebras are $C^*$-stable (see \cite[Theorem A.2]{DY} and \cite[Corollary 4.5]{HRor}), for periodic orbits $Q$ and $Q'$ we have that $\mathcal{S}(Q)\cong \mathcal{S}(Q')$ and $\mathcal{R}^s(Q)\cong \mathcal{R}^s(Q')$. The same holds for the unstable (Ruelle) algebras.

It is worth mentioning that these $C^*$-algebras fit into Elliott's classification program of simple, separable, nuclear $C^*$-algebras, and satisfy the UCT \cite{PS}. In addition, the stable and unstable algebras have finite nuclear dimension \cite[Corollary 3.8]{DS2} and are quasidiagonal \cite{DGY}. If the Smale space is mixing, the stable and unstable algebras are also simple and have unique traces \cite[Theorem 3.3]{Putnam_algebras}. We note that these traces are unbounded. In general, Ruelle algebras are also simple and purely infinite \cite{PS}. As a result, they can be classified up to isomorphism by $\Kt$-theory \cite{KirP}.

The stable and unstable algebras of a subshift of finite type are approximately finite dimensional and the associated Ruelle algebras are strongly Morita equivalent to Cuntz-Krieger algebras. Further, the stable and unstable algebras of the dyadic solenoid are strongly Morita equivalent to the Bunce-Deddens algebra of type $2^{\infty}$. In the case of hyperbolic toral automorphisms, these $C^*$-algebras are equivalent to irrational rotation algebras. For further details see \cite{Putnam_algebras,PS}.

\subsection{K-theoretic preliminaries}\label{sec:K-theoretic_prelim}

We assume that the reader is familiar with Hilbert $C^*$-modules as defined in \cite{Lance}. To a pair $(A,B)$ of separable $\mathbb Z_2$-graded $C^*$-algebras, Kasparov \cite{Kasparov5,Kasparov2} associated an abelian group $\KKt_0(A,B)$ formed by equivalence classes of elements that are described in the following definition.

\begin{definition}[{\cite[Definition 17.1.1]{Blackadar} \cite[Definition 2.1]{Kasparov2}}]\label{def:Kasparovbimodules}
Let $A,B$ be separable $\mathbb Z_2$-graded $C^*$-algebras. A \textit{Kasparov} $(A,B)$-\textit{bimodule} is a triple $(E,\rho,F)$ such that
\begin{enumerate}[(1)]
\item $E$ is a countably generated $\mathbb Z_2$-graded Hilbert $B$-module;
\item $\rho:A\to \mathcal B_B (E)$ is a graded $*$-homomorphism to the adjointable operators;
\item the operator $F\in \mathcal B_B(E)$ anti-commutes with the grading of $E$ and satisfies $$\rho(a)(F^*-F)\in \mathcal{K}_B(E),\quad \rho(a)(F^2-1)\in \mathcal{K}_B(E),\quad [\rho(a),F]\in \mathcal{K}_B(E),$$ for all $a\in A$, where $\mathcal{K}_B(E)$ is the two-sided ideal of $B$-compact operators.
\end{enumerate}
\end{definition}

The triples in the definition above are typically called Kasparov modules and define elements of $\KKt_0(A,B)$, under unitary and operator homotopy equivalence \cite{Kasparov2}. Let $\mathbb C_1$ be the Clifford algebra with one generator. The group $\KKt_0(A\otimes \mathbb C_1,B)$ is denoted by $\KKt_1(A,B)$. 

In the sequel all $C^*$-algebras, except for $\mathbb C_1$, will have trivial grading. The two extremes of $\KKt$-theory are $\Kt$-theory and $\Kt$-homology: $\KKt_i(\mathbb C, B)\cong \Kt_i(B)$ and $\KKt_i(A, \mathbb C)=\Kt^i(A)$, respectively. The $\Kt$-homology groups are built from equivalence classes of \textit{Fredholm modules} over $A$, see \cite{HR}. Of specific interest to us are the odd Fredholm modules over $A$. These are the Kasparov $(A,\mathbb C)$-bimodules without any mention on gradings.

\begin{notation}
By \cite[IV-A-Proposition 13]{Connes}, the odd Fredholm modules correspond to Kasparov $(A\otimes \mathbb C_1, \mathbb C)$-bimodules. Under this identification, from this point on we will denote the classes in $\KKt_1(A,\mathbb C)$ as $\Kt$-homology classes of odd Fredholm modules over $A$.
\end{notation}

Notably, the group $\KKt_1(A,\mathbb C)$ is naturally isomorphic to the group of equivalence classes of invertible extensions of $A$ by the compact operators $\mathcal{K}(H)$, where $H$ is a separable Hilbert space. The isomorphism takes the class $[H,\rho,F]$ in $\KKt_1(A,\mathbb C)$ (under the normalisation $F=F^*, F^2=1$) to the class of the Toeplitz extension $\tau: A\to \mathcal{Q}(H)$ given by $\tau(a)=P\rho(a)P+\mathcal{K}(H)$, where $P$ is the projection $(F+1)/2$. The inverse map is defined through Stinespring dilation. Also, we note that if $A$ is nuclear then all extensions are invertible due to the Choi-Effros Lifting Theorem. For these details we refer to \cite{HR}. In general, this approach provides an abstract description of the inverse map and it is a challenging task to find a concrete one. Computing explicitly the inverse map for a specific case is the main goal of this paper. We refer to Theorem \ref{thm:Liftingtheorem} of the current subsection for more details.

The external tensor product of a $C^*$-algebra $D$ on the right of a class is defined to be 
\begin{equation}\label{eq:tensortriple}
\tau_D[\mathcal{E}]=[E\otimes D, \rho \otimes \id , F\otimes \id]\in \KKt_0(A\otimes D, B\otimes D),
\end{equation}
such that the graded tensor products of $C^*$-algebras are equipped with the spatial norm. Similarly, the left tensor product by $D$ gives the class $\tau^D[\mathcal{E}]$. 

We note that $\KKt$-theory forms a bifunctor of pairs of $C^*$-algebras that is $C^*$-stable, homotopy-invariant in both variables and satisfies \textit{formal Bott periodicity}; that is, $\KKt_1(A\otimes \mathbb C_1,B)\cong \KKt_0(A,B)$ and $\KKt_0(A,B\otimes \mathbb C_1)\cong \KKt_1(A,B)$.

The main tool in $\KKt$-theory is the \textit{Kasparov product}
\begin{equation}\label{eq:Kasparovprod1}
\otimes_D:\KKt_0(A,D)\times \KKt_0(D,B)\to \KKt_0(A,B)
\end{equation}
that generalises the cup-cap product from topological $\Kt$-theory. It is associative and functorial in every possible way and turns $\KKt$-theory into a category \cite{Blackadar}. Moreover, the Kasparov product provides $\KKt_0(A,A)$ with a ring structure with identity $1_A$. Using tensor products, the Kasparov product is generalised to the following operation:
\begin{equation}\label{eq:Kasparovprod2}
\KKt_0(A_1,B_1\otimes D)\times \KKt_0(D\otimes A_2,B_2)\to \KKt_0(A_1\otimes A_2,B_1\otimes B_2).
\end{equation}
Abusing notation, for $x\in \KKt_0(A_1,B_1\otimes D)$ and $y\in \KKt_0(D\otimes A_2,B_2)$, operation (\ref{eq:Kasparovprod2}) is denoted as
\[
x\otimes_D y \coloneqq \tau_{A_2}(x)\otimes_{B_1\otimes D\otimes A_2}\tau^{B_1}(y).
\]
The Kasparov product also naturally encodes \textit{Bott periodicity}, see \cite[Section 19.2.5]{Blackadar}.

We conclude this subsection with a discussion of smooth extensions and summable Fredholm modules.  Recall that $H$ is a separable Hilbert space. Given a compact operator $T$, let $(s_n(T))_{n\in \mathbb N}$ be the sequence of its singular values in decreasing order, counting multiplicities.

For $p>0$, the \textit{Schatten} $p$-\textit{ideal} on $H$ is the set
\begin{equation}\label{eq:Schattenideal}
\mathcal{L}^p(H)=\{T\in \mathcal{K}(H): (s_n(T))_{n\in \mathbb N}\in \ell^p(\mathbb N)\},
\end{equation}
equipped with the \textit{(quasi-)norm} 
\begin{equation}\label{eq:Schattennorm}
\|T\|_p=\|(s_n(T))_{n\in \mathbb N}\|_{\ell^p(\mathbb N)}.
\end{equation}
For $p\geq 1$, the \textit{Logarithmic integral} $p$\textit{-ideal} on $H$ is defined as 
\begin{equation}\label{eq:Liideal}
\Li^{1/p}(H)=\{T\in \mathcal{K}(H): s_n(T)=O((\log n)^{-1/p})\}
\end{equation}
with norm described on \cite[p. 391]{Connes}. For simplicity, denote $\Li^1(H)$ by $\Li(H)$. For $p\geq 1$, the ideals $\mathcal{L}^p(H)$ and $\Li^{1/p}(H)$ are examples of symmetrically normed ideals in $\mathcal{B}(H)$. For $p\in (0,1),$ the ideals $\mathcal{L}^p(H)$ are only quasi-normed but still enjoy the nice properties of symmetrically normed ideals. For details we refer to \cite{Gero_Thesis}.

The following notion was introduced by Douglas \cite{Douglas} to study analytic properties of extensions and later by Douglas and Voiculescu \cite{DV} in index theory. Let $\mathcal{I}$ denote one of the aforementioned ideals or $\mathcal{K}(H)$.

\begin{definition}[\cite{Douglas}]\label{def:summableextension} Let $\mathcal{A}$ be a dense $*$-subalgebra of the $C^*$-algebra $A$. An extension $\tau:A\to \mathcal{Q}(H)$ will be called $\mathcal{I}$\textit{-smooth on} $\mathcal{A}$ if there is a linear map $\eta: \mathcal{A}\to \mathcal{B}(H)$ such that $$\eta(ab)-\eta(a)\eta(b)\in \mathcal{I},\quad  \eta(a^*)-\eta(a)^*\in \mathcal{I}$$ and $\tau(a)=\eta(a)+\mathcal{K}(H)$, for all $a,b \in \mathcal{A}$. If $\mathcal{I}=\mathcal{L}^p(H)$, the extension will be called $p$\textit{-smooth} (or just \textit{finitely smooth}).
\end{definition}

Summability of Fredholm modules was introduced by Connes \cite{Connes} and is regarded as a refinement of the notion of smoothness for extensions. Summability plays a central role in Connes' approach to index theory. 

\begin{definition}[{\cite{Connes, GM}}]
Let $\mathcal{A}$ be a dense $*$-subalgebra of the $C^*$-algebra $A$. Also, consider the ideal $\mathcal{I}^{1/2}=\text{span}\{T\in \mathcal{K}(H):T^*T,\, TT^*\in \mathcal{I}\}$. A Fredholm module $(H,\rho,F)$ over $A$ is $\mathcal{I}^{1/2}$\textit{-summable on} $\mathcal{A}$ if, for every $a\in \mathcal{A}$, we have 
\[
\rho(a)(F^*-F)\in \mathcal{I},\quad \rho(a)(F^2-1)\in \mathcal{I},\quad [\rho(a),F]\in \mathcal{I}^{1/2}.
\]
If $\mathcal{I}^{1/2}=\mathcal{L}^p(H)$, the Fredholm module is called $p$\textit{-summable} and if $\mathcal{I}^{1/2}=\Li^{1/2}(H)$ the Fredholm module is called $\theta$-summable.
\end{definition}

In the extension picture of $\KKt_1(A,\mathbb C)$, given a class of an odd $\mathcal{I}^{1/2}$-summable Fredholm module we obtain a class of an $\mathcal{I}$-smooth extension. On the other hand, the inverse procedure of finding summable Fredholm module representatives for classes of smooth invertible extensions is a subtle process, see \cite{Goffeng}. However, a general recipe to this end is the following. 

\begin{thm}[{\cite[Theorem 2.2.1]{GM}}]\label{thm:Liftingtheorem}
Let $\tau:A\to \mathcal{Q}(H)$ be an extension that is $\mathcal{I}$-smooth on a dense $*$-subalgebra $\mathcal{A}\subset A$ and $\eta:\mathcal{A}\to \mathcal{B}(H)$ be a linear map as in Definition \ref{def:summableextension}. If there is an isometry $V:H\to H'$ and a representation $\rho:A\to \mathcal{B}(H')$ such that $\eta(a)-V^*\rho(a)V\in \mathcal{I},$ for all $a\in \mathcal{A}$, then the triple $(H', \rho , 2VV^*-1)$ is an odd Fredholm module over $A$. Moreover, it is $\mathcal{I}^{1/2}$-summable on $\mathcal{A}$ and represents the equivalence class of $\tau$.
\end{thm}

\subsection{KK-duality}\label{sec:SW-duality} Spanier-Whitehead duality \cite{BG} relates the homology of a finite complex with the cohomology of a dual complex. The following noncommutative analogue is based on definitions given in \cite{EmersonDuality,KP,KPW}.

\begin{definition}\label{def:SW-Kduality}
Let $A$ and $B$ be separable $C^*$-algebras. We say that $A$ and $B$ are \textit{Spanier-Whitehead} $\KKt$\textit{-dual} (or just $\KKt$\textit{-dual}) if there is a $\Kt$-homology class $\mu \in \KKt_i(A\otimes B,\mathbb C)$ and a $\Kt$-theory class $\nu \in \KKt_i(\mathbb C, A\otimes B)$ such that 
\begin{align*}
\nu \otimes_B \mu=1_A \quad \text{and} \quad \nu  \otimes_A \mu=(-1)^i1_B,
\end{align*}
where the flip isomorphism $A\otimes B\to B\otimes A$ is employed where required to ensure the Kasparov product makes sense. In analogy with the commutative setting, we call the $\Kt$-homology class in a $\KKt$-duality pair a \emph{fundamental class}.
If $A$ is $\KKt$\textit{-dual} to its opposite algebra $A^{\textit{op}}$, we say $A$ satisfies \textit{Poincar{\'e} duality}.
\end{definition}

Given a $\KKt$-duality pair $(\mu,\nu)$ between $A$ and $B$, there are isomorphisms
\begin{equation}\label{eq:slantprod1}
\nu \otimes_B -:\KKt_j(B, \mathbb C)\to \KKt_{j+i}(\mathbb C, A),
\end{equation}
\begin{equation}\label{eq:slantprod2}
- \otimes_A \mu: \KKt_j(\mathbb C, A)\to \KKt_{j+i}(B,\mathbb C),
\end{equation}
see \cite[Lemma 9]{EmersonDuality} for the most general framework. It is also worth mentioning that $\nu$ is unique for $\mu$ and vice versa. The question whether a separable $C^*$-algebra has a Spanier-Whitehead $\KKt$-dual is addressed in \cite{KS}, which gives an excellent overview of $\KKt$-duality.

Connes \cite[Chapter VI]{Connes} proved that the irrational rotation algebras satisfy Poincar{\'e} duality. The first odd Poincar\'e duality was proved by Kaminker and Putnam \cite{KP} for Cuntz-Krieger algebras. Further, Popescu and Zacharias \cite{PZ} showed that higher rank graph algebras satisfy Poincar\'e duality. Emerson \cite{EmersonDuality} showed that $C(\partial \Gamma)\rtimes \Gamma$ satisfies Poincar{\'e} duality for a large class of hyperbolic groups $\Gamma$. Also, Echterhoff, Emerson, and Kim \cite{EEK2} showed $\KKt$-duality for certain classes of orbifold $C^*$-algebras. Nishikawa and Proietti studied Spanier-Whitehead duality for discrete groups and compared it with the Baum-Connes conjecture \cite{NP}. As we will see in the next section, Kaminker, Putnam, and the second author \cite{KPW} proved $\KKt$-duality between the stable and unstable Ruelle algebras.

\subsection{KK-duality for Ruelle algebras}\label{sec:K-duality_Ruelle}

We now explain the $\KKt$-duality of Ruelle algebras by focussing on the fundamental class, represented by an extension in \cite{KPW}. Then, due to nuclearity of Ruelle algebras, the fundamental class has an abstract Fredholm module representative. The thrust of this paper is to define a concrete Fredholm module version of the fundamental class. We refer to \cite[Section 5]{KPW} for the $\Kt$-theory duality class since we do not require it here.

As a standing assumption for this section, let $(X,\varphi)$ be an irreducible Smale space with periodic orbits $P,Q$ such that $P\cap Q=\varnothing$. Let $\mathscr{H}=\ell^2(X^h(P,Q))$ and note that $X^h(P,Q)$ has no periodic points.

\begin{thm}[{\cite[Theorem 1.1 and Corollary 4.3]{KPW}}]\label{thm:Ruelleduality}
The Ruelle algebras $\mathcal{R}^s(Q)$ and $\mathcal{R}^u(P)$ are Spanier-Whitehead $\KKt$-dual. Moreover, if the $\Kt$-theory group of $\mathcal{S}(Q)$ or $\mathcal{U}(P)$ have finite rank, then $\mathcal{R}^s(Q)\cong \mathcal{R}^u(P)$ and both Ruelle algebras satisfy Poincar{\'e} duality.
\end{thm}

\begin{remark}
We note that Proietti and Yamashita \cite{PY} recently announced a proof that the $\Kt$-theory of both $\mathcal{S}(Q)$ or $\mathcal{U}(P)$ have finite rank, thereby removing that hypothesis from the above theorem.
\end{remark}

Philosophically, the $\KKt$-duality in \cite{KPW} builds on the previous $\KKt$-duality result of Kaminker and Putnam \cite{KP} for Cuntz-Krieger algebras. However, the two approaches are completely different. Cuntz-Krieger algebras are quotients of Toeplitz algebras which admit representations on Fock spaces, and this is used heavily in the construction of the fundamental class. On the other hand, general Ruelle algebras do not have such an a priori tractable description. Nevertheless, it seems possible that they admit interesting Cuntz-Pimsner models, but at the moment this is only known in a few specific cases \cite{DGMW}.

Using the open subsets of $G^s(Q), G^u(P)$ of Lemma \ref{stable nbhd}, we begin with the following straightforward result. 

\begin{lemma}[{\cite[Lemma 6.1]{KPW}}]\label{lem:stableunstablerankone}
Suppose $a\in C_c(G^s(Q))$ and $b\in C_c(G^u(P))$ such that $\supp(a)\subset V^s(v,w,h^s,\eta,N)$ and $\supp(b)\subset V^u(v',w',h^u,\eta',N')$, then $\rank(ab)$ and $\rank(ba)$ are at most one. Consequently, for every $a \in \mathcal{S}(Q)$ and $b \in \mathcal{U}(P)$ we have that the products $ab$ and $ba$ are in $\mathcal{K}(\mathscr{H})$.
\end{lemma}

The hyperbolic nature of the dynamics, along with the fact that $P\cap Q=\varnothing$ leads to the following lemma. Recall the definition of the automorphism $\alpha_s$ from \eqref{eq:groupoid_autom}.

\begin{lemma}[{\cite[Lemma 6.2]{KPW}}]\label{lem:minusinftylimit}
Suppose $a\in C_c(G^s(Q))$ and $b\in C_c(G^u(P))$ such that $\supp(a)\subset V^s(v,w,h^s,\eta,N)$ and $\supp(b)\subset V^u(v',w',h^u,\eta',N')$, then there exists $M\in \mathbb N$ such that $\alpha_s^{-n}(a)b=b\alpha_s^{-n}(a)=0$, for $n\geq M$. Consequently, for every $a \in \mathcal{S}(Q)$ and $b \in \mathcal{U}(P)$ we have that
\[
\lim_{n\to +\infty} \alpha_s^{-n}(a)b=0 \,\,\,\, \text{and} \,\,\,\, \lim_{n\to +\infty} b\alpha_s^{-n}(a)=0.
\]
\end{lemma}

The next lemma will be modified to our needs to construct a Fredholm module representing the fundamental class of \cite{KPW}. We record it here as a point of reference.

\begin{lemma}[{\cite[Lemma 6.3]{KPW}}]\label{lem:plusinftyKPW}
For any $a\in \mathcal{S}(Q)$ and $b\in \mathcal{U}(P)$, we have 
\begin{align*}
&\lim_{n\to +\infty}\|\alpha_s^n(a)b-b\alpha_s^n(a)\|=0,\\
&\lim_{n\to +\infty}\|\alpha_s^n(a)\alpha_u^{-n}(b)-\alpha_u^{-n}(b)\alpha_s^n(a)\|=0.
\end{align*}
\end{lemma}

To define the fundamental class in \cite{KPW}, the main idea is to find  faithful representations of both Ruelle algebras on the same Hilbert space, which commute modulo compacts, but their non-zero products are never compact. To this end, consider the inflated representation $\overline{\rho_{s}}:\mathcal{R}^s(Q)\to \mathcal{B}(\mathscr{H}\otimes \ell^2(\mathbb Z))$ given, for $a\in \mathcal{S}(Q)$, by 
\begin{equation}\label{eq:inflatedstableRuelle}
a\mapsto \bigoplus_{n\in \mathbb Z} \alpha_s^n(a)\,\,\,\, \text{and} \,\,\,\, u\mapsto 1\otimes B,
\end{equation}
where $u$ is the unitary given by $u\delta_x=\delta_{\varphi(x)}$ and $B$ is the left bilateral shift given by $B\delta_n=\delta_{n-1}$. Moreover, let $\overline{\rho_{u}}:\mathcal{R}^u(P)\to \mathcal{B}(\mathscr{H}\otimes \ell^2(\mathbb Z))$ be the representation given, for $b\in \mathcal{U}(P)$, by
\begin{equation}\label{eq:inflatedusntableRuelle}
b\mapsto b\otimes 1 \,\,\,\, \text{and} \,\,\,\, u\mapsto u\otimes B^*.
\end{equation}
These representations are faithful and satisfy $[\overline{\rho_s}(a),\overline{\rho_u}(u)]=0$, $[\overline{\rho_u}(b),\overline{\rho_s}(u)]=0$ and $[\overline{\rho_s}(u), \overline{\rho_u}(u)]=0$. Thus, Lemmas \ref{lem:stableunstablerankone}, \ref{lem:minusinftylimit} and \ref{lem:plusinftyKPW} yield the following lemma.

\begin{lemma}\label{lem:commutationRuelles}
The $C^*$-algebras $\overline{\rho_s}(\mathcal{R}^s(Q))$ and $\overline{\rho_u}(\mathcal{R}^u(P))$ commute modulo compact operators on $\mathscr{H}\otimes \ell^2(\mathbb Z).$
\end{lemma}

\begin{lemma}[{\cite[Lemma 4.4.13]{Whittaker_PhD}}]\label{notzeroext}
There are $a\in \mathcal{S}(Q)$ and $b\in \mathcal{U}(P)$ such that the operator $\overline{\rho_s}(a)\overline{\rho_u}(b)$ is not compact.
\end{lemma}

Since Ruelle algebras are nuclear, there is a well-defined map $\tau_{\Delta}: \mathcal{R}^s(Q)\otimes \mathcal{R}^u(P)\to \mathcal{Q}(\mathscr{H}\otimes \ell^2(\mathbb Z))$ given on elementary tensors by
\begin{equation}\label{eq: tau_Delta}
\begin{split}
\tau_{\Delta}(au^j&\otimes bu^{j'})= (\overline{\rho_s}\cdot \overline{\rho_u})(a u^j\otimes b u^{j'})+ \mathcal{K}(\mathscr{H}\otimes \ell^2(\mathbb Z))\\
&=(\sum_{n\in \mathbb Z}\alpha_s^n(a)\otimes e_{n,n})(1\otimes B^j)(b\otimes 1) (u^{j'}\otimes B^{-j'})+ \mathcal{K}(\mathscr{H}\otimes \ell^2(\mathbb Z))\\
&= \sum_{n\in \mathbb Z}\alpha_s^n(a)bu^{j'}\otimes e_{n,n+j-j'} + \mathcal{K}(\mathscr{H}\otimes \ell^2(\mathbb Z))
\end{split}
\end{equation}
where $a\in \mathcal{S}(Q),\, b\in \mathcal{U}(P),\, j,j'\in \mathbb Z$, and $e_{n,m}$ are matrix units. Thus $\tau_{\Delta}$ is an extension (or rather a Busby invariant) of $\mathcal{R}^s(Q)\otimes \mathcal{R}^u(P)$ by the compacts. Since $(X,\varphi)$ is irreducible, $\mathcal{R}^s(Q)\otimes \mathcal{R}^u(P)$ is simple and Lemma \ref{notzeroext} proves that $\tau_{\Delta}$ is not the zero extension.

Alternatively, the extension $\tau_{\Delta}$ is constructed as follows. Let $\mathcal{E}$ be the $C^*$-algebra generated by $\overline{\rho_s}(\mathcal{R}^s(Q)), \, \overline{\rho_u}(\mathcal{R}^u(P))$ and $\mathcal{K}(\mathscr{H}\otimes \ell^2(\mathbb Z))$. Since neither $\overline{\rho_s}(\mathcal{R}^s(Q))$ nor $\overline{\rho_u}(\mathcal{R}^u(P))$ contains non-zero compact operators, Lemma \ref{lem:commutationRuelles} shows that
\begin{equation}
\mathcal{E}/ \mathcal{K}(\mathscr{H}\otimes \ell^2(\mathbb Z)) \cong \mathcal{R}^s(Q)\otimes \mathcal{R}^u(P).
\end{equation}

\begin{definition}[{\cite[Definition 6.6]{KPW}}]\label{def:KPWextclass}
The fundamental class $\Delta$ in the $\KKt$-duality pair from Theorem \ref{thm:Ruelleduality} is represented by the extension
\[
0\rightarrow  \mathcal{K}(\mathscr{H}\otimes \ell^2(\mathbb Z)) \rightarrow \mathcal{E} \rightarrow \mathcal{R}^s(Q)\otimes \mathcal{R}^u(P)\rightarrow 0,
\]
or equivalently, by its Busby invariant $\tau_ {\Delta}:\mathcal{R}^s(Q)\otimes \mathcal{R}^u(P)\to \mathcal{Q}(\mathscr{H}\otimes \ell^2(\mathbb Z))$ (\ref{eq: tau_Delta}).
\end{definition}

\section{Lifting the fundamental class: the symmetric representation and intuition towards a geometric fundamental class}\label{sec:Towards_a_KK_1_lift}

Let $(X,\varphi)$ be an irreducible Smale space and fix two periodic orbits $P$ and $Q$ such that $P\cap Q=\varnothing$. In this section, we untwist the extension $\tau_{\Delta}$ from the previous section and develop intuition into building a geometric Fredholm module representative of the fundamental class $\Delta \in \KKt_1(\mathcal{R}^s(Q)\otimes \mathcal{R}^u(P),\mathbb C)$ from Subsection \ref{sec:K-duality_Ruelle}. More precisely, recalling that $\mathscr{H}=\ell^2(X^h(P,Q))$, we aim to explicitly find:
\begin{enumerate}[(i)]
\item a representation $\rho:\mathcal{R}^s(Q)\otimes \mathcal{R}^u(P)\to \mathcal{B}(\mathcal{H})$ on a Hilbert space $\mathcal{H}$,
\item and isometry $V:\mathscr{H}\otimes \ell^2(\mathbb Z)\to \mathcal{H}$, and
\item a unitary $U\in \mathcal{B}(\mathscr{H}\otimes \ell^2(\mathbb Z))$
\end{enumerate}
such that, for all $x\in \mathcal{R}^s(Q)\otimes \mathcal{R}^u(P)$, we have
\begin{equation}\label{eq:KK_1_lift_equation}
(\ad_{\pi}(U)\circ \tau_{\Delta})(x)=V^*\rho(x)V+\mathcal{K}(\mathscr{H}\otimes \ell^2(\mathbb Z)),
\end{equation}
where $\ad_{\pi}(U)$ is conjugation by $\pi(U)=U+\mathcal{K}(\mathscr{H}\otimes \ell^2(\mathbb Z))$. Using Theorem \ref{thm:Liftingtheorem}, it then follows that $\Delta$ is represented by the odd Fredholm module $(\mathcal{H},\rho,2V V^*-1)$ over $\mathcal{R}^s(Q)\otimes \mathcal{R}^u(P)$. In this section we construct (i), (iii) and give intuition towards (ii). The latter is explicitly constructed in Section \ref{sec:averaging_isom}. Then, in Section \ref{sec:DilateKPW} we prove (\ref{eq:KK_1_lift_equation}) and investigate summability properties of the Fredholm module $(\mathcal{H},\rho,2VV^*-1)$. 

To reiterate, an abstract Fredholm module representation of the fundamental class exists due to the nuclearity of Ruelle algebras, the Choi-Effros Lifting Theorem and Stinespring dilation. However, we need a concrete Fredholm module representation for index computations and summability. To do this requires some brute force and heavy analysis. 

It is fairly straightforward to construct the representation $\rho$ in (i), since it should untwist the extension $\tau_{\Delta}:\mathcal{R}^s(Q)\otimes \mathcal{R}^u(P)\to \mathcal{Q}(\mathscr{H}\otimes \ell^2(\mathbb Z))$ in the sense that we want to remove the interaction between the Ruelle algebras in the image. For this, the Hilbert space $\mathcal{H}=\mathscr{H}\otimes \mathscr{H}\otimes \ell^2(\mathbb Z)$ is sufficiently large. On the other hand, constructing a suitable isometry $V$ is hard. The main source of difficulty is the simultaneous actions of two copies of $\mathbb Z$ taking place in 
\[
\mathcal{R}^s(Q)\otimes \mathcal{R}^u(P)=(\mathcal{S}(Q)\rtimes_{\alpha_s} \mathbb Z)\otimes (\mathcal{U}(P)\rtimes_{\alpha_u} \mathbb Z).
\]
Our method is to construct $V$ through an averaging process that asymptotically commutes with the $\mathbb Z^2$-action, so that $VV^*$ commutes modulo compacts with the coefficient algebra $\mathcal{S}(Q)\otimes \mathcal{U}(P)$. The main ingredient is the refining sequence of $\delta$-enlarged Markov partitions defined in \cite{Gero}, and described here in Theorem \ref{thm:theoremgraphSmalespaces}.

\subsection{The symmetric representation}\label{sec:sym_repr} 

We now present an untwisted version of the extension $\tau_{\Delta}:\mathcal{R}^s(Q)\otimes \mathcal{R}^u(P)\to \mathcal{Q}(\mathscr{H}\otimes \ell^2(\mathbb Z))$. In what follows, we consider operators $a\in \mathcal{S}(Q),\, b\in \mathcal{U}(P)$ and $j,j'\in \mathbb Z$, as well as, matrix units $e_{n,m}$. Also, recall the representations $\overline{\rho_s}$ from (\ref{eq:inflatedstableRuelle}) and $\overline{\rho_u}$ from (\ref{eq:inflatedusntableRuelle}) constructed in Subsection \ref{sec:K-duality_Ruelle}.

Consider the (covariant) representations of $\mathcal{R}^s(Q)$ and $\mathcal{R}^u(P)$ on $\mathscr{H}\otimes \mathscr{H}\otimes \ell^2(\mathbb Z)$ which are respectively given on generators by
\begin{align*}
a\mapsto 1\otimes \bigoplus_{n\in \mathbb Z} \alpha_s^n(a) \enspace &\text{and} \enspace u\mapsto 1\otimes 1\otimes B,\\
b\mapsto b\otimes 1\otimes 1 \enspace &\text{and} \enspace u\mapsto u\otimes u\otimes B^*.
\end{align*}
Note that we put $u\mapsto u\otimes u\otimes B^*$ instead of $u\mapsto u\otimes 1\otimes B^*$ so that the two representations commute. Their product is the representation $\widetilde{\rho}:\mathcal{R}^s(Q)\otimes \mathcal{R}^u(P)\to \mathcal{B}(\mathscr{H}\otimes \mathscr{H}\otimes \ell^2(\mathbb Z))$ given by 
\begin{equation}\label{eq:non_symmetric_rep}
\widetilde{\rho}(au^j\otimes bu^{j'})=\sum_{n\in \mathbb Z} bu^{j'}\otimes \alpha_s^n(a)u^{j'}\otimes e_{n,n+j-j'}.
\end{equation}

The representation (\ref{eq:non_symmetric_rep}) untwists $\tau_{\Delta}$ in \eqref{eq: tau_Delta} by separating the actions of $\overline{\rho_s}(\mathcal{S}(Q))$ and $\overline{\rho_u}(\mathcal{U}(P))$ on $\mathscr{H}\otimes \ell^2(\mathbb Z)$. However, observe that in both \eqref{eq: tau_Delta} and (\ref{eq:non_symmetric_rep}), the dynamics are only applied on $a\in \mathcal{S}(Q)$ and not on $b\in \mathcal{U}(P)$. We aim to fix this asymmetry as this will also facilitate the construction of the isometry $V$. The idea is to take \enquote{half} the dynamics of $a$ and put them on $b$.

Let $U=\bigoplus_{n\in \mathbb Z} u^{-\lfloor n/2 \rfloor}$ and consider the symmetrised extension $\ad_{\pi}(U)\circ \tau_{\Delta}$ that yields the same class as $\tau_{\Delta}$ and is given by 
\begin{align}
\label{eq:tau_Delta_sym}
(\ad_{\pi}(U)\circ &\tau_{\Delta})(au^j\otimes bu^{j'})=(\ad(U)\circ (\overline{\rho_s}\cdot \overline{\rho_u}))(au^j\otimes bu^{j'})+\mathcal{K}(\mathscr{H}\otimes \ell^2(\mathbb Z))\\
\notag
&=\sum_{n\in \mathbb Z} u^{-\lfloor \frac{n}{2} \rfloor }b\alpha_s^{n}(a)u^{j'+\lfloor \frac{n+j-j'}{2}\rfloor}\otimes e_{n,n+j-j'}+\mathcal{K}(\mathscr{H}\otimes \ell^2(\mathbb Z)).
\end{align}
Each of these off-diagonal matrices is the sum of the even-coordinates matrix 
\[
\sum_{n\in \mathbb Z} \alpha_u^{-n}(b)\alpha_s^n(a)u^{j'+\lfloor \frac{j-j'}{2}\rfloor}  \otimes e_{2n,2n+j-j'}+\mathcal{K}(\mathscr{H}\otimes \ell^2(\mathbb Z)),
\]
and the odd-coordinates matrix 
\[
\sum_{n\in \mathbb Z} \alpha_u^{-n}(b)\alpha_s^{n+1}(a)u^{j'+\lfloor \frac{j-j'+1}{2}\rfloor } \otimes e_{2n+1,2n+1+j-j'}+\mathcal{K}(\mathscr{H}\otimes \ell^2(\mathbb Z)).
\]

Similarly, let $W=\bigoplus_{n\in \mathbb Z} u^{-\lfloor n/2 \rfloor}\otimes u^{-\lfloor n/2 \rfloor}$ and consider the \textit{symmetric representation} $\rho=\ad(W)\circ \widetilde{\rho}$ given by 
\begin{equation}\label{eq:symmetric_rep}
\rho(au^j\otimes bu^{j'})= \sum_{n\in \mathbb Z} u^{-\lfloor \frac{n}{2} \rfloor} b u^{j'+\lfloor \frac{n+j-j'}{2}\rfloor } \otimes u^{-\lfloor \frac{n}{2} \rfloor} \alpha_s^n(a) u^{j'+\lfloor \frac{n+j-j'}{2}\rfloor }\otimes e_{n,n+j-j'}.
\end{equation}
For a given generator $au^j\otimes bu^{j'} \in \mathcal{R}^s(Q)\otimes \mathcal{R}^u(P)$, the corresponding even-coordinates matrix is 
\[
\sum_{n\in \mathbb Z} \alpha_u^{-n}(b) u^{j'+\lfloor \frac{j-j'}{2}\rfloor } \otimes \alpha_s^n(a) u^{j'+\lfloor \frac{j-j'}{2}\rfloor }\otimes e_{2n,2n+j-j'},
\]
and the odd-coordinates matrix is 
\[
\sum_{n\in \mathbb Z} \alpha_u^{-n}(b) u^{j'+\lfloor \frac{j-j'+1}{2}\rfloor } \otimes \alpha_s^{n+1}(a) u^{j'+\lfloor \frac{j-j'+1}{2}\rfloor }\otimes e_{2n+1,2n+1+j-j'}.
\]
As before, the representation (\ref{eq:symmetric_rep}) is the untwisted version of (\ref{eq:tau_Delta_sym}).

\subsection{Intuition towards constructing the isometry $V$}\label{sec:intuition_isometry}

This subsection contains an exposition of ideas that eventually lead to the definition of the isometry $V:\mathscr{H}\otimes \ell^2(\mathbb Z)\to \mathscr{H}\otimes \mathscr{H}\otimes \ell^2(\mathbb Z)$. Not all of them are mathematically rigorous. However, we believe their presentation will aid in understanding the proofs in the final section. 

We want to find $V$ that satisfies (\ref{eq:KK_1_lift_equation}), for the symmetric representation $\rho$ in (\ref{eq:symmetric_rep}). Using linearity, and the continuity of $V^*\rho(\cdot )V$, this can be reduced to finding $V$ such that, for all $a\in \mathcal{S}(Q),\, b\in \mathcal{U}(P)$ and $j,j'\in \mathbb Z$,
\begin{equation}\label{eq:intuition_isom_1}
V^*\rho(au^j\otimes bu^{j'})V-(\ad(U)\circ (\overline{\rho_s}\cdot \overline{\rho_u}))(au^j\otimes bu^{j'})\in \mathcal{K}(\mathscr{H}\otimes \ell^2(\mathbb Z)).
\end{equation}
Since the operators $\rho(au^j\otimes bu^{j'})$ and $(\ad(U)\circ (\overline{\rho_s}\cdot \overline{\rho_u}))(au^j\otimes bu^{j'})$ are off-diagonal matrices (and actually in the same off-diagonal), it seems natural to expect $V$ to be diagonal; that is, $V=\bigoplus_{n\in \mathbb Z}V_n$, where each $V_n:\mathscr{H}\to \mathscr{H}\otimes \mathscr{H}$ is an isometry. With this in mind, calculating (\ref{eq:intuition_isom_1}) for a given $au^j\otimes bu^{j'}$, we obtain a rather long expression which can be written as the sum of the even-coordinates matrix 
\begin{equation}\label{eq:intuition_isom_2}
\sum_{n\in \mathbb Z}(V_{2n}^*(\alpha_u^{-n}(b) u^{j'+\lfloor \frac{j-j'}{2}\rfloor } \otimes \alpha_s^n(a) u^{j'+\lfloor \frac{j-j'}{2}\rfloor})V_{2n+j-j'}-\alpha_u^{-n}(b)\alpha_s^n(a)u^{j'+\lfloor \frac{j-j'}{2}\rfloor})\otimes e_{2n,2n+j-j'}
\end{equation}
and the odd-coordinates matrix, involving the terms $\alpha_s^{n+1}(a),\alpha_u^{-n}(b)$, $u^{j'+\lfloor \frac{j-j'+1}{2}\rfloor }$ and $V_{2n+1}^*$, $V_{2n+1+j-j'}$ at the matrix unit $e_{2n+1,2n+1+j-j'}$.

Therefore, to show that for every $au^j\otimes bu^{j'}$ the operator (\ref{eq:intuition_isom_1}) is compact, we simply have to show that, for every $n\in \mathbb Z$,  
\begin{equation}\label{eq:intuition_isom_3}
V_{2n}^*(\alpha_u^{-n}(b) u^{j'+\lfloor \frac{j-j'}{2}\rfloor } \otimes \alpha_s^n(a) u^{j'+\lfloor \frac{j-j'}{2}\rfloor})V_{2n+j-j'}-\alpha_u^{-n}(b)\alpha_s^n(a)u^{j'+\lfloor \frac{j-j'}{2}\rfloor}\in \mathcal{K}(\mathscr{H}),
\end{equation}
and also that
\begin{equation}\label{eq:intuition_isom_4}
\lim_{n\to \pm \infty} \|V_{2n}^*(\alpha_u^{-n}(b) u^{j'+\lfloor \frac{j-j'}{2}\rfloor } \otimes \alpha_s^n(a) u^{j'+\lfloor \frac{j-j'}{2}\rfloor})V_{2n+j-j'}-\alpha_u^{-n}(b)\alpha_s^n(a)u^{j'+\lfloor \frac{j-j'}{2}\rfloor}\|=0.
\end{equation}
Similarly for the odd-coordinates matrix. To simplify the notation, it suffices to prove that, for every $a\in \mathcal{S}(Q),\, b\in \mathcal{U}(P)$ and $i,k,l\in \mathbb Z$, we have 
\begin{equation}\label{eq:intuition_isom_5}
V_{2n+l}^*(\alpha_u^{-n}(b) u^{i} \otimes \alpha_s^n(a) u^{i}) V_{2n+k} -\alpha_u^{-n}(b)\alpha_s^n(a)u^{i}\in \mathcal{K}(\mathscr{H}),
\end{equation}
and also that 
\begin{equation}\label{eq:intuition_isom_6}
\lim_{n\to \pm \infty} \|V_{2n+l}^*(\alpha_u^{-n}(b) u^{i} \otimes \alpha_s^n(a) u^{i}) V_{2n+k} -\alpha_u^{-n}(b)\alpha_s^n(a)u^{i}\|=0.
\end{equation}
Then, the proof for the odd-coordinates matrices follows by applying (\ref{eq:intuition_isom_5}) and (\ref{eq:intuition_isom_6}) to $\alpha_s(a)$ in place of $a$. The sequel will be about finding an isometry $V$ that satisfies (\ref{eq:intuition_isom_5}) and (\ref{eq:intuition_isom_6}).

As we see in (\ref{eq:intuition_isom_5}) and (\ref{eq:intuition_isom_6}), the $\mathbb Z$-actions are not negligible. Consequently, such an isometry $V=\bigoplus_{n\in \mathbb Z}V_n$ has to satisfy the following quasi-invariance properties,
\begin{equation}\label{eq:intuition_isom_7}
\lim_{n\to \pm \infty} \|V_{n+1}-V_n\|=0,
\end{equation}  
and in addition,
\begin{equation}\label{eq:intuition_isom_8}
\lim_{n\to \pm \infty} \|(u\otimes u)V_{n}-V_nu\|=0.
\end{equation}
Then, condition (\ref{eq:intuition_isom_6}) holds if, for every $a\in \mathcal{S}(Q),\, b\in \mathcal{U}(P)$, we have
\begin{equation}\label{eq:intuition_isom_9}
\lim_{n\to \pm \infty} \|V_{2n}^*(\alpha_u^{-n}(b) \otimes \alpha_s^n(a)) V_{2n} -\alpha_u^{-n}(b)\alpha_s^n(a)\|=0.
\end{equation}
To summarise, if we find $V$ that satisfies (\ref{eq:intuition_isom_5}), (\ref{eq:intuition_isom_7}), (\ref{eq:intuition_isom_8}) and (\ref{eq:intuition_isom_9}), then the triple $(\mathscr{H}\otimes \mathscr{H}\otimes \ell^2(\mathbb Z),\rho, 2VV^*-1)$ is an odd Fredholm module representative of the fundamental class $\Delta \in \KKt_1(\mathcal{R}^s(Q)\otimes \mathcal{R}^u(P),\mathbb C)$.

In order to find such an isometry $V$ we need to understand the dynamics of the extension $\ad_{\pi}(U)\circ \tau_{\Delta}$ in (\ref{eq:tau_Delta_sym}). The $2n^{\text{th}}$ coordinate of $(\ad(U)\circ (\overline{\rho_s}\cdot \overline{\rho_u}))(a\otimes b)$ is $\alpha_u^{-n}(b)\alpha_s^n(a)$. If $a,b$ are compactly supported, Lemma \ref{lem:minusinftylimit} implies there is $n_0\in \mathbb N$ so that $\alpha_u^{-n}(b)\alpha_s^n(a)=0$, when $n\leq -n_0.$ Therefore, the interesting part is when $n$ approaches $+\infty$. In this case, whenever $\alpha_u^{-n}(b)\alpha_s^n(a)\neq 0$, the dynamics are fairly complicated since the sources and ranges of $\alpha_u^{-n}(b)$ and $\alpha_s^n(a)$ stretch along large segments of global stable and unstable sets, intersecting in many places, thus creating small rectangular neighbourhoods, like a grid. In fact, as $n$ grows, the neighbourhoods become smaller and more numerous. The symmetric representation $\rho$ untwists this picture by separating stable and unstable sets, and the isometry $V$ has to compress it back to the twisted version, with as small an error as possible. Since the aforementioned rectangular neighbourhoods are small, the compression should be related (locally) to the bracket map of the Smale space.

Now that we have built some intuition about $V$, let us be even less mathematically rigorous but hopefully illuminating. We are looking for isometries $\mathscr{H}\to \mathscr{H}\otimes \mathscr{H}$. Recall that $\mathscr{H}=\ell^2(X^h(P,Q))$ and hence these isometries correspond to inclusions $X^h(P,Q)\hookrightarrow X^h(P,Q) \times X^h(P,Q)$.  However, $X^h(P,Q)$ is just a countable dense subset of $X$, so we can search for topological embeddings $X\hookrightarrow X\times X$ instead. The first embedding that comes to mind is the diagonal embedding, however it is too \enquote{small} for our purposes. Inspired by the easy part of Whitney's Embedding Theorem \cite{Whitney} we proceed as follows.

Let $\mathcal{U}=\{U_1,\ldots,U_{\ell}\}$ be a cover of $X$ by open rectangles, and choose a point $g_r\in U_r$, for all $1\leq r\leq \ell$. Then we have that each $U_r=[X^u(g_r,U_r),X^s(g_r,U_r)]$, where $X^u(g_r,U_r):= X^u(g_r,2\ep'_X)\cap U_r$ and $X^s(g_r,U_r):= X^s(g_r,2\ep'_X)\cap U_r$. Moreover, we have the homeomorphisms $\psi_r:U_r\to  X^u(g_r,U_r)\times X^s(g_r,U_r)$ given by 
\begin{equation}
\psi_r(x)=([x,g_r],[g_r,x]).
\end{equation}
That is, each $\psi_r$ is the inverse of the bracket map at $g_r$, and the family $\{U_r,\psi_r\}_{r=1}^{\ell}$ can be considered as an \textit{abstract foliated atlas} (note that a non-empty intersection of rectangles is again a rectangle). Indeed, the notions of transversality and holonomy maps are well-defined in this setting. Also, due to irreducibility of the Smale space $(X,\varphi)$, the topological dimensions $\dim X^u(y),\, \dim X^s(y)$ are independent of $y\in X$ and $\dim X^u(y)+\dim X^s(y) =\dim X$, see \cite[Proposition 5.29]{DKW}. However, the local leaves $X^u(y,\varepsilon_X),\, X^s(y,\varepsilon_X)$ are in general non-Euclidean. 

\newpage
At this point let us make the following (rather vague) assumptions. First, that $X$ is a topological manifold, and let $m=\dim X,\, m_u=\dim X^u(y),\, m_s=\dim X^s(y)$, for some $y\in X$. Second, that the each homeomorphism $\psi_r$ yields a homeomorphism $\widetilde{\psi_r}:U_r\to \mathbb R^m$ (onto its image) that has the form
\begin{equation}
\widetilde{\psi_r}(x)=(\psi^u_r([x,g_r]),\psi^s_r([g_r,x])),
\end{equation}
where $\psi^u_r:X^u(g_r,U_r) \to \mathbb R^{m_u}$ and $\psi^s_r:X^s(g_r,U_r) \to \mathbb R^{m_s}$.

Consider $\{F_r\}_{r=1}^{\ell}$ to be a partition of unity subordinate to the open cover $\{U_r\}_{r=1}^{\ell}$. Then, if we extend every $\widetilde{\psi_r}$ to be zero outside of $U_r$, the map $\Psi:X\to \mathbb R^{\ell(m+1)}$ defined as 
\begin{equation}\label{eq:top_emb_Psi}
\Psi=(F_1\widetilde{\psi_1},\ldots,F_{\ell}\widetilde{\psi_{\ell}},F_1,\ldots, F_{\ell}), 
\end{equation}
is continuous. Also, the coordinates $(F_1,\ldots, F_{\ell})$ are included to guarantee that $\Psi$ is injective.

With this in mind, and following the discussion so far, we can estimate the form of the desired isometries $\mathscr{H}\to \mathscr{H}\otimes \mathscr{H}$. Consider again the cover $\{U_r\}_{r=1}^{\ell}$ of open rectangles, choose the points $\{g_r\}_{r=1}^{\ell}$ to be in $X^h(P,Q)$, and let $\{F_r\}_{r=1}^{\ell}$ be a partition of unity  subordinate to this cover. Our model isometry is defined on basis vectors as
\begin{equation}\label{eq:model_isom}
\delta_x\mapsto \sum_{r=1}^{\ell} F_r(x)^{1/2}\delta_{[x,g_r]}\otimes \delta_{[g_r,x]}.
\end{equation}
It is well-defined with respect to the \textit{standard convention}; the bracket map returns the empty set for points $x,y$ with $d(x,y)> \varepsilon_X$ and that, the Dirac delta function on the empty set returns zero \cite[p. 281]{KPW}. The fact that it is actually an isometry is proved in Section \ref{sec:averaging_isom}. Now, it is hard to miss that this isometry is similar to the topological embedding $\Psi$ from \eqref{eq:top_emb_Psi}. One difference is that we use square roots of the partition functions, but this is only because $\sum_r F_r(x)=1$. Moreover, we do not have anything corresponding to the coordinates $(F_1,\ldots, F_{\ell})$ of $\Psi$. However, there is no need because \eqref{eq:model_isom} maps onto orthogonal vectors, and using the bracket axioms, injectivity is guaranteed.

\begin{remark}
It is important to mention that the isometry (\ref{eq:model_isom}) is strongly related to the $\varepsilon$-partitions needed to define the $\Kt$-theory duality class of the Spanier-Whitehead $\KKt$-duality between the Ruelle algebras, see \cite[Section 5]{KPW}. We find it intriguing that the same ingredients are needed to define both $\KKt$-duality classes, at least in the form of a Kasparov module.
\end{remark}

Finally, the quasi-invariance properties (\ref{eq:intuition_isom_7}), (\ref{eq:intuition_isom_8}) can be achieved by defining each isometry $V_n$ as the average of \enquote{basic} isometries $\mathscr{H}\to \mathscr{H}\otimes \mathscr{H}$ of the type (\ref{eq:model_isom}), with mutually orthogonal ranges. For this to happen, one has to carefully choose the covers $\{U_r\}_{r=1}^{\ell}$, the set of points $\{g_r\}_{r=1}^{\ell}\subset X^h(P,Q)$ and the partitions of unity $\{F_r\}_{r=1}^{\ell}$. The key tool is Theorem \ref{thm:theoremgraphSmalespaces} which provides refining sequences of $\delta$-enlarged Markov partitions. This is studied in Sections \ref{sec:pou_Markov} and \ref{sec:averaging_isom}.

\section{Lifting the fundamental class: Dynamic partitions of unity and aperiodic samples of Markov partitions}\label{sec:pou_Markov} 

In this section we develop tools for constructing the isometry $V$ of \eqref{eq:KK_1_lift_equation}. Let $(X,\varphi)$ be an irreducible Smale space and $P,Q$ be periodic orbits such that $P\cap Q= \varnothing$. Further, let $(\mathcal{R}_n^{\delta})_{n\geq 0}$ be a refining sequence of $\delta$-enlarged Markov partitions obtained by Theorem \ref{thm:theoremgraphSmalespaces}. 

The first tool will be a sequence of Lipschitz partitions of unity $(\mathcal{F}_n)_{n\geq 0}$ on $X$ associated to the sequence $(\mathcal{R}_n^{\delta})_{n\geq 0}$. The second tool of our construction will be a set $\mathcal{G}\subset X^h(P,Q)$ of dynamically independent points that are chosen from each rectangle in every open cover $\mathcal{R}_n^{\delta}$, which we will call an aperiodic sample. Before constructing these tools recall that $P\cap Q=\varnothing$ and hence the set $X^h(P,Q)$ has no periodic points. For the sequel, it will be convenient to choose a random ordering in every $\mathcal{R}_n^{\delta}$ and write 
\begin{equation}\label{eq:orderingcovers}
\mathcal{R}_n^{\delta}=\{R_{n,k}^{\delta}:1\leq k\leq \# \mathcal{R}_n^{\delta}\}.
\end{equation}

First, define $h_{0,1}(x)=1$, for all $x\in X$, and for every $n\in \mathbb N$ and $1\leq k \leq \# \mathcal{R}_n^{\delta}$ consider the function $h_{n,k}:X\to [0, \infty)$ given by  
\begin{equation}\label{eq:pou1}
h_{n,k}(x)=d(x,X\setminus R_{n,k}^{\delta}).
\end{equation}
Every such function is $1$-Lipschitz by the triangle inequality. Then, for every $n\geq 0$, the family of functions $\mathcal{F}_n=\{F_{n,k}:1\leq k \leq \# \mathcal{R}_n^{\delta}\}$ on $X$, defined as
\begin{equation}\label{eq:pou2}
F_{n,k}(x)=\frac{h_{n,k}(x)}{\sum_j h_{n,j}(x)},
\end{equation}
forms a partition of unity on $X$. Note that $\mathcal{F}_n$ is not subordinate to $\mathcal{R}_n^{\delta}$, unless $X$ is zero-dimensional, because $F_{n,k}(x)>0$ if and only if $x\in R_{n,k}^{\delta}.$ By repeating a simple (but interesting) calculation found in \cite[Proposition 1]{Bell} and Theorem \ref{thm:theoremgraphSmalespaces} we obtain the following.

\begin{prop}\label{prop:Lip_pou_con}
For every $n\geq 0$ and $1\leq k \leq \# \mathcal{R}_n^{\delta}$, the Lipschitz constant of $F_{n,k}$ satisfies $$\Lip(F_{n,k})\leq \frac{2(\# \mathcal{R}_1^{\delta})^2+1}{\Leb(\mathcal{R}_n^{\delta})}.$$
\end{prop}

We now move on to the construction of the second key tool. For this we require the next definition. Let $\mathcal{R}^{\delta}$ denote $\coprod_{n\geq 0} \mathcal{R}_n^{\delta}$, which is the set of vertices of the approximation graph associated to $(\mathcal{R}_n^{\delta})_{n\geq 0}$, see \cite[Definition 2.18]{Gero}.

\begin{definition}\label{def:sampling}
An $\mathcal{R}^{\delta}$\textit{-sampling function} is a function $c:\mathcal{R}^{\delta}\to X$ such that $c(R^{\delta})\in R^{\delta}$, for all $R^{\delta}\in \mathcal{R}^{\delta}$. The image $c(\mathcal{R}^{\delta})$ will be called an $\mathcal{R}^{\delta}$\textit{-sample} of the Smale space $(X,\varphi)$.
\end{definition}

\begin{remark}
The existence of an $\mathcal{R}^{\delta}$-sampling function follows from the Axiom of Choice. Also, note that every $\mathcal{R}^{\delta}$-sample is automatically a dense subset of $X$.
\end{remark}

The following definition plays a crucial role in constructing the isometry $V$. 

\begin{definition}\label{def:aperiodic_sampling}
An \textit{aperiodic} $\mathcal{R}^{\delta}$\textit{-sample} of the Smale space $(X,\varphi)$ is the image $c(\mathcal{R}^{\delta})$ of an injective $\mathcal{R}^{\delta}$-sampling function $c:\mathcal{R}^{\delta}\to X$ where $\varphi^j(c(\mathcal{R}^{\delta}))\cap c(\mathcal{R}^{\delta})=\varnothing$, for all $j\in \mathbb Z\setminus \{0\}$.
\end{definition}

\begin{remark}\label{rem:aperiodic_sampling}
The ordering in (\ref{eq:orderingcovers}) yields the bijection $$\{(n,k):n\geq 0, \, 1\leq k\leq \# \mathcal{R}_n^{\delta}\}\to \mathcal{R}^{\delta}$$ that maps $(n,k)$ to $R_{n,k}^{\delta}$. Therefore, given an injective $\mathcal{R}^{\delta}$-sampling function $c$, one has $c(R_{n,k}^{\delta})\neq c(R_{m,\ell}^{\delta})$ whenever $(n,k)\neq (m,\ell)$. In Proposition \ref{prop:aperiodic_set} we construct a particular injective $\mathcal{R}^{\delta}$-sampling function $c'$ whose image is aperiodic, and from that point on, every $c'(R_{n,k}^{\delta})$ will be denoted by $g_{n,k}$ and the image $c'(\mathcal{R}^{\delta})$ will be $\mathcal{G}=\{g_{n,k}:n\geq 0, \, 1\leq k\leq \# \mathcal{R}_n^{\delta}\}.$
\end{remark}

We now aim to construct an aperiodic $\mathcal{R}^{\delta}$-sample of $(X,\varphi)$ inside $X^h(P,Q)$. First, we require the next lemma which is an easy consequence of irreducibility.

\begin{lemma}\label{lem:thin_orbit1}
For every $x\in X^h(P,Q)$, the orbit $O(x)=\{\varphi^n(x):n\in \mathbb Z\}$ is not dense in $X$.
\end{lemma}

\begin{proof}
Let $z\in X\setminus (P\cup Q)$ and $\eta =d(z, P\cup Q)>0$, and assume to the contrary that there is some $x\in X^h(P,Q)$ for which $O(x)$ is dense in $X$. Since $x\in X^h(P,Q)$, there is some $N\in \mathbb N$ so that for every $n> N$ we have $d(\varphi^n(x),P)<\eta /2$ and $d(\varphi^{-n}(x),Q)<\eta /2$. Due to irreducibility, the space $X$ has no isolated points and hence the set $T(x)=O(x)\setminus \{\varphi^{-N}(x),\ldots , x, \ldots , \varphi^N(x)\}$ is still dense in $X$. However, $T(x)\cap B(z, \eta /3)=\varnothing$, a contradiction.
\end{proof}

Further, we require the next elementary lemma which holds for dynamical systems ($Z,\psi$), where $Z$ is an infinite topological space and $\psi:Z\to Z$ is a continuous map.

\begin{lemma} \label{lem:thin_orbit2}
For a dynamical system $(Z,\psi)$ the following are equivalent:
\begin{enumerate}[(1)]
\item $(Z,\psi)$ is irreducible;
\item every closed, $\psi$-invariant, proper subset of Z has empty interior.
\end{enumerate}
\end{lemma}

For the next proposition we use the ordering (\ref{eq:orderingcovers}) and Remark \ref{rem:aperiodic_sampling}. We also note that the irreducibility of $(X,\varphi)$ plays an important role.

\begin{prop}\label{prop:aperiodic_set}
The Smale space $(X,\varphi)$ has an aperiodic $\mathcal{R}^{\delta}$-sample $\mathcal{G} \subseteq X^h(P,Q)$.  
\end{prop}

\begin{proof}
We construct a sampling function $c:\mathcal{R}^{\delta}\to X$ inductively. 

Let $B=\{(n,k):n\geq 0, \, 1\leq k\leq \# \mathcal{R}_n^{\delta}\}$, which is obviously in bijection with $\mathcal{R}^{\delta}$, and endow $B$ with the lexicographic order $(n,k)\leq (m,\ell)$, if $(n<m)$ or $(n=m$ and $k\leq \ell)$. It will be convenient to write $B=\{e_i:i\in \mathbb N\}$, where $e_1=(0,1)$ and $e_i<e_{i+1}$, for all $n\in \mathbb N$. We will construct a set $\mathcal{G}=\{g_{e_i}:i\in \mathbb N\}$ in  $X^h(P,Q)$ as follows:
\begin{enumerate}[(i)]
\item Choose $g_{e_1}\in X^h(P,Q)$;
\item Having chosen $g_{e_1},\ldots, g_{e_n}$, choose $$g_{e_{n+1}}\in R_{e_{n+1}}^{\delta}\cap (X^h(P,Q)\setminus \bigcup_{i=1}^{n}\cl (O(g_{e_i}))).$$
\end{enumerate}
To see that $\mathcal{G}$ is well defined, notice from Lemmas \ref{lem:thin_orbit1} and \ref{lem:thin_orbit2}, for every $g\in X^h(P,Q)$, the set $\cl(O(g))$ is a closed, $\varphi$-invariant, proper subset of $X$ with empty interior. Therefore, every $X\setminus \cl(O(g))$ is open and dense in $X$. Moreover, since the intersection of an open dense set with a dense set is a dense set we can always inductively choose new points to belong to $\mathcal{G}$.

The sampling function $c:\mathcal{R}^{\delta}\to X$ is given by $c(R_{e_{i}}^{\delta})=g_{e_i}\in R_{e_{i}}^{\delta}$, where $i\in \mathbb N$, and the $\mathcal{R}^{\delta}$-sample is $c(\mathcal{R}^{\delta})=\mathcal{G}$. To see that it is injective, let $e_m\neq e_n$ and assume that $e_m < e_n$. Then, $$g_{e_{n}}\in X^h(P,Q)\setminus \bigcup_{i=1}^{n-1}\cl (O(g_{e_i})),$$ while $g_{e_m}\in \bigcup_{i=1}^{n-1}\cl (O(g_{e_i})),$ and hence $g_{e_n}\neq g_{e_m}$.

Moreover, $\varphi^j(\mathcal{G})\cap \mathcal{G}=\varnothing$, for all $j\in \mathbb Z\setminus \{0\}$. For this, assume to the contrary that there are $g_{e_m},g_{e_n}\in \mathcal{G}$ and $j\in \mathbb Z\setminus \{0\}$ such $\varphi^j(g_{e_n})=g_{e_m}$. If $n=m$ we get a contradiction because $g_{e_n}$ is not periodic. Assume that $m<n$. Since $X^h(P,Q)$ and every orbit is $\varphi$-invariant, we have that $$\varphi^j(g_{e_{n}})\in X^h(P,Q)\setminus \bigcup_{i=1}^{n-1}\cl (O(g_{e_i})),$$ while $g_{e_m}\in \bigcup_{i=1}^{n-1}\cl (O(g_{e_i})),$ and hence $\varphi^j(g_{e_n})\neq g_{e_m}$, leading to a contradiction. Similarly if $n<m$. Finally, we can drop the notation of the $e_i$'s and get the desired aperiodic $\mathcal{R}^{\delta}$-sample $$\mathcal{G}=\{g_{n,k}:n\geq 0, \, 1\leq k\leq \# \mathcal{R}_n^{\delta}\}$$ of the Smale space $(X,\varphi)$.
\end{proof}

\section{Lifting the fundamental class: the averaging isometries}\label{sec:averaging_isom} 

In this section we construct an isometry $V$ satisfying equation (\ref{eq:KK_1_lift_equation}). Then, in Section \ref{sec:DilateKPW} we prove that $V$, together with the symmetric representation $\rho$ in (\ref{eq:symmetric_rep}), yields a $\theta$-summable Fredholm module representative of the fundamental class $\Delta$. Following the discussion in Subsection \ref{sec:intuition_isometry}, the isometry $V$ should have the form $\bigoplus_{n\in \mathbb Z} V_n$ for isometries $V_n:\mathscr{H}\to \mathscr{H}\otimes \mathscr{H}$, where $\mathscr{H}=\ell^2(X^h(P,Q))$. Each $V_n$ should be an average of isometries (of the form (\ref{eq:model_isom})) with mutually orthogonal ranges, so that the quasi-invariance properties (\ref{eq:intuition_isom_7}) and (\ref{eq:intuition_isom_8}) are met, and also that property (\ref{eq:intuition_isom_9}) holds. 

Let $(\mathcal{R}_n^{\delta})_{n\geq 0}$ be a refining sequence of $\delta$-enlarged Markov partitions of the irreducible Smale space $(X,\varphi)$, obtained by Theorem \ref{thm:theoremgraphSmalespaces}. Recall the set $\mathcal{R}^{\delta}=\coprod_{n\geq 0} \mathcal{R}_n^{\delta}$. Similarly, as in Section \ref{sec:pou_Markov}, choose a random ordering in every $\mathcal{R}_n^{\delta}$ and write 
\begin{equation}\label{eq:orderingcovers2}
\mathcal{R}_n^{\delta}=\{R_{n,k}^{\delta}:1\leq k\leq \# \mathcal{R}_n^{\delta}\}.
\end{equation}
For our purposes here, it is important to recall that for $n\in \mathbb N$ and $1\leq k\leq \# \mathcal{R}_n^{\delta}$, if $x,y\in R_{n,k}^{\delta}$ then 
\begin{equation}\label{eq:averaging1}
d(\varphi^r(x),\varphi^r(y))< \varepsilon_X',
\end{equation} 
for all $|r|\leq n-1$.

Our goal is to show that the refining sequence $(\mathcal{R}_n^{\delta})_{n\geq 0}$ produces a family 
\begin{equation}\label{eq:averaging2}
\mathcal{T}=\{\iota_{n,r}: n\in \mathbb N,\, |r|\leq n-1\}
\end{equation}
of \textit{basic isometries} $\iota_{n,r}:\mathscr{H}\to \mathscr{H}\otimes \mathscr{H}$ with mutually orthogonal ranges, which we can average in a certain way. From the results of Section \ref{sec:pou_Markov} (specifically see (\ref{eq:pou2}) and Proposition \ref{prop:aperiodic_set}) we have:
\begin{enumerate}[(i)]
\item for $n\geq 0$, there is a Lipschitz partition of unity $\mathcal{F}_n=\{F_{n,k}:1\leq k \leq \# \mathcal{R}_n^{\delta}\}$ on $X$ such that $F_{n,k}(x)>0$ if and only if $x\in R_{n,k}^{\delta}.$ As usual, we require the square roots of these functions, so we define the family 
\[
\mathcal{F}_n^{1/2}=\{f_{n,k}:1\leq k \leq \# \mathcal{R}_n^{\delta}\}, \quad \text{where $ f_{n,k}=F_{n,k}^{1/2};$}
\]
\item there is an aperiodic $\mathcal{R}^{\delta}$-sample 
$$
\mathcal{G}=\{g_{n,k}\in R_{n,k}^{\delta}: n\geq 0,\, 1\leq k \leq \# \mathcal{R}_n^{\delta}\}\subset X^h(P,Q),
$$
meaning that $g_{n,k}\neq g_{m,\ell}$, if $(n,k)\neq (m,\ell)$, and $\varphi^r(\mathcal{G})\cap \mathcal{G}=\varnothing$, for all $r\in \mathbb Z \setminus \{0\}.$
\end{enumerate}
 
Given (i) and (ii) we now construct the family $\mathcal{T}$. For every $n\in \mathbb N$, define the operator $\iota_{n,0}:\mathscr{H}\to \mathscr{H}\otimes \mathscr{H}$ on basis vectors by
\begin{equation}\label{eq:averaging3}
\iota_{n,0}(\delta_y)=\sum_{k=1}^{\#\mathcal{R}_n^{\delta}} f_{n,k}(y)\delta_{[y,g_{n,k}]}\otimes \delta_{[g_{n,k},y]}.
\end{equation}
It is well-defined given our \textit{standard convention} that the bracket map returns the empty set for points $x,y$ with $d(x,y)> \varepsilon_X$ and that, the Dirac delta function on the empty set returns zero. Recall that each $R_{n,k}^{\delta}$ is a rectangle with $\diam (R_{n,k}^{\delta})\leq \varepsilon_X'$ and therefore we can use the notation $X^u(g_{n,k},R_{n,k}^{\delta})=X^u(g_{n,k},2\varepsilon_X')\cap R_{n,k}^{\delta}$. Similarly for the stable case. 

The adjoint is given on basis vectors by 
\begin{equation}\label{eq:averaging4}
\iota_{n,0}^*(\delta_x \otimes \delta_z)=f_{n,k}([x,z])\delta_{[x,z]},
\end{equation}
if $x\in X^u(g_{n,k},R_{n,k}^{\delta}),\, z\in X^s(g_{n,k},R_{n,k}^{\delta})$ (for a unique $k$), and is zero otherwise. Indeed, we have that
\[
\langle \iota_{n,0}(\delta_y),\delta_x\otimes \delta_z \rangle=\sum_{k} f_{n,k}(y),
\]
where the sum is taken over all $k$ such that $f_{n,k}(y)\neq 0$ and $x=[y,g_{n,k}],\, z=[g_{n,k},y]$. Let $k_1$ be one of these $k$s and since $f_{n,k_1}(y)\neq 0$ it holds $y\in R_{n,k_1}^{\delta}$. Consequently, we have $x\in X^u(g_{n,k_1},R_{n,k_1}^{\delta}),\, z\in X^s(g_{n,k_1},R_{n,k_1}^{\delta}).$ Now, if $k_2$ is one of these $k$s then, $$g_{n,k_2}\in X^s(g_{n,k_1},\varepsilon_X)\cap X^u(g_{n,k_1},\varepsilon_X)=\{g_{n,k_1}\},$$ and hence $k_2=k_1$. Finally, note that $[x,z]=y$.

\begin{lemma}\label{lem:isometries}
For every $n\in \mathbb N$ the operator $\iota_{n,0}$ is an isometry.
\end{lemma}

\begin{proof}
For $\delta_y\in \mathscr{H}$ we have 
\begin{align*}
\iota_{n,0}^* \iota_{n,0}(\delta_y)&=\sum_{k=1}^{\#\mathcal{R}_n^{\delta}} f_{n,k}(y)\iota_{n,0}^*(\delta_{[y,g_{n,k}]}\otimes \delta_{[g_{n,k},y]})= \sum_{k=1}^{\#\mathcal{R}_n^{\delta}} f_{n,k}(y)^2\delta_y=\delta_y. \qedhere
\end{align*}
\end{proof}

Let $n\in \mathbb N$ and for every $r\in \mathbb Z$ we define the isometry $\iota_{n,r}=(u\otimes u)^r \iota_{n,0} u^{-r},$ where $u$ is the unitary on $\mathscr{H}$ given by $u(\delta_x)=\delta_{\varphi (x)}.$ The sequence $(\iota_{n,r})_{r\in \mathbb Z}$ can be considered as the $\mathbb Z$-orbit of $\iota_{n,0}$, and its interesting part lies in the central terms for $|r|\leq n-1$. Indeed, if $|r|\leq n-1$ then from (\ref{eq:averaging1}) we have that 
\begin{equation}\label{eq:averaging5}
\iota_{n,r}(\delta_y)=\sum_{k=1}^{\#\mathcal{R}_n^{\delta}} f_{n,k}(\varphi^{-r}(y))\delta_{[y,\varphi^r(g_{n,k})]}\otimes \delta_{[\varphi^r(g_{n,k}),y]}.
\end{equation}
Then, it is straightforward to see that 
\begin{equation}\label{eq:averaging6}
\iota_{n,r}^*(\delta_x \otimes \delta_z)=f_{n,k}(\varphi^{-r}[x,z])\delta_{[x,z]},
\end{equation}
if $\varphi^{-r}(x)\in X^u(g_{n,k},R_{n,k}^{\delta}),\, \varphi^{-r}(z)\in X^s(g_{n,k},R_{n,k}^{\delta})$ (for a unique $k$), and is zero otherwise. Since $\mathcal{G}$ is an aperiodic $\mathcal{R}^{\delta}$-sample, we obtain the following orthogonality result.

\begin{lemma}\label{lem:orthogonality}
For every $m,n\in \mathbb N$ and $r,s\in \mathbb Z$ with $|r|\leq n-1,\, |s|\leq m-1$ so that $(n,r)\neq (m,s)$, we have that $\iota_{m,s}^*\iota_{n,r}=0$.
\end{lemma}

\begin{proof}
Let $m,n,r,s$ as in the statement and assume to the contrary that $\iota_{m,s}^*\iota_{n,r}\neq 0$. Then there is a basis vector $\delta_y\in \mathscr{H}$ such that $\iota_{m,s}^*\iota_{n,r}(\delta_y)\neq 0$. We have that 
\[
\iota_{m,s}^*\iota_{n,r}(\delta_y)=\sum_{k=1}^{\#\mathcal{R}_n^{\delta}} f_{n,k}(\varphi^{-r}(y))\iota_{m,s}^*(\delta_{[y,\varphi^r(g_{n,k})]}\otimes \delta_{[\varphi^r(g_{n,k}),y]}),
\]
and hence there is $k$ such that $f_{n,k}(\varphi^{-r}(y))\iota_{m,s}^*(\delta_{[y,\varphi^r(g_{n,k})]}\otimes \delta_{[\varphi^r(g_{n,k}),y]})\neq 0.$ Now, since $f_{n,k}(\varphi^{-r}(y))\neq 0$ and $|r|\leq n-1$ we obtain $\varphi^{-r}(y)\in R_{n,k}^{\delta}$ and so $d(y,\varphi^r(g_{n,k}))\leq \varepsilon_X'$. Therefore, 
\begin{equation}\label{eq:orthogonality1}
[y,\varphi^r(g_{n,k})]\in X^u(\varphi^r(g_{n,k}),\varepsilon_X/2),\quad  [\varphi^r(g_{n,k}),y]\in X^s(\varphi^r(g_{n,k}),\varepsilon_X/2).
\end{equation}
Moreover, since $\iota_{m,s}^*(\delta_{[y,\varphi^r(g_{n,k})]}\otimes \delta_{[\varphi^r(g_{n,k}),y]})\neq 0$, there exists some $1\leq l\leq \mathcal{R}_m^{\delta}$ so that 
$$\varphi^{-s}[y,\varphi^r(g_{n,k})]\in X^u(g_{m,l}, R_{m,l}^{\delta}),\quad \varphi^{-s}[\varphi^r(g_{n,k}),y]\in X^s(g_{m,l}, R_{m,l}^{\delta}).$$
The fact that $|s|\leq m-1$ implies that 
$$[y,\varphi^r(g_{n,k})]\in X^u(\varphi^s(g_{m,l}),\varepsilon_X'),\quad [\varphi^r(g_{n,k}),y]\in X^s(\varphi^s(g_{m,l}),\varepsilon_X'),$$
and using (\ref{eq:orthogonality1}) we get that 
\begin{equation}\label{eq:orthogonality2}
\varphi^r(g_{n,k})\in X^s(\varphi^s(g_{m,l}),\varepsilon_X)\cap X^u(\varphi^s(g_{m,l}),\varepsilon_X).
\end{equation}
Consequently, $\varphi^r(g_{n,k})=\varphi^s(g_{m,l}).$ If $r=s$ then $g_{n,k}=g_{m,l}$. Since the $\mathcal{R}^{\delta}$-sample is aperiodic we get that $n=m$, obtaining a contradiction. If $r\neq s$, then $\varphi^{r-s}(g_{n,k})=g_{m,l}$ and hence $\varphi^{r-s}(\mathcal{G})\cap \mathcal{G}\neq \varnothing$, leading again to a contradiction since $\mathcal{G}$ is an aperiodic $\mathcal{R}^{\delta}$-sample. Thus $\iota_{m,s}^*\iota_{n,r}=0$.
\end{proof}

\begin{remark}\label{rem:holdercoef}
We follow the notation of (\ref{eq:averaging5}). Restricting to isometries $\iota_{n,r}$ with $|r|\leq n-1$ allows us to control the coefficient functions $f_{n,k}\circ \varphi^{-r}$. More precisely, in Section \ref{sec:DilateKPW} we will consider sequences $(x_n)_{n\in \mathbb N},\, (y_n)_{n\in \mathbb N}$ in $X$ satisfying $d(x_n,y_n)\leq \lambda^{-n}$ eventually. For these sequences we want to have that 
\begin{equation}\label{eq:holdercoef}
\lim_{n\to \infty}|f_{n,k}\circ \varphi^{-r}(x_n)-f_{n,k}\circ \varphi^{-r}(y_n)|=0,
\end{equation}
for every $k,r$. Even better we'd have that the limit is converging to zero exponentially. From Proposition \ref{prop:Lip_pou_con} we have that for $c=(2(\# \mathcal{R}_1^{\delta})^2+1)^{1/2}$, each $f_{n,k}$ is $1/2$-H{\"o}lder with coefficient $$\text{H{\"o}l}(f_{n,k})\leq \frac{c}{\Leb(\mathcal{R}_n^{\delta})^{1/2}}.$$ If the metric $d$ is self-similar then $\varphi,\varphi^{-1}$ are $\lambda$-Lipschitz, and from Theorem \ref{thm:theoremgraphSmalespaces} there is a constant $c'>0$ such that for every $n,k$ we have
\[
\text{H{\"o}l}(f_{n,k})\leq c' \lambda^{n/2}.
\]
Therefore, each $f_{n,k}\circ \varphi^{-r}$ is $1/2$-H{\"o}lder with $\text{H{\"o}l}(f_{n,k}\circ \varphi^{-r})\leq c' \lambda^{(n+|r|)/2}\leq c' \lambda^n.$ Hence, the limit may not be zero. This can be fixed if we linearly slow down the coefficient functions; that is, for every $n\in \mathbb N,\, k\in \{1,\ldots ,\#\mathcal{R}_{\lceil n/4 \rceil}^{\delta}\}$, $|r|\leq \lceil n/4 \rceil-1$ we have
\begin{equation}\label{eq:holdercoef}
|f_{\lceil n/4 \rceil ,k}\circ \varphi^{-r}(x_n)-f_{\lceil n/4 \rceil,k}\circ \varphi^{-r}(y_n)|\leq c' \lambda^{1-n/4}.
\end{equation}
This idea works also for a non-self-similar metric $d$, for which the map $\varphi$ is bi-Lipschitz. However, the slow down may then not be linear, and this will effect the summability properties of the Fredholm module representative of the fundamental class.
\end{remark}

In order to visualise the basic isometries we consider their range projections. Let $n\in \mathbb N$, and for $r\in \mathbb Z$ consider the projections $p_{n,r}=\iota_{n,r}\iota_{n,r}^*=(u\otimes u)^r p_{n,0}(u\otimes u)^{-r}$. For $|r|\leq n-1$ they are given on basis vectors by
\begin{equation}\label{eq:averaging7}
p_{n,r}(\delta_x\otimes \delta_z)=f_{n,l}(\varphi^{-r}[x,z])\sum_{k=1}^{\# \mathcal{R}_n^{\delta}}f_{n,k}(\varphi^{-r}[x,z])\delta_{[x, \varphi^r(g_{n,k})]}\otimes \delta_{[\varphi^r(g_{n,k}),z]},
\end{equation}
if $\varphi^{-r}(x)\in X^u(g_{n,l},R_{n,l}^{\delta}),\, \varphi^{-r}(z)\in X^s(g_{n,l},R_{n,l}^{\delta})$ (for a unique $l$), and is zero otherwise. Also, from Lemma \ref{lem:orthogonality}, for $m,n\in \mathbb N$ and $r,s\in \mathbb Z$ with $|r|\leq n-1,\, |s|\leq m-1$ and $(n,r)\neq (m,s)$, we have that 
\begin{equation}\label{eq:averaging8}
p_{m,s}p_{n,r}=0.
\end{equation}

We now construct the isometry $V=\bigoplus_{n\in \mathbb Z} V_n$, from isometries $V_n:\mathscr{H}\to \mathscr{H}\otimes \mathscr{H}$ by averaging over particular basic isometries. It will aid intuition to envisage the family $\mathcal{T}=\{\iota_{n,r}: n\in \mathbb N,\, |r|\leq n-1\}$ as a triangle
\vspace{-0.6cm}
\begin{center}
$$\begin{array}{cccccccccc}
& & & & \iota_{1,0} \\
& & &   \iota_{2,\text{-}1} & \iota_{2,0} & \iota_{2,1}\\
& & \iota_{3,\text{-}2} & \iota_{3,\text{-}1} & \iota_{3,0} & \iota_{3,1} & \iota_{3,2}\\
& \iota_{4,\text{-}3} & \iota_{4,\text{-}2} & \iota_{4,\text{-}1} &\iota_{4,0} & \iota_{4,1} & \iota_{4,2} & \iota_{4,3}\\
\Ddots & \vdots & \vdots & \vdots & \vdots & \vdots &\vdots & \vdots &\ddots 
\end{array}$$
\end{center}
The gist is to find an appropriate way to average in $\mathcal{T}$ and construct each $V_n$. Recall that we want to achieve the following quasi-invariance conditions:
\begin{equation}\label{eq:averaging9}
\lim_{n\to \pm \infty} \|V_{n+1}-V_n\|=0,
\end{equation}  
and
\begin{equation}\label{eq:averaging10}
\lim_{n\to \pm \infty} \|(u\otimes u)V_{n}-V_nu\|=0.
\end{equation}
Every $V_n$ will be the average of finitely many isometries. Also, the bigger the $n$, the more isometries that will need to be averaged. The first quasi-invariance condition indicates that the isometries whose average is $V_{n}$ should coincide more and more with the isometries that produce $V_{n+1}.$ The second quasi-invariance condition suggests that the chosen isometries for every $V_n$, should be further and further away (horizontally) from the boundaries of the triangle; the sets $\{\iota_{n,-n+1}:n\in \mathbb N\}$ and $\{\iota_{n,n-1}:n\in \mathbb N\}$. In this way, every $V_n$, eventually, will not get shifted outside of the triangle by powers of $u$ and $u\otimes u$. Of course, to achieve (\ref{eq:averaging10}) for any fixed power of $u$ and $u\otimes u$, it suffices to choose isometries from $\mathcal{T}$ that are simply not in the two boundary isometries. However, we are interested in achieving condition (\ref{eq:intuition_isom_6}), which requires that we asymptotically stay in the triangle completely. This also forces to construct each $V_n$ as a moving average directed towards the bottom of the triangle. 

More precisely, let $\iota_{0,0}:\mathscr{H}\to \mathscr{H}\otimes \mathscr{H}$ be the diagonal embedding, and for every $m\geq 0$ define
\begin{equation}\label{eq:averaging11}
\theta_m=\sum_{r=-m}^{m} \iota_{2m,r}.
\end{equation}
Due to Lemma \ref{lem:orthogonality}, for every $m,m'\in \mathbb N$ with $m\neq m'$, it holds that $\theta_m^*\theta_m=(2m+1)I$ and $\theta_m^*\theta_{m'}=0$. Therefore, for every $n\geq 0$, the operator
\begin{equation}\label{eq:averaging12}
W_n=c_n^{-1}\sum_{m=n}^{2n}\theta_m,
\end{equation}
where $c_n=((n+1)(3n+1))^{1/2}$, is an isometry in $\mathcal{B}(\mathscr{H},\mathscr{H}\otimes \mathscr{H}).$ Remark \ref{rem:holdercoef} suggests slowing down the sequence $(W_n)_{n\geq 0}$ by considering the sequence $(\gamma_n)_{n\geq 0}$ with $\gamma_n=\lceil n/16 \rceil$, to define 
\begin{equation}\label{eq:averaging13}
V_n=W_{\gamma_n}.
\end{equation}
Moreover, for $n<0$, we define $V_n=V_{-n}$ and obtain the desired isometry
\begin{equation}\label{eq:averaging14}
V:=\bigoplus_{n\in \mathbb Z} V_n: \mathscr{H}\otimes \ell^2(\mathbb Z)\to \mathscr{H}\otimes \mathscr{H}\otimes \ell^2(\mathbb Z).
\end{equation}

\begin{remark}\label{rem:averaging_Markov_par}
Let $m\in \mathbb N$ and $|r|\leq m-1$. From Theorem \ref{thm:theoremgraphSmalespaces} recall that the open cover $\varphi^r(\mathcal{R}_m^{\delta})$, corresponding to the basic isometry $\iota_{m,r}$, refines the $\delta$-enlarged Markov partition $\mathcal{R}_{m-|r|}^{\delta}$. This means that the open covers associated with the basic isometries constituting each $W_{n}$ (\ref{eq:averaging12}) refine the $\delta$-enlarged Markov partition $\mathcal{R}_{n}^{\delta}$. Consequently, as $n$ goes to $+ \infty$, the diameter of the supports of the coefficient functions (square roots of partitions of unity) of the basic isometries in $W_n$ tends to zero. The same holds for the isometry $V_n=W_{\gamma_{|n|}}$, as $|n|$ tends to infinity. 
\end{remark}

\begin{prop}\label{prop:quasi-invariance}
For every $j\in \mathbb Z$, the isometry $V:\mathscr{H}\otimes \ell^2(\mathbb Z)\to \mathscr{H}\otimes \mathscr{H}\otimes \ell^2(\mathbb Z)$ satisfies 
\begin{align*}
&\lim_{n\to \pm \infty} \|V_{n+j}-V_n\|=0,\\
&\lim_{n\to \pm \infty} \|(u\otimes u)^jV_{n}-V_nu^j\|=0.
\end{align*}
\end{prop}

\begin{proof}
It suffices to consider the limits to $+\infty$ and the isometries $W_n$ instead of $V_n$. Also, we can assume that $j\geq 0$. For the first limit we estimate $\|W_{n+j}-W_n\|^2$. Specifically, it suffices to show that for all $\eta \in \mathscr{H}$ with $\|\eta\|=1$, the inner products $\langle W_{n+j}(\eta), W_n(\eta)\rangle$ and $\langle W_{n}(\eta), W_{n+j}(\eta)\rangle$ are independent of $\eta$ and converge to one. Indeed, for $n\geq j$ we have that 
\begin{align*}
W_{n+j}^* W_n &= (c_{n+j}c_n)^{-1}(\sum_{m=n+j}^{2(n+j)}\theta_{m}^*)(\sum_{m'=n}^{2n}\theta_{m'})\\
&= (c_{n+j}c_n)^{-1}\sum_{m=n+j}^{2n}(2m+1)I\\
&= \frac{(n+1-j)(3n+1+j)}{((n+1+j)(3n+1+3j)(n+1)(3n+1))^{1/2}}I.
\end{align*}
Consequently, if we denote the latter multiple of the identity operator $I$ by $a_{n,j}$, since $W_n$ and $W_{n+j}$ are isometries, and $\|\eta\|=1$, we obtain that $$\|W_{n+j}-W_n\|=(2-2a_{n,j})^{1/2}=O(n^{-1}).$$

Similarly, we estimate $\|(u\otimes u)^jW_{n}-W_nu^j\|^2$ by showing that $W_n^*(u\otimes u)^jW_{n}$ is a multiple of $u^j$. Specifically, for $n\geq j$ we have
\begin{align*}
(u\otimes u)^jW_{n}&= c_n^{-1}\sum_{m=n}^{2n}(u\otimes u)^j\sum_{r=-m}^{m} \iota_{2m,r}\\
&= c_n^{-1}\sum_{m=n}^{2n}\sum_{r=-m+j}^{m+j} \iota_{2m,r}u^j.
\end{align*} 
Since the basic isometries have mutually orthogonal ranges we have
\begin{align*}
W_n^*(u\otimes u)^jW_{n}&=c_n^{-2}\sum_{m=n}^{2n}\sum_{r=-m+j}^{m}\iota_{2m,r}^*\iota_{2m,r}u^j\\
&= \frac{3n+1-j}{3n+1}u^j.
\end{align*}
As a result, we have $\|(u\otimes u)^jW_{n}-W_nu^j\|=O(n^{-1/2}).$ \qedhere
\end{proof}

\section{Lifting the fundamental class: a Fredholm module representative}\label{sec:DilateKPW}

Theorem \ref{thm:KK_1_lift_Ruelle} below proves the main result of the paper. We use the constructions of Sections \ref{sec:Towards_a_KK_1_lift} and \ref{sec:averaging_isom} to build a $\theta$-summable Fredholm module representative of the fundamental class $\Delta$ of Definition \ref{def:KPWextclass}. The proof is obtained through four rather technical lemmas. We remind the reader that Subsection \ref{sec:intuition_isometry} provides an overview of the proof and gives intuition on our constructions. We will refer back to it often.

Let $(X,\varphi)$ be an irreducible Smale space. Fix two periodic orbits $P$ and $Q$ such that $P\cap Q=\varnothing$ and consider the stable and unstable groupoids $G^s(Q)$ and $G^u(P)$. Without loss of generality assume its metric $d$ is self-similar, owing to Theorem \ref{lem:Artiguelemma} and the straightforward fact that the stable and unstable groupoids are invariants of topological conjugacy.  In \cite{Gero2}, the first author studies the smooth structure and $\Kt$-homology of the Ruelle algebras associated with these groupoids. In particular, \cite[Theorem 6.4]{Gero2} proves that the self-similar metric $d$ yields well-behaved compatible groupoid metrics $D_{s,d}$ and $D_{u,d}$ on $G^s(Q)$ and $G^u(P)$. This gives the following proposition.

\begin{prop}[{\cite[Proposition 7.1]{Gero2}}]\label{prop:Lip_alg_Smale_groupoids}
The complex vector space of compactly supported Lipschitz functions $\Lip_c(G^s(Q),D_{s,d})$ forms a dense $*$-subalgebra of $\mathcal{S}(Q)$. Moreover, it is $\alpha_s$-invariant, and therefore, the algebraic crossed product 
\[
\Lambda_{s,d}(Q,\alpha_s)=  \Lip_c(G^s(Q),D_{s,d})\rtimes_{\alpha_s,\text{alg}} \mathbb Z
\] 
is a dense $*$-subalgebra of the Ruelle algebra $\mathcal{R}^s(Q)$. Similarly, in the unstable case
\[
\Lambda_{u,d}(P,\alpha_u)=  \Lip_c(G^u(P),D_{u,d})\rtimes_{\alpha_u,\text{alg}} \mathbb Z
\]
is dense $*$-subalgebra of $\mathcal{R}^u(P)$.
\end{prop}

In \cite[Theorem 7.15]{Gero2}, it is proved that the extension $\tau_{\Delta}:\mathcal{R}^s(Q)\otimes \mathcal{R}^u(P)\to \mathcal{Q}(\mathscr{H}\otimes \ell^2(\mathbb Z))$ from \cite{KPW} is finitely-smooth on the algebra $\Lambda_{s,d}(Q,\alpha_s)\otimes_{\text{alg}}\Lambda_{u,d}(P,\alpha_u)$.

One of the technical parts of Section \ref{sec:Towards_a_KK_1_lift} was constructing the symmetric representation (\ref{eq:symmetric_rep}):
\[
\rho:\mathcal{R}^s(Q)\otimes \mathcal{R}^u(P)\to \mathcal{B}(\mathscr{H}\otimes \mathscr{H}\otimes \ell^2(\mathbb Z)).
\]
Then, using a refining sequence of $\delta$-enlarged Markov partitions $(\mathcal{R}_n^{\delta})_{n\geq 0}$, we constructed the isometry (\ref{eq:averaging14}) 
\[
V=\bigoplus_{n\in \mathbb Z} V_n:\mathscr{H}\otimes \ell^2(\mathbb Z)\to  \mathscr{H}\otimes \mathscr{H}\otimes \ell^2(\mathbb Z).
\]
Each isometry $V_n$ is the average of certain basic isometries $\iota_{m,s}:\mathscr{H}\to \mathscr{H}\otimes \mathscr{H}$ (\ref{eq:averaging5}). We now state the main result of the paper.

\begin{thm}\label{thm:KK_1_lift_Ruelle}
The triple $
(\mathscr{H}\otimes \mathscr{H}\otimes \ell^2(\mathbb Z), \rho, 2V V^*-1)$ is an odd Fredholm module over $\mathcal{R}^s(Q)\otimes \mathcal{R}^u(P)$ representing the fundamental class $\Delta \in \KKt_1(\mathcal{R}^s(Q)\otimes \mathcal{R}^u(P),\mathbb C)$, and is $\theta$-summable on the dense $*$-subalgebra $\Lambda_{s,d}(Q,\alpha_s)\otimes_{\text{alg}}\Lambda_{u,d}(P,\alpha_u)$.
\end{thm}

As noted above, we will prove Theorem \ref{thm:KK_1_lift_Ruelle} with a sequence of technical lemmas. The following establishes a refined version of \eqref{eq:intuition_isom_5}. Note that the self-similarity of the metric $d$ is not used.

\begin{lemma}\label{lem:KK_1_lift_ranks}
Let $a\in C_c(G^s(Q)),\, b\in C_c(G^u(P))$ and $i,k,l\in \mathbb Z$. Then, for every $n\in \mathbb Z$, the operator $$T_n=V_{2n+l}^*(\alpha_u^{-n}(b) u^{i} \otimes \alpha_s^n(a) u^{i}) V_{2n+k} -\alpha_u^{-n}(b)\alpha_s^n(a)u^{i}$$ has finite rank. In particular,
\begin{enumerate}[(i)]
\item there is $n_0\in \mathbb N$ such that $\rank (T_n)=0$, for all $n\leq -n_0$;
\item there is a constant $C_1>0$ so that, for every $\varepsilon >0$ there is $n_1\in \mathbb N$ such that, for all $n\geq n_1$, we have $\rank (T_n)\leq C_1 e^{2(\ent(\varphi)+\varepsilon)n}.$
\end{enumerate}
\end{lemma}

\begin{proof}
Assume that $\supp(a)\subset V^s(v,w,h^s,\eta,N)$ and $\supp(b)\subset V^u(v',w',h^u,\eta',N')$ and let $S_n= V_{2n+l}^*(\alpha_u^{-n}(b) u^{i} \otimes \alpha_s^n(a) u^{i}) V_{2n+k}$. Fix some $n\in \mathbb Z$. The operator $\alpha_u^{-n}(b)\alpha_s^n(a)u^{i}$ has finite rank (see the proof of \cite[Lemma 6.3]{KPW}) and also
\begin{equation}\label{eq:KK_1_lift_ranks_1}
\rank(\alpha_u^{-n}(b)\alpha_s^n(a)u^{i})=\rank(\alpha_u^{-n}(b)\alpha_s^n(a))=\rank(b\alpha_s^{2n}(a)).
\end{equation}
We claim that the operator $S_n$ has finite rank too. If $y\in X^h(P,Q)$ and $S_n(\delta_y)\neq 0$, then there is some basic isometry $\iota_{2m,r}$, for $m\in \{\gamma_{|2n+k|},\ldots, 2\gamma_{|2n+k|}\}$ and $|r|\leq m$, such that $(\alpha_u^{-n}(b) u^{i} \otimes \alpha_s^n(a) u^{i})\iota_{2m,r}(\delta_y)\neq 0$. We have 
\[
(u^i\otimes u^i)\iota_{2m,r}(\delta_y)=\sum_{j=1}^{\# \mathcal{R}_{2m}^\delta}f_{2m,j}(\varphi^{-r}(y))\delta_{\varphi^i[y,\varphi^r(g_{2m,j})]}\otimes \delta_{\varphi^i[\varphi^r(g_{2m,j}),y]},
\]
and hence there is some $j\in \{1,\ldots, \# \mathcal{R}_{2m}^\delta\}$ such that $f_{2m,j}(\varphi^{-r}(y))\neq 0$, meaning that $\varphi^i[y,\varphi^r(g_{2m,j})],\,\varphi^i[\varphi^r(g_{2m,j}),y]$ are well-defined with
\[
\alpha_u^{-n}(b)\delta_{\varphi^i[y,\varphi^r(g_{2m,j})]} \neq 0, \quad  \alpha_s^n(a)\delta_{\varphi^i[\varphi^r(g_{2m,j}),y]}\neq 0.
\]
Consequently, we have
\begin{equation}\label{eq:KK_1_lift_ranks_2}
\varphi^{n+i}[y,\varphi^r(g_{2m,j})] \in X^s(w',\eta'), \quad \varphi^{i-n}[\varphi^r(g_{2m,j}),y] \in X^u(w,\eta).
\end{equation}
Let $\ell \in \mathbb N$ be large enough so that $\ell+n+i>0$ and $i-n-\ell <0$. Then, 
$$
\varphi^{\ell+n+i}(y)\in X^s(\varphi^{\ell+n+i}[y,\varphi^r(g_{2m,j})],\lambda^{-(\ell+n+i)}\varepsilon_X/2)$$ and 
$$\varphi^{\ell+n+i}[y,\varphi^r(g_{2m,j})]\in X^s(\varphi^{\ell}(w'),\lambda^{-\ell}\eta').$$ Therefore, $\varphi^{\ell+n+i}(y)\in X^s(\varphi^{\ell}(w'),\varepsilon_X)$, or equivalently 
\begin{equation}\label{eq:KK_1_lift_ranks_3}
y\in \varphi^{-(\ell+n+i)}(X^s(\varphi^{\ell}(w'),\varepsilon_X)).
\end{equation}
Similarly, we have that 
\begin{equation}\label{eq:KK_1_lift_ranks_4}
y\in \varphi^{-(i-n-\ell)}(X^u(\varphi^{-\ell}(w),\varepsilon_X)).
\end{equation}
The intersection of the sets in (\ref{eq:KK_1_lift_ranks_3}) and (\ref{eq:KK_1_lift_ranks_4}) is finite since the first is a compact segment of the global stable set $X^s(\varphi^{-(n+i)}(w'))$, and the second is a compact segment of the global unstable set $X^u(\varphi^{-(i-n)}(w)).$ As a result, we obtain that the operator $S_n$ has finite rank.

We now prove part (i). From (\ref{eq:KK_1_lift_ranks_1}) and again using the proof of \cite[Lemma 6.3]{KPW} we can find $n_0\in \mathbb N$ so that $\rank(\alpha_u^{-n}(b)\alpha_s^n(a)u^{i})=0$, if $n\leq - n_0$. We aim to choose a (possibly) larger $n_0$ so that we also have $\rank(S_n)=0$, if $n\leq - n_0$. First note that since $P\cap Q=\varnothing$, there is some $\varepsilon_0>0$ so that $B(Q,\varepsilon_0)\cap B(P,\varepsilon_0)=\varnothing$. Also, note that $w\in X^u(Q),\, w'\in X^s(P)$, and for every $m\in \{\gamma_{|2n+k|},\ldots, 2\gamma_{|2n+k|}\}$ and $|r|\leq m$ it holds that \[
\text{diam}(\varphi^r(\mathcal{R}_{2m}^{\delta}))\leq \lambda^{1-\gamma_{|2n+k|}}\varepsilon_X,\]
see Remark \ref{rem:averaging_Markov_par} and Theorem \ref{thm:theoremgraphSmalespaces}. It suffices to consider $n_0\in \mathbb N$ large enough so that, if $n\leq -n_0$, we also have that 
\begin{enumerate}[(i)]
\item $n+i<0, \, i-n >0$;
\item $\varphi^{-(n+i)}(w')\in B(P,\varepsilon_0/4),\, \lambda^{n+i}\eta' <\varepsilon_0/4$;
\item $\varphi^{-(i-n)}(w)\in B(Q,\varepsilon_0/4),\, \lambda^{-(i-n)}\eta <\varepsilon_0/4$;
\item  $\lambda^{1-\gamma_{|2n+k|}}\varepsilon_X<\varepsilon_0/2$.
\end{enumerate}

Suppose that $n\leq -n_0$ and assume to the contrary that there is $y\in X^h(P,Q)$ with $S_n(\delta_y)\neq 0$. Following the proof so far, we can find $m\in \{\gamma_{|2n+k|},\ldots, 2\gamma_{|2n+k|}\},\, |r|\leq m$ and $j\in \{1,\ldots, \# \mathcal{R}_{2m}^\delta\}$ so that $f_{2m,j}(\varphi^{-r}(y))\neq 0$ and (\ref{eq:KK_1_lift_ranks_2}) holds. We have that
\begin{align*}
[y,\varphi^r(g_{2m,j})] &\in X^s(\varphi^{-(n+i)}(w'),\lambda^{n+i}\eta'),\\
[\varphi^r(g_{2m,j}),y] &\in X^u(\varphi^{-(i-n)}(w),\lambda^{-(i-n)}\eta).
\end{align*}
As a result, we get that $[y,\varphi^r(g_{2m,j})]\in B(P,\varepsilon_0/2)$ and $[\varphi^r(g_{2m,j}),y]\in B(Q,\varepsilon_0/2).$ Also, since the coefficient function $f_{2m,j}\circ \varphi^{-r}$ is non-zero exactly on the rectangle $\varphi^r(R_{2m,j}^{\delta})$, we have that $y,\, [y,\varphi^r(g_{2m,j})],\, [\varphi^r(g_{2m,j}),y]\in \varphi^r(R_{2m,j}^{\delta})$ and hence their mutual distances are less than $\varepsilon_0/2$. Therefore, we obtain that $y\in B(Q,\varepsilon_0)\cap B(P,\varepsilon_0)$, which is a contradiction.

To prove part (ii), let $\varepsilon >0$. Again, from (\ref{eq:KK_1_lift_ranks_1}) and this time \cite[Lemma 7.6]{Gero2}, there is some $n_1\in \mathbb N$ such that, if $n\geq n_1$, then we have that 
\[
\rank(\alpha_u^{-n}(b)\alpha_s^n(a)u^{i})< e^{2(\ent(\varphi)+\varepsilon)n}.
\]
We now prove that a similar statement holds for $S_n$, given that $n\in \mathbb N$ is sufficiently large. First, consider a (possibly) larger $n_1$ so that $n_1\geq |i|$ and assume that $n\geq n_1$. If $y\in X^h(P,Q)$ and $S_n(\delta_y)\neq 0$, similarly as before, we can find $m\in \{\gamma_{|2n+k|},\ldots, 2\gamma_{|2n+k|}\},\, |r|\leq m$ and $j\in \{1,\ldots, \# \mathcal{R}_{2m}^{\delta}\}$ so that $f_{2m,j}(\varphi^{-r}(y))\neq 0$ and (\ref{eq:KK_1_lift_ranks_2}) holds. However, this time we have that 
\begin{align*}
\varphi^{n+i}(y)&\in X^s(\varphi^{n+i}[y,\varphi^r(g_{2m,j})],\lambda^{-(n+i)}\varepsilon_X/2),\\
\varphi^{i-n}(y)&\in X^u(\varphi^{i-n}[\varphi^r(g_{2m,j}),y],\lambda^{i-n}\varepsilon_X/2),
\end{align*}
and hence we have 
\[
y\in \varphi^{-(n+i)}(X^s(w',\varepsilon_X))\cap \varphi^{n-i}(X^u(w,\varepsilon_X)).
\]
Working as in \cite[Lemma 7.6]{Gero2}, we can choose $n_1$ even larger so that, if $n\geq n_1$, it holds
\begin{align*}
\rank(S_n)&\leq \# \varphi^{-(n+i)}(X^s(w',\varepsilon_X))\cap \varphi^{n-i}(X^u(w,\varepsilon_X))\\ 
&=\# X^s(w',\varepsilon_X)\cap \varphi^{2n}(X^u(w,\varepsilon_X))\\
&< e^{2(\ent(\varphi)+\varepsilon)n}.
\end{align*}
This completes the proof of the lemma.
\end{proof}

For the next \enquote{orthogonality} lemma, recall from the proof of Proposition \ref{prop:aperiodic_set} that $\mathcal{G}=\{g_{n,k}\in R_{n,k}^{\delta}:n\geq 0, \, 1\leq k\leq \# \mathcal{R}_n^\delta\}$ is an aperiodic $\mathcal{R}^{\delta}$-sample, see Definition \ref{def:aperiodic_sampling}. Also, note that the diameters in a $\delta$-enlarged Markov partition $\mathcal{R}_1^{\delta}$ are assumed to be less than some $\varepsilon ' < \varepsilon_X'$. 

\begin{lemma}\label{lem:KK_1_lift_orthogonality}
Let $a\in C_c(G^s(Q)),\, b\in C_c(G^u(P))$ and $i\in \mathbb Z$. Also, let $m,t\in \mathbb N$ and $r,s\in \mathbb Z$ with $|r+i|\leq m-1,\, |s|\leq t-1$ and $(t,s)\neq (m,r+i)$. Then, there is $n_2=n_2(a,b)\in \mathbb N$ so that, for all $n\geq n_2$, it holds $$\iota_{t,s}^*(\alpha_u^{-n}(b)u^i\otimes \alpha_s^n(a)u^i)\iota_{m,r}=0.$$
\end{lemma}

\begin{proof}
Assume that $\supp(a)\subset V^s(v,w,h^s,\eta,N)$ and $\supp(b)\subset V^u(v',w',h^u,\eta',N')$. First, we should note that if $n\geq N,N'$ then, for every $x\in \varphi^n(X^u(w,\eta))$ we have 
\[
\varphi^n\circ h^s\circ \varphi^{-n}(x)\in X^s(x,\lambda^{-n+N}\varepsilon_X/2),\]
and similarly, for every $x\in \varphi^{-n}(X^s(w',\eta'))$ we have
\[
\varphi^{-n}\circ h^u\circ \varphi^n(x)\in X^u(x,\lambda^{-n+N'}\varepsilon_X/2).
\]
We let $n_2\in \mathbb N$ to be larger than $N,N'$ so that $\lambda^{-n_2+N}\varepsilon_X/2,\, \lambda^{-n_2+N'}\varepsilon_X/2\leq \varepsilon_X-2\varepsilon'$.

Now suppose $n\geq n_2$ and assume to the contrary that there is $y\in X^h(P,Q)$ such that 
\begin{equation}\label{eq:KK_1_lift_orthogonality_1}
\iota_{t,s}^*(\alpha_u^{-n}(b)u^i\otimes \alpha_s^n(a)u^i)\iota_{m,r}(\delta_y)\neq 0.
\end{equation}
Since $|r+i|\leq m-1$, we have
\[
(u^i\otimes u^i)\iota_{m,r}(\delta_y)=\sum_{j=1}^{\# \mathcal{R}_{m}^\delta}f_{m,j}(\varphi^{-r}(y))\delta_{[\varphi^i(y),\varphi^{r+i}(g_{m,j})]}\otimes \delta_{[\varphi^{r+i}(g_{m,j}),\varphi^i(y)]}.
\]
Then, from (\ref{eq:KK_1_lift_orthogonality_1}) there is some $j\in \{1,\ldots, \#\mathcal{R}_m^\delta\}$ so that 
\begin{equation}\label{eq:KK_1_lift_orthogonality_2}
f_{m,j}(\varphi^{-r}(y))\iota_{t,s}^*(\alpha_u^{-n}(b)\delta_{[\varphi^i(y),\varphi^{r+i}(g_{m,j})]}\otimes \alpha_s^n(a)\delta_{[\varphi^{r+i}(g_{m,j}),\varphi^i(y)]})\neq 0.
\end{equation}
In particular, $\varphi^{-n}\circ h^u\circ \varphi^n[\varphi^i(y),\varphi^{r+i}(g_{m,j})]$ and $\varphi^n\circ h^s\circ \varphi^{-n}[\varphi^{r+i}(g_{m,j}),\varphi^i(y)]$ are well-defined, and there is a (unique) $k\in \{1,\ldots, \#\mathcal{R}_t^\delta\}$ so that  
\begin{align*}
\varphi^{-s-n}\circ h^u\circ \varphi^n[\varphi^i(y),\varphi^{r+i}(g_{m,j})]&\in X^u(g_{t,k},R_{t,k}^{\delta}),\\
\varphi^{n-s}\circ h^s\circ \varphi^{-n}[\varphi^{r+i}(g_{m,j}),\varphi^i(y)]&\in X^s(g_{t,k},R_{t,k}^{\delta}).
\end{align*}
Since $|s|\leq t-1$, it follows that 
\begin{align*}
\varphi^{-n}\circ h^u\circ \varphi^n[\varphi^i(y),\varphi^{r+i}(g_{m,j})]&\in X^u(\varphi^s(g_{t,k}),\varepsilon'),\\
\varphi^{n}\circ h^s\circ \varphi^{-n}[\varphi^{r+i}(g_{m,j}),\varphi^i(y)]&\in X^s(\varphi^s(g_{t,k}),\varepsilon').
\end{align*}
Moreover, 
\begin{align*}
\varphi^{-n}\circ h^u\circ \varphi^n[\varphi^i(y),\varphi^{r+i}(g_{m,j})]&\in X^u([\varphi^i(y),\varphi^{r+i}(g_{m,j})],\varepsilon_X-2\varepsilon'),\\
\varphi^{n}\circ h^s\circ \varphi^{-n}[\varphi^{r+i}(g_{m,j}),\varphi^i(y)]&\in X^s([\varphi^{r+i}(g_{m,j}),\varphi^i(y)],\varepsilon_X-2\varepsilon'),
\end{align*}
and since $\varphi^i(y),\varphi^{r+i}(g_{m,j})\in \varphi^{r+i}(R_{m,j}^{\delta})$ and $\varphi^{r+i}(R_{m,j}^{\delta})$ is a rectangle with diameter less than $\varepsilon'$, we also have that 
\begin{align*}
[\varphi^i(y),\varphi^{r+i}(g_{m,j})]&\in X^u(\varphi^{r+i}(g_{m,j}),\varepsilon'),\\
[\varphi^{r+i}(g_{m,j}),\varphi^i(y)]&\in X^s(\varphi^{r+i}(g_{m,j}),\varepsilon').
\end{align*}

Hence, we have $\varphi^{r+i}(g_{m,j})\in X^s(\varphi^s(g_{t,k}),\varepsilon_X)\cap X^u(\varphi^s(g_{t,k}),\varepsilon_X)=\{\varphi^s(g_{t,k})\}.$ If $s\neq r+i$ then $\varphi^{r+i-s}(g_{m,j})=g_{t,k}$, which is a contradiction since $\mathcal{G}$ is aperiodic. Finally, if $s=r+i$ then $g_{m,j}=g_{t,k}$, and from aperiodicity it holds that $m=t$. However, this is again a contradiction since $(t,s)\neq (m,r+i)$.
\end{proof}

For the next lemma, the fact that the metric $d$ is self-similar plays an important role. Further, recall the sequence $\gamma_{|n|}=\lceil |n|/16 \rceil$ used to define $V_n$ in \eqref{eq:averaging13}.

\begin{lemma}\label{lem:KK_1_lift_convergence}
Let $a\in \Lip_c(G^s(Q),D_{s,d}),\, b\in \Lip_c(G^u(P),D_{u,d})$ and $i,k\in \mathbb Z$. Then, there is a constant $C_2>0$ and some $n_3\in \mathbb N$ so that $2n_3+k>0$ and $\gamma_{2n_3+k}\geq |i|+1$ and, for all $n\geq n_3$, if $m\in \{\gamma_{2n+k},\ldots, 2\gamma_{2n+k}\},\, |r|\leq m$, it holds that $$\|\iota_{2m,r+i}^*(\alpha_u^{-n}(b)u^i\otimes \alpha_s^n(a)u^i)\iota_{2m,r}-\alpha_u^{-n}(b)\alpha_s^n(a)u^i\|\leq C_2(\lambda^{-n/8}+2^{-n/(8\lceil \log_{\lambda}3 \rceil)}).$$
\end{lemma}

\begin{proof}
Assume that $\supp(a)\subset V^s(v,w,h^s,\eta,N)$ and $\supp(b)\subset V^u(v',w',h^u,\eta',N')$. First, choose $n_3\in \mathbb N$ so that $2n_3+k>0$ and $\gamma_{2n_3+k}\geq |i|+1$. Since $(\gamma_{n})_{n\in \mathbb N}$ is increasing we have that $\gamma_{2n+k}\geq |i|+1$, for all $n\geq n_3.$ Therefore, for all $m,r$ as in the statement it holds that $|r+i|\leq 2m-1$. Along the proof, the integer $n_3$ will be chosen sufficiently large for our arguments to work. Also, the constant $C_2$ will be described at the end of the proof. Moreover, for a large portion of the proof, the fact that we consider $m\in \mathbb N$ that increases with respect to $n\in \mathbb N$; meaning $m\geq \gamma_{2n+k}$, is not important, and what one should keep in mind is that $m,r$ satisfy $|r+i|\leq 2m-1$. 

Let $n\geq n_3$ and $m,r$ as in the statement. First note that if $\alpha_u^{-n}(b)\alpha_s^n(a)u^i\neq 0$, then $X^s(\varphi^{-n}(w'))\cap X^u (\varphi^n(v))\neq \varnothing$ and hence $\varphi^{-n}(w'),\,\varphi^n(v)$ must lie in the same mixing component. Therefore, in this case, all points $\varphi^{-n}(w'),\, \varphi^{-n}(v'),\, \varphi^n(w),\, \varphi^n(v)$ must lie in the same mixing component. Similarly, if $(\alpha_u^{-n}(b)u^i\otimes \alpha_s^n(a)u^i)\iota_{2m,r}\neq 0$, then from (\ref{eq:KK_1_lift_ranks_3}) and (\ref{eq:KK_1_lift_ranks_4}) of Lemma \ref{lem:KK_1_lift_ranks} we see that $X^s(\varphi^{-n}(w'))\cap X^u (\varphi^n(w))\neq \varnothing$, and hence $\varphi^{-n}(w'),\, \varphi^{-n}(v'),\, \varphi^n(w),\, \varphi^n(v)$ are in the same mixing component. As a result, the only possibility for the operators $\iota_{2m,r+i}^*(\alpha_u^{-n}(b)u^i\otimes \alpha_s^n(a)u^i)\iota_{2m,r}$ and $\alpha_u^{-n}(b)\alpha_s^n(a)u^i$ to be non-zero is when $n$ takes values in a certain strictly increasing arithmetic-like sequence. Therefore, without loss of generality, we can assume that $(X,\varphi)$ is mixing. 

Recall that $\text{source}(a)$ and $\text{source}(b)$ denote the images of $\supp(a),\, \supp(b)$ via the groupoid source map. Both $\text{source}(a)$ and $\text{source}(b)$ are compact subsets of $X^u(w,\eta)$ and $X^s(w',\eta')$, respectively. From compactness, we can find an $\varepsilon>0$ such that the $\varepsilon$-neighbourhoods $N_{\varepsilon}(\text{source}(a))\subset X^u(w,\eta)$ and $N_{\varepsilon}(\text{source}(b))\subset X^s(w',\eta')$. Let us now choose $n_3\in \mathbb N$ to be sufficiently large so that $$\lambda^{i-n_3}\varepsilon_X,\, \lambda^{-i-n_3}\varepsilon_X,\, \lambda^{-n_3+N}\varepsilon_X <\varepsilon.$$ In particular, this also means that $n_3\geq N,\, |i|.$ Consider $y\in X^h(P,Q)$ and $n\geq n_3$ with $\alpha_u^{-n}(b)\alpha_s^n(a)u^i \delta_y \neq 0.$ Then $$\varphi^{-n+i}(y)\in \text{source}(a), \quad \varphi^{2n}\circ h^s \circ \varphi^{-n+i}(y)\in \text{source}(b),$$ and $\alpha_u^{-n}(b)\alpha_s^n(a)u^i\delta_y$ equals
$$b(h^u\circ \varphi^{2n}\circ h^s \circ \varphi^{-n+i}(y), \varphi^{2n}\circ h^s \circ \varphi^{-n+i}(y)) a( h^s \circ \varphi^{-n+i}(y), \varphi^{-n+i}(y)) \delta_{\varphi^{-n}\circ h^u\circ \varphi^{2n}\circ h^s \circ \varphi^{-n+i}(y)}.$$ The goal is to find a similar description for $(\alpha_u^{-n}(b) u^{i} \otimes \alpha_s^n(a) u^{i})\iota_{2m,r}(\delta_y)$ which is written as 
\[
\sum_{j=1}^{\# \mathcal{R}_{2m}^{\delta}}f_{2m,j}(\varphi^{-r}(y))\alpha_u^{-n}(b) \delta_{\varphi^i[y,\varphi^r(g_{2m,j})]}\otimes \alpha_s^n(a) \delta_{\varphi^i[\varphi^r(g_{2m,j}),y]}.
\]
Indeed, for every $j\in \{1,\ldots, \# \mathcal{R}_{2m}^\delta\}$ with $f_{2m,j}(\varphi^{-r}(y))\neq 0$ we have that 
\[
\varphi^{-n+i}[\varphi^r(g_{2m,j}),y]\in X^u(\varphi^{-n+i}(y),\lambda^{-n+i}\varepsilon_X/2)
\subset  N_{\varepsilon/2}(\text{source}(a))
\subset X^u(w,\eta).
\]
As a result, $\alpha_s^n(a) \delta_{\varphi^i[\varphi^r(g_{2m,j}),y]}$ equals
$$ a(h^s\circ \varphi^{-n+i}[\varphi^r(g_{2m,j}),y],\varphi^{-n+i}[\varphi^r(g_{2m,j}),y])\delta_{\varphi^n\circ h^s\circ \varphi^{-n+i}[\varphi^r(g_{2m,j}),y]}.$$
Moreover, for such $j$ it holds that
\begin{align*}
\varphi^{n+i}[y,\varphi^r(g_{2m,j})]&\in X^s(\varphi^{n+i}(y),\lambda^{-(n+i)}\varepsilon_X/2)\\
&\subset X^s(\varphi^{2n}\circ h^s \circ \varphi^{-n+i}(y),\lambda^{-(n+i)}\varepsilon_X/2+\lambda^{-2n+N}\varepsilon_X/2)\\
&\subset  N_{\varepsilon}(\text{source}(b)) \subset X^u(w',\eta').
\end{align*}
Therefore,
\[
\alpha_u^{-n}(b) \delta_{\varphi^i[y,\varphi^r(g_{2m,j})]}=b(h^u\circ \varphi^{n+i}[y,\varphi^r(g_{2m,j})],\varphi^{n+i}[y,\varphi^r(g_{2m,j})])\delta_{\varphi^{-n}\circ h^u\circ \varphi^{n+i}[y,\varphi^r(g_{2m,j})]}.
\]

Conversely, let $y\in X^h(P,Q)$ and $n\geq n_3$. If $(\alpha_u^{-n}(b) u^{i} \otimes \alpha_s^n(a) u^{i})\iota_{2m,r}(\delta_y)\neq 0$, then there is some $j_0\in \{1,\ldots, \# \mathcal{R}_{2m}^\delta\}$ so that $$f_{2m,j_0}(\varphi^{-r}(y))\alpha_u^{-n}(b) \delta_{\varphi^i[y,\varphi^r(g_{2m,j_0})]}\otimes \alpha_s^n(a) \delta_{\varphi^i[\varphi^r(g_{2m,j_0}),y]}\neq 0.$$ Specifically, the expressions $\varphi^i[y,\varphi^r(g_{2m,j_0})],\, \varphi^i[\varphi^r(g_{2m,j_0}),y]$ are well-defined. Also, $$\varphi^{-n+i}[\varphi^r(g_{2m,j_0}),y]\in \text{source}(a) \quad \text{ and } \quad \varphi^{n+i}[y,\varphi^r(g_{2m,j_0})]\in \text{source}(b).$$ Working as before, we obtain that $\varphi^{-n+i}(y)\in N_{\varepsilon/2}(\text{source}(a))\subset X^u(w,\eta)$ and also that $\varphi^{n+i}(y)\in N_{\varepsilon/2}(\text{source}(b)) \subset X^u(w',\eta')$. Therefore, the point $\varphi^{2n}\circ h^s \circ \varphi^{-n+i}(y)$ is well-defined and lies so close to $\varphi^{n+i}(y)$ (in the stable direction) so that it is actually in $N_{\varepsilon}(\text{source}(b)) \subset X^u(w',\eta')$. As a result, $\alpha_u^{-n}(b)\alpha_s^n(a)u^i \delta_y$ is given by $$b(h^u\circ \varphi^{2n}\circ h^s \circ \varphi^{-n+i}(y), \varphi^{2n}\circ h^s \circ \varphi^{-n+i}(y)) a( h^s \circ \varphi^{-n+i}(y), \varphi^{-n+i}(y)) \delta_{\varphi^{-n}\circ h^u\circ \varphi^{2n}\circ h^s \circ \varphi^{-n+i}(y)}.$$ Moreover, for every $j\in \{1,\ldots, \# \mathcal{R}_{2m}^\delta\}$ with $f_{2m,j}(\varphi^{-r}(y))\neq 0$, by comparing with $\varphi^{-n+i}(y)$, $\varphi^{n+i}(y)$, we have that 
\[
\varphi^{-n+i}[\varphi^r(g_{2m,j}),y]\in N_{\varepsilon}(\text{source}(a)) \quad \text{ and } \quad 
\varphi^{n+i}[y,\varphi^r(g_{2m,j})]\in N_{\varepsilon}(\text{source}(b)),
\]
and hence lie in $X^u(w,\eta)$ and $X^u(w',\eta')$, respectively.

To summarise, let $n\geq n_3$ and consider $y\in X^h(P,Q)$ so that $\alpha_u^{-n}(b)\alpha_s^n(a)u^i \delta_y \neq 0$ or $(\alpha_u^{-n}(b) u^{i} \otimes \alpha_s^n(a) u^{i})\iota_{2m,r}(\delta_y)\neq 0$. Then,  $\varphi^{-n}\circ h^u\circ \varphi^{2n}\circ h^s \circ \varphi^{-n+i}(y)$ is well-defined. Similarly, for $\varphi^n\circ h^s\circ \varphi^{-n+i}[\varphi^r(g_{2m,j}),y]$ and $\varphi^{-n}\circ h^u\circ \varphi^{n+i}[y,\varphi^r(g_{2m,j})]$, for all $j\in \{1,\ldots, \# \mathcal{R}_{2m}^\delta\}$ with $f_{2m,j}(\varphi^{-r}(y))\neq 0$. Our goal now is to evaluate the operator $\iota_{2m,r+i}^*(\alpha_u^{-n}(b)u^i\otimes \alpha_s^n(a)u^i)\iota_{2m,r}$ on $\delta_y$.

At this point we shall consider a (possibly) larger $n_3$ so that $n_3\geq N'$. Then for every $y$ as above we have that
\begin{equation}\label{eq:KK_1_lift_convergence_1}
d(\varphi^{-n}\circ h^u\circ \varphi^{2n}\circ h^s \circ \varphi^{-n+i}(y), \varphi^i(y))\leq K_0\lambda^{-n},
\end{equation}
where $K_0=
\lambda^{N'}\varepsilon_X/2+\lambda^{N}\varepsilon_X/2.$ Moreover, from \cite[Lemma 2.2]{Putnam_algebras}, for a (possibly) larger $n_3$ we can guarantee that 
\begin{equation}\label{eq:KK_1_lift_convergence_2}
\varphi^n\circ h^s\circ \varphi^{-2n}\circ h^u \circ \varphi^n (\varphi^i(y))= \varphi^{-n}\circ h^u\circ \varphi^{2n}\circ h^s \circ \varphi^{-n}(\varphi^i(y)).
\end{equation}
Consequently, we obtain the important fact that
\begin{equation}\label{eq:KK_1_lift_convergence_3}
\varphi^{-n}\circ h^u\circ \varphi^{2n}\circ h^s \circ \varphi^{-n+i}(y)=[\varphi^{-n}\circ h^u\circ \varphi^{n+i}(y), \varphi^n\circ h^s\circ \varphi^{-n+i}(y)].
\end{equation}
In addition, for every $j\in \{1,\ldots, \# \mathcal{R}_{2m}^\delta\}$ with $f_{2m,j}(\varphi^{-r}(y))\neq 0$, we claim that
\begin{equation}\label{eq:KK_1_lift_convergence_5}
\begin{split}
\varphi^n\circ h^s\circ \varphi^{-n+i}[\varphi^r(g_{2m,j}),y]&=[\varphi^{r+i}(g_{2m,j}),\varphi^n\circ h^s\circ \varphi^{-n+i}(y)],\\
\varphi^{-n}\circ h^u\circ \varphi^{n+i}[y,\varphi^r(g_{2m,j})]&=[\varphi^{-n}\circ h^u\circ \varphi^{n+i}(y),\varphi^{r+i}(g_{2m,j})].
\end{split}
\end{equation}
We now prove the stable case, and the unstable case is similar. We have that both $\varphi^{-r}(y),\, g_{2m,j}$ lie in $R_{2m,j}^{\delta}$, and since $|r+i|\leq 2m-1$ it holds that $$\varphi^i[\varphi^r(g_{2m,j}),y]=[\varphi^{r+i}(g_{2m,j}),\varphi^{i}(y)].$$ For brevity denote $x=\varphi^{r+i}(g_{2m,j}),\, z=\varphi^{i}(y).$ For every $0\leq \ell \leq n-N$ we have that
\begin{align*}
d(\varphi^{\ell+N-n}(z),\varphi^{\ell+N}\circ h^s \circ \varphi^{-n}(z))&\leq \lambda^{-\ell}d(\varphi^{N-n}(z),\varphi^{N}\circ h^s \circ \varphi^{-n}(z))\\
&= \lambda^{-\ell}d(\varphi^{N-n}(z),[\varphi^{N-n}(z),\varphi^N(v)])\\
&\leq \lambda^{-\ell}\varepsilon_X/2,
\end{align*}
and also that $d(\varphi^{\ell+N-n}(z),\varphi^{\ell+N-n}[x,z])\leq \lambda^{\ell+N-n}\varepsilon_X/2.$ Therefore, for all $0\leq \ell \leq n-N$ it holds that $$d(\varphi^{\ell+N-n}[x,z], \varphi^{\ell+N}\circ h^s \circ \varphi^{-n}(z))\leq \varepsilon_X,$$ and in addition 
\begin{align*}
[\varphi^{\ell+N-n}[x,z],\varphi^{\ell+N}\circ h^s \circ \varphi^{-n}(z)]&=\varphi^{\ell}[\varphi^{N-n}[x,z],[\varphi^{N-n}(z), \varphi^N(v)]]\\
&=\varphi^{\ell}[\varphi^{N-n}[x,z],\varphi^N(v)].
\end{align*}
For $\ell=n-N$ we obtain that 
\begin{align*}
[x,\varphi^n\circ h^s \circ \varphi^{-n}(z)]&=[[x,z], \varphi^n\circ h^s \circ \varphi^{-n}(z)]\\
&= \varphi^{n-N}[\varphi^{N-n}[x,z],\varphi^N(v)]\\
&= \varphi^n\circ h^s \circ \varphi^{-n}[x,z],
\end{align*}
and this completes the proof of the claim. In fact, we can assume that $n_3$ is slightly larger (and still independent of $y$) so that, 
\begin{equation}\label{eq:KK_1_lift_convergence_6}
\begin{split}
[\varphi^{-n}\circ h^u\circ \varphi^{n+i}(y),\varphi^{r+i}(g_{2m,j})]&\in X^u(\varphi^{r+i}(g_{2m,j}),\varepsilon_X/2),\\
[\varphi^{r+i}(g_{2m,j}),\varphi^n\circ h^s\circ \varphi^{-n+i}(y)]&\in X^s(\varphi^{r+i}(g_{2m,j}),\varepsilon_X/2). 
\end{split}
\end{equation}
This last condition allows us to use the adjoint $\iota_{2m,r+i}^*$ on $(\alpha_u^{-n}(b) u^{i} \otimes \alpha_s^n(a) u^{i})\iota_{2m,r}(\delta_y)$ which (from the proof so far) is a linear combination of the basis vectors $$\delta_{[\varphi^{-n}\circ h^u\circ \varphi^{n+i}(y),\varphi^{r+i}(g_{2m,j})]}\otimes \delta_{[\varphi^{r+i}(g_{2m,j}),\varphi^n\circ h^s\circ \varphi^{-n+i}(y)]},$$ where $j\in \{1,\ldots, \# \mathcal{R}_{2m}^\delta\}$ with $f_{2m,j}(\varphi^{-r}(y))\neq 0$. For these $j$ we have that the expression $$\iota_{2m,r+i}^*(\delta_{[\varphi^{-n}\circ h^u\circ \varphi^{n+i}(y),\varphi^{r+i}(g_{2m,j})]}\otimes \delta_{[\varphi^{r+i}(g_{2m,j}),\varphi^n\circ h^s\circ \varphi^{-n+i}(y)]})$$ is equal to $$f_{2m,j}(\varphi^{-(r+i)}[\varphi^{-n}\circ h^u\circ \varphi^{n+i}(y), \varphi^n\circ h^s\circ \varphi^{-n+i}(y)])\delta_{[\varphi^{-n}\circ h^u\circ \varphi^{n+i}(y), \varphi^n\circ h^s\circ \varphi^{-n+i}(y)]},$$ which is equivalently written as $$f_{2m,j}(\varphi^{-(r+i)}\circ \varphi^{-n}\circ h^u\circ \varphi^{2n}\circ h^s \circ \varphi^{-n+i}(y))\delta_{\varphi^{-n}\circ h^u\circ \varphi^{2n}\circ h^s \circ \varphi^{-n+i}(y)}.$$ 

A remark is in order. In general, the points $\varphi^{-n}\circ h^u\circ \varphi^{n+i}(y),\, \varphi^n\circ h^s\circ \varphi^{-n+i}(y)$ might not be in the same rectangle with $\varphi^{r+i}(g_{2m,j})$, namely the rectangle $\varphi^{r+i}(R_{2m,j}^{\delta})$. In turn, this might lead to $$[\varphi^{-n}\circ h^u\circ \varphi^{n+i}(y), \varphi^n\circ h^s\circ \varphi^{-n+i}(y)]\not \in \varphi^{r+i}(R_{2m,j}^{\delta}),$$ and since $f_{2m,j}(\varphi^{-(r+i)}(x))\neq 0$ if and only if $x\in \varphi^{r+i}(R_{2m,j}^{\delta})$, we will have that $$\iota_{2m,r+i}^*(\delta_{[\varphi^{-n}\circ h^u\circ \varphi^{n+i}(y),\varphi^{r+i}(g_{2m,j})]}\otimes \delta_{[\varphi^{r+i}(g_{2m,j}),\varphi^n\circ h^s\circ \varphi^{-n+i}(y)]})=0.$$ However, this is exactly what the slow-down sequence $(\gamma_{n})_{n\in \mathbb N}$ fixes; it forces the index $2m$ to go slower to infinity than $n$. In this way, the refining process of $(\mathcal{R}_n^{\delta})_{n\in \mathbb N}$ has a lag so that both $\varphi^{-n}\circ h^u\circ \varphi^{n+i}(y),\, \varphi^n\circ h^s\circ \varphi^{-n+i}(y)$ manage to converge to $\varphi^i(y)$ fast enough so that they tend to lie in the same rectangle(s) with $\varphi^i(y)$, and hence (eventually) with $\varphi^{r+i}(g_{2m,j})$. This can be achieved by using the Lebesgue covering numbers of the $\delta$-enlarged Markov partitions which can be explicitly described using the self-similarity of the metric $d$, see Theorem \ref{thm:theoremgraphSmalespaces}.

At this point let us summarise what we have proved so far. Let $n\geq n_3$ and $y\in X^h(P,Q)$ so that $\alpha_u^{-n}(b)\alpha_s^n(a)u^i \delta_y \neq 0$ or $(\alpha_u^{-n}(b) u^{i} \otimes \alpha_s^n(a) u^{i})\iota_{2m,r}(\delta_y)\neq 0$. Then, the following expressions are well-defined 
\begin{align*}
y_1^s&=\varphi^{-n+i}(y)\\
y_2^s&=h^s\circ \varphi^{-n+i}(y)\\
y_1^u&=\varphi^{2n}\circ h^s \circ \varphi^{-n+i}(y)\\
y_2^u&=h^u\circ \varphi^{2n}\circ h^s \circ \varphi^{-n+i}(y)\\
y^{s,u}&=\varphi^{-n}\circ h^u\circ \varphi^{2n}\circ h^s \circ \varphi^{-n+i}(y),
\intertext{and for every $j\in \{1,\ldots, \# \mathcal{R}_{2m}^\delta\}$ with $f_{2m,j}(\varphi^{-r}(y))\neq 0$, the same holds for the expressions}
y_{1,j}^s&=\varphi^{-n+i}[\varphi^r(g_{2m,j}),y]\\
y_{2,j}^s&=h^s\circ \varphi^{-n+i}[\varphi^r(g_{2m,j}),y]\\
y_{1,j}^u&=\varphi^{n+i}[y,\varphi^r(g_{2m,j})]\\
y_{2,j}^u&=h^u\circ \varphi^{n+i}[y,\varphi^r(g_{2m,j})].
\end{align*}
Moreover, the operator $\iota_{2m,r+i}^*(\alpha_u^{-n}(b)u^i\otimes \alpha_s^n(a)u^i)\iota_{2m,r}-\alpha_u^{-n}(b)\alpha_s^n(a)u^i $ evaluated on the basis vector $\delta_y$ is given by $$(\sum_{j}f_{2m,j}(\varphi^{-r}(y))f_{2m,j}(\varphi^{-r-i}(y^{s,u}))b(y_{2,j}^u,y_{1,j}^u)a(y_{2,j}^s,y_{1,j}^s)-b(y_2^u,y_1^u)a(y_2^s,y_1^s))\delta_{y^{s,u}},$$ where the sum is taken over the $j\in \{1,\ldots, \# \mathcal{R}_{2m}^\delta\}$ with $f_{2m,j}(\varphi^{-r}(y))\neq 0$. As a result, for $n\geq n_3$, we have that $$\|\iota_{2m,r+i}^*(\alpha_u^{-n}(b)u^i\otimes \alpha_s^n(a)u^i)\iota_{2m,r}-\alpha_u^{-n}(b)\alpha_s^n(a)u^i\|$$ $$=\sup_y |\sum_{j}f_{2m,j}(\varphi^{-r}(y))f_{2m,j}(\varphi^{-r-i}(y^{s,u}))b(y_{2,j}^u,y_{1,j}^u)a(y_{2,j}^s,y_{1,j}^s)-b(y_2^u,y_1^u)a(y_2^s,y_1^s)|,$$ where the supremum is taken over the $y\in X^h(P,Q)$ such that $\alpha_u^{-n}(b)\alpha_s^n(a)u^i \delta_y \neq 0$ or $(\alpha_u^{-n}(b) u^{i} \otimes \alpha_s^n(a) u^{i})\iota_{2m,r}(\delta_y)\neq 0$.

By writing $f_{2m,j}(\varphi^{-r-i}(y^{s,u}))$ as $f_{2m,j}(\varphi^{-r-i}(y^{s,u}))-f_{2m,j}(\varphi^{-r}(y))+f_{2m,j}(\varphi^{-r}(y))$, we first consider $$|\sum_{j}f_{2m,j}(\varphi^{-r}(y))(f_{2m,j}(\varphi^{-r-i}(y^{s,u}))-f_{2m,j}(\varphi^{-r}(y)))b(y_{2,j}^u,y_{1,j}^u)a(y_{2,j}^s,y_{1,j}^s)|$$ $$\leq \|b\|_{\infty}\|a\|_{\infty}\sum_j|f_{2m,j}(\varphi^{-r-i}(y^{s,u}))-f_{2m,j}(\varphi^{-r}(y))|.$$ Following Theorem \ref{thm:theoremgraphSmalespaces}, the number of $j\in \{1,\ldots, \# \mathcal{R}_{2m}^\delta\}$ with $f_{2m,j}(\varphi^{-r}(y))\neq 0$ is at most $(\# \mathcal{R}_1^\delta)^2$. Also, for such $j$ it holds that 
\begin{align*}
|f_{2m,j}(\varphi^{-r-i}(y^{s,u}))-f_{2m,j}(\varphi^{-r}(y))|&\leq |f_{2m,j}^2(\varphi^{-r-i}(y^{s,u}))-f_{2m,j}^2(\varphi^{-r}(y))|^{1/2}\\
&\leq \Lip(f_{2m,j}^2\circ \varphi^{-r})^{1/2}d(\varphi^{-i}(y^{s,u}),y)^{1/2}.
\end{align*}
Again from Theorem \ref{thm:theoremgraphSmalespaces}, since the metric $d$ is self-similar, there is an $0<\eta\leq \varepsilon_X$ so that $\Leb(\mathcal{R}_{2m}^{\delta})\geq \lambda^{-2m+1}\eta$. Then, using the fact that $\varphi$ is $\lambda$-bi-Lipschitz and that $|r|\leq m\leq 2\gamma_{2n+k}< n/4+k/8+2$, from Proposition \ref{prop:Lip_pou_con} we obtain a constant $K_1>0$ that $$\Lip(f_{2m,j}^2\circ \varphi^{-r})<K_1\lambda^{3n/4}.$$ Moreover, from (\ref{eq:KK_1_lift_convergence_1}) we have that $d(y^{s,u},\varphi^i(y))\leq K_0\lambda^{-n}$ and hence $$d(\varphi^{-i}(y^{s,u}),y)\leq K_0\lambda^{|i|}\lambda^{-n}.$$ Consequently, it holds that 
\begin{equation}\label{eq:KK_1_lift_convergence_7}
|\sum_{j}f_{2m,j}(\varphi^{-r}(y))(f_{2m,j}(\varphi^{-r-i}(y^{s,u}))-f_{2m,j}(\varphi^{-r}(y)))b(y_{2,j}^u,y_{1,j}^u)a(y_{2,j}^s,y_{1,j}^s)|< K_2\lambda^{-n/8},
\end{equation}
where $K_2=\|b\|_{\infty}\|a\|_{\infty}(\# \mathcal{R}_1^\delta)^2 (K_0K_1\lambda^{|i|})^{1/2}.$

For the other term of the triangle inequality we use the fact that $$\sum_{j}f_{2m,j}(\varphi^{-r}(y))^2=1,$$ and consider $$|\sum_{j}f_{2m,j}(\varphi^{-r}(y))^2(b(y_{2,j}^u,y_{1,j}^u)a(y_{2,j}^s,y_{1,j}^s)-b(y_2^u,y_1^u)a(y_2^s,y_1^s))|$$ $$\leq \sum_{j}f_{2m,j}(\varphi^{-r}(y))^2|b(y_{2,j}^u,y_{1,j}^u)a(y_{2,j}^s,y_{1,j}^s)-b(y_2^u,y_1^u)a(y_2^s,y_1^s)|.$$ Since $a,b$ are Lipschitz, for every $j$ we have that 
\begin{align*}
|a(y_{2,j}^s,y_{1,j}^s)-a(y_2^s,y_1^s)|&\leq \Lip(a)D_{s,d}((y_{2,j}^s,y_{1,j}^s), (y_2^s,y_1^s)),\\
|b(y_{2,j}^u,y_{1,j}^u)-b(y_2^u,y_1^u)|&\leq \Lip(b)D_{u,d}((y_{2,j}^u,y_{1,j}^u),(y_2^u,y_1^u)).
\end{align*}
To estimate $D_{s,d}((y_{2,j}^s,y_{1,j}^s), (y_2^s,y_1^s))$ recall that $y_1^s=\varphi^{-n+i}(y),\, y_2^s=h^s\circ \varphi^{-n+i}(y)$ and also $y_{1,j}^s=\varphi^{-n+i}[\varphi^r(g_{2m,j}),y],\, y_{2,j}^s=h^s\circ \varphi^{-n+i}[\varphi^r(g_{2m,j}),y]$. Moreover, note that $n$ is large enough so that \[
\varphi^n(y_1^s)\in X^s(\varphi^n(y_2^s),\varepsilon_X'/2) \quad \text{ and } \quad  \varphi^n(y_{1,j}^s)\in X^s(\varphi^n(y_{2,j}^s),\varepsilon_X'/2).
\]
We aim to show that $(y_{2,j}^s,y_{1,j}^s)$ lies in a sufficiently small bisection around $(y_2^s,y_1^s)$, and this requires to be precise about the distance of $y_1^s$ with $y_{1,j}^s$. We have that both $y,\, \varphi^r(g_{2m,j})$ lie in the rectangle $\varphi^{r}(R_{2m,j}^{\delta})\in \varphi^r(\mathcal{R}_{2m}^{\delta})$. From Theorem \ref{thm:theoremgraphSmalespaces} we see that the cover $\varphi^r(\mathcal{R}_{2m}^{\delta})$ refines $\mathcal{R}_{\gamma_{2n+k}}^{\delta}$ and using the fact that $\gamma_{2n+k}\geq n/8+k/16$, we have that 
\[
\text{diam}(\varphi^r(R_{2m,j}^{\delta}))\leq K_3\lambda^{-n/8},
\]
where $K_3=\lambda^{1-k/16}\varepsilon_X.$ Assuming that $n$ is sufficiently large so that $K_3\lambda^{-n/8}<\varepsilon_X$, one has that $[\varphi^r(g_{2m,j}),y]\in X^u(y,K_3\lambda^{-n/8})$ and hence $$\varphi^{-n+i}[\varphi^r(g_{2m,j}),y]\in X^u(\varphi^{-n+i}(y), \lambda^{i}K_3\lambda^{-n-n/8}).$$ Now, assuming that $n$ is a bit larger so that $\lambda^{i}K_3\lambda^{-n/8}<\varepsilon_X'/2$, we can consider the bisection $V^s(y_2^s,y_1^s,\lambda^{i}K_3\lambda^{-n-n/8},n)$ and it is straightforward to show that $$(y_{2,j}^s,y_{1,j}^s)\in V^s(y_2^s,y_1^s,\lambda^{i}K_3\lambda^{-n-n/8},n).$$ Then, using \cite[Proposition 6.5]{Gero2} we can find a constant $K_4>0$ (independent of the points $(y_{2,j}^s,y_{1,j}^s), (y_2^s,y_1^s)$ and $n$)  so that $$D_{s,d}((y_{2,j}^s,y_{1,j}^s), (y_2^s,y_1^s))\leq K_4 2^{-n/(8\lceil \log_{\lambda}3 \rceil)}.$$ Working in exactly the same way, but assuming $n$ is slightly larger than before, we can find a constant $K_5>0$ so that $$D_{u,d}((y_{2,j}^u,y_{1,j}^u),(y_2^u,y_1^u))\leq K_5 2^{-n/(8\lceil \log_{\lambda}3 \rceil)}.$$ To conclude, there is a constant $K_6>0$ (independent of $y$ and $n$) so that 
\begin{equation}\label{eq:KK_1_lift_convergence_8}
|\sum_{j}f_{2m,j}(\varphi^{-r}(y))^2(b(y_{2,j}^u,y_{1,j}^u)a(y_{2,j}^s,y_{1,j}^s)-b(y_2^u,y_1^u)a(y_2^s,y_1^s))|\leq K_6 2^{-n/(8\lceil \log_{\lambda}3 \rceil)}.
\end{equation}

As a result, from (\ref{eq:KK_1_lift_convergence_7}) and (\ref{eq:KK_1_lift_convergence_8}), if $n_3$ is sufficiently large, for every $n\geq n_3$ we have that $$\|\iota_{2m,r+i}^*(\alpha_u^{-n}(b)u^i\otimes \alpha_s^n(a)u^i)\iota_{2m,r}-\alpha_u^{-n}(b)\alpha_s^n(a)u^i\|\leq C_2(\lambda^{-n/8}+2^{-n/(8\lceil \log_{\lambda}3 \rceil)}),$$ where $C_2=\max\{K_2,K_6\}$. If either $a$ or $b$ is not supported on a bisection, then the same result holds for large enough $C_2, n_3.$ This completes the proof.
\end{proof}

The next result is an application of Lemmas \ref{lem:KK_1_lift_orthogonality} and \ref{lem:KK_1_lift_convergence}.

\begin{lemma}\label{lem:KK_1_lift_convergence2}
Let $a\in \Lip_c(G^s(Q),D_{s,d}),\, b\in \Lip_c(G^u(P),D_{u,d})$ and $i,k,l\in \mathbb Z$. Then, there is a constant $C_3>0$ and $n_4\in \mathbb N$ so that, for every $n\geq n_4$, we have 
\[
\|V_{2n+l}^*(\alpha_u^{-n}(b) u^{i} \otimes \alpha_s^n(a) u^{i}) V_{2n+k} -\alpha_u^{-n}(b)\alpha_s^n(a)u^{i}\|\leq \frac{C_3}{n}.
\]
\end{lemma}

\begin{proof}
Without loss of generality we can assume that $l\leq k$ and $i\geq 0$. Let us consider $n\in \mathbb N$ large enough so that $2n+k,\, 2n+l>0,\,\gamma_{2n+k}\geq i+1$. Then, it holds that $V_{2n+l}^*(\alpha_u^{-n}(b) u^{i} \otimes \alpha_s^n(a) u^{i}) V_{2n+k}$ is equal to $$c_{\gamma_{2n+l}}^{-1}c_{\gamma_{2n+k}}^{-1}\sum_{m'=\gamma_{2n+l}}^{2\gamma_{2n+l}}\sum_{r'=-m'}^{m'}\sum_{m=\gamma_{2n+k}}^{2\gamma_{2n+k}}\sum_{r=-m}^{m}\iota_{2m',r'}^*(\alpha_u^{-n}(b) u^{i} \otimes \alpha_s^n(a) u^{i})\iota_{2m,r}.$$ From the \enquote{orthogonality} Lemma \ref{lem:KK_1_lift_orthogonality} we can find $n_2\in \mathbb N$ so that, for every $n\geq n_2$, we have $\gamma_{2n+k}\leq 2\gamma_{2n+l}$ and the last expression is equal to $$c_{\gamma_{2n+l}}^{-1}c_{\gamma_{2n+k}}^{-1}\sum_{m=\gamma_{2n+k}}^{2\gamma_{2n+l}}\sum_{r=-m}^{m-i}\iota_{2m,r+i}^*(\alpha_u^{-n}(b) u^{i} \otimes \alpha_s^n(a) u^{i})\iota_{2m,r}.$$ At this point, let us consider the numbers $\zeta_n\in (0,1]$ given by $$\zeta_n=c_{\gamma_{2n+l}}^{-1}c_{\gamma_{2n+k}}^{-1}\sum_{m=\gamma_{2n+k}}^{2\gamma_{2n+l}}\sum_{r=-m}^{m-i}1,$$ and with elementary computations one can see that there is a constant $K>0$ such that $1-\zeta_n\leq K/n$. Then, we have that 
\begin{equation}\label{eq:KK_1_lift_convergence2_1}
\|(1-\zeta_n)\alpha_u^{-n}(b)\alpha_s^n(a)u^{i}\|\leq \frac{K\|b\|\|a\|}{n}.
\end{equation}
From Lemma \ref{lem:KK_1_lift_convergence}, there is a constant $C_2>0$ and $n_3\in \mathbb N$ so that, if in addition $n\geq n_3$ then 
\begin{equation}\label{eq:KK_1_lift_convergence2_2}
\|V_{2n+l}^*(\alpha_u^{-n}(b) u^{i} \otimes \alpha_s^n(a) u^{i}) V_{2n+k} -\zeta_n \alpha_u^{-n}(b)\alpha_s^n(a)u^{i}\|\leq C_2(\lambda^{-n/8}+2^{-n/(8\lceil \log_{\lambda}3 \rceil)}).
\end{equation}
The proof follows from (\ref{eq:KK_1_lift_convergence2_1}) and (\ref{eq:KK_1_lift_convergence2_2}).
\end{proof}
 
\begin{proof}[Proof of Theorem \ref{thm:KK_1_lift_Ruelle}]

Recall the representations $\overline{\rho_s}:\mathcal{R}^s(Q)\to \mathcal{B}(\mathscr{H}\otimes \ell^2(\mathbb Z))$ and $\overline{\rho_u}:\mathcal{R}^u(P)\to \mathcal{B}(\mathscr{H}\otimes \ell^2(\mathbb Z))$ in \eqref{eq:inflatedstableRuelle} and \eqref{eq:inflatedusntableRuelle}, which commute modulo compacts and whose product in the Calkin algebra gives $\tau_{\Delta}$. Also, recall the unitary $U=\bigoplus_{n\in \mathbb Z} u^{-\lfloor n/2 \rfloor}$ and the symmetrised extension $\ad_{\pi}(U)\circ \tau_{\Delta}$ in (\ref{eq:tau_Delta_sym}) which on $x\in \mathcal{R}^s(Q)\otimes_{\text{alg}} \mathcal{R}^u(P)$ is given by
\begin{equation}\label{eq:KK_1_lift_eq_2}
(\ad_{\pi}(U)\circ \tau_{\Delta})(x)=(\ad(U)\circ (\overline{\rho_s}\cdot \overline{\rho_u}))(x)+\mathcal{K}(\mathscr{H}\otimes \ell^2(\mathbb Z)).
\end{equation}
In \cite[Proposition 7.9]{Gero2}, it is proved that for every $p>2\ent(\varphi)\lceil \log_{\lambda}3 \rceil/ \log 2$, the algebras $\overline{\rho_s}(\Lambda_{s,d}(Q,\alpha_s))$ and $\overline{\rho_u}(\Lambda_{u,d}(P,\alpha_u))$ commute modulo the Schatten $p$-ideal 
$\mathcal{I}=\mathcal{L}^p(\mathscr{H}\otimes \ell^2(\mathbb Z))$. Therefore, for every $x,y\in \Lambda_{s,d}(Q,\alpha_s)\otimes_{\text{alg}}\Lambda_{u,d}(P,\alpha_u)$, keeping in mind that $U\in \mathcal{B}(\mathscr{H}\otimes \ell^2(\mathbb Z))$, we have that 
\begin{equation}\label{eq:KK_1_lift_eq_3}
\begin{split}
(\ad(U)\circ(\overline{\rho_s}\cdot \overline{\rho_u}))(xy)-(\ad(U)\circ(\overline{\rho_s}\cdot \overline{\rho_u}))(x)(\ad(U)\circ(\overline{\rho_s}\cdot \overline{\rho_u}))(y)&\in \mathcal{I}\\
\text{and}\enspace (\ad(U)\circ(\overline{\rho_s}\cdot \overline{\rho_u}))(x^*)-(\ad(U)\circ(\overline{\rho_s}\cdot \overline{\rho_u}))(x)^*&\in \mathcal{I}.
\end{split}
\end{equation}
Consequently, the extension $\ad_{\pi}(U)\circ \tau_{\Delta}$ is $p$-smooth on $\Lambda_{s,d}(Q,\alpha_s)\otimes_{\text{alg}}\Lambda_{u,d}(P,\alpha_u)$.

Further, from Lemma \ref{lem:KK_1_lift_ranks}, Lemma \ref{lem:KK_1_lift_convergence2} and \cite[Lemma 7.4]{Gero2} we have that for all $x\in \Lambda_{s,d}(Q,\alpha_s)\otimes_{\text{alg}}\Lambda_{u,d}(P,\alpha_u)$:
\begin{equation}\label{eq:KK_1_lift_eq_4}
V^*\rho(x)V-(\ad(U)\circ (\overline{\rho_s}\cdot \overline{\rho_u}))(x)\in \Li (\mathscr{H}\otimes \ell^2(\mathbb Z)).
\end{equation}
Then, following Theorem \ref{thm:Liftingtheorem}, the triple $(\mathscr{H}\otimes \mathscr{H}\otimes \ell^2(\mathbb Z), \rho, 2V V^*-1)$ defines an odd Fredholm module over $\mathcal{R}^s(Q)\otimes \mathcal{R}^u(P)$. In particular, it is $\theta$-summable on $\Lambda_{s,d}(Q,\alpha_s)\otimes_{\text{alg}}\Lambda_{u,d}(P,\alpha_u)$ and represents the extension class $[\ad_{\pi}(U)\circ \tau_{\Delta}]=[\tau_{\Delta}]$, hence represents the fundamental class $\Delta$ from \cite{KPW}.
\end{proof}

\begin{Acknowledgements}
The first author would like to thank Siegfried Echterhoff, Heath Emerson, James Gabe, Magnus Goffeng, Bram Mesland, Ian Putnam and Christian Voigt for many useful conversations on index theory. Also, this author thanks EPSRC (grants NS09668/1, M5086056/1) and the London Mathematical Society together with the Heilbronn Institute for Mathematical Research (Early Career Fellowship) for supporting this research.
\end{Acknowledgements}

\noindent
Conflict of Interest Statement: On behalf of all authors, the corresponding author states that there is no conflict of interest.


\begin{thebibliography}{99}

\bibitem{AKM}
R. L. Adler, A. G. Konheim, M. H. McAndrew;  Topological entropy, \textit{Trans. Amer. Math. Soc.} \textbf{114} (1965), 309\texttt{-}319.

\bibitem{Artigue}
A. Artigue; Self-similar hyperbolicity, \textit{Erg. Th. Dyn. Sys.} \textbf{38} (2018), 2422\texttt{-}2446.

\bibitem{Barreira}
L. Barreira; A non-additive thermodynamic formalism and applications to dimension theory of hyperbolic dynamical systems, \textit{Erg. Th. Dyn. Sys.} \textbf{16} (1996), 871\texttt{-}927.

\bibitem{BG}
J. C. Becker, D. H. Gottlieb; \textit{A History of Duality in Algebraic Topology}, North\texttt{-}Holland, 1999.

\bibitem{Bell}
G. C. Bell; Property A for groups acting on metric spaces, \textit{Topol. Appl.} \textbf{130} (2003), 239\texttt{-}251. 

\bibitem{Blackadar}
B. Blackadar; \textit{$\Kt$-Theory for Operator Algebras}, Cambridge Univ. Press, Cambridge, 1998.

\bibitem{Bowen2}
R. Bowen; Markov partitions for Axiom A diffeomorphisms, \textit{Amer. J. Math.} \textbf{92} (1970), 725\texttt{-}747.

\bibitem{Bowen3}
R. Bowen; Markov partitions and minimal sets for Axiom A diffeomorphisms, \textit{Amer. J. Math.} \textbf{92} (1970), 907\texttt{-}918. 

\bibitem{Bowen4}
R. Bowen; Periodic points and measures for Axiom A diffeomorphisms, \textit{Trans. Amer. Math. Soc.} \textbf{154} (1971), 377\texttt{-}397.

\bibitem{Connes}
A. Connes; \textit{Noncommutative Geometry}, Academic Press Inc., London and San Diego, 1994.

\bibitem{DGMW}
R. Deeley, M. Goffeng, B. Mesland, M. F. Whittaker; Wieler solenoids, Cuntz–Pimsner algebras and $\Kt$-theory, \textit{Erg. Th. Dyn. Sys.} \textbf{38} (2018), 2942\texttt{-}2988.

\bibitem{DGY}
R. Deeley, M. Goffeng, A. Yashinski; Smale space $C^*$-algebras have nonzero projections, \textit{Proc. Amer. Math. Soc.} \textbf{148} (2020), 1625\texttt{-}1639. 

\bibitem{DKW}
R. Deeley, B. R. Killough, M. F. Whittaker; Dynamical correspondences for Smale spaces, \textit{New York J. Math.} \textbf{22} (2016), 943\texttt{-}988.

\bibitem{DS2}
R. Deeley, K. Strung; Nuclear dimension and classification of $C^*$-algebras associated to
Smale spaces, \textit{Trans. Amer. Math. Soc.} \textbf{370} (2018), 3467\texttt{-}3485.

\bibitem{DY}
R. Deeley, A. Yashinski; The stable algebra of a Wieler Solenoid: Inductive limits and $\Kt$-theory, \textit{Erg. Th. Dyn. Sys.} \textbf{40} (2020), 2734\texttt{-}2768.

\bibitem{Douglas}
R. G. Douglas; On the smoothness of elements of $\Ext$, \textit{Topics in modern operator theory}, Birkh{\"a}user Verlag (1981), 63\texttt{-}69.

\bibitem{DV}
R. G. Douglas, D. Voiculescu; On the smoothness of sphere extensions, \textit{J. Operator Th.} \textbf{6} (1981), 103\texttt{-}111.

\bibitem{DE} A. Duwenig, H. Emerson; Transversals, duality and irrational rotation, \textit{Trans. Amer. Math. Soc. B} \textbf{7} (2020), 254\texttt{-}289.

\bibitem{EEK2}
S. Echterhoff, H. Emerson, H. J. Kim; $\KKt$-theoretic duality for proper twisted actions, \textit{Math. Ann.} \textbf{340} (2008), 839\texttt{-}873.

\bibitem{EEK}
S. Echterhoff, H. Emerson, H. J. Kim; A Lefschetz fixed-point formula for certain orbifold $C^*$-algebras, \textit{J. Noncomm. Geom.} \textbf{4} (2010), 125\texttt{-}155.

\bibitem{EmersonDuality}
H. Emerson; Noncommutative Poincar{\'e} duality for boundary actions of  hyperbolic groups, \textit{J. Reine Angew. Math.} \textbf{564} (2003), 1\texttt{-}33.

\bibitem{EmersonLef}
H. Emerson; Lefschetz numbers for $C^*$-algebras, \textit{Canad. Math. Bull.} \textbf{54} (2011), 82\texttt{-}99.

\bibitem{Fathi}
A. Fathi; Expansiveness, hyperbolicity and Hausdorff dimension, \textit{Comm. Math. Phys.} \textbf{126} (1989), 249\texttt{-}262.

\bibitem{Fried}
D. Fried; Finitely presented dynamical systems, \textit{Erg. Th. Dyn. Sys.} \textbf{7} (1987), 489\texttt{-}507.

\bibitem{Gero}
\text{D. M.} Gerontogiannis; Ahlfors regularity and fractal dimension of Smale spaces, \textit{Erg. Th. Dyn. Sys.} \textbf{42} (2022), 2281\texttt{-}2332. 

\bibitem{Gero2}
\text{D. M.} Gerontogiannis; On finitely summable Fredholm modules from Smale spaces, \textit{Trans. Amer. Math. Soc.} \textbf{375} (2022), 8885\texttt{-}8944.

\bibitem{Gero_Thesis}
D. M. Gerontogiannis; \textit{Ahlfors regularity, extensions by Schatten ideals and a geometric fundamental class of Smale space $C^*$-algebras using dynamical partitions of unity,} Ph.D. Thesis, Univ. of Glasgow, 2021.

\bibitem{Goffeng}
M. Goffeng; Equivariant extensions of $*$-algebras, \textit{New York J. Math.} \textbf{16} (2010), 369\texttt{-}385.

\bibitem{GM}
M. Goffeng, B. Mesland; Spectral triples and finite summability on Cuntz-Krieger algebras, \textit{Doc. Math.} \textbf{20} (2015), 89\texttt{-}170.

\bibitem{HR}
N. Higson, J. Roe; \textit{Analytic $\Kt$-homology}, Oxford Univ. Press, Oxford, 2000.

\bibitem{HRor}
J. Hjelmborg, M. R{\o}rdam, On stability of $C^*$-algebras, \textit{J. Funct. Anal.} \textbf{155} (1998), 153\texttt{-}170.

\bibitem{KP}
J. Kaminker, I. F. Putnam; $\Kt$-theoretic duality of shifts of finite type, \textit{Comm. Math. Phys.} \textbf{187} (1997), 509\texttt{-}522. 

\bibitem{KPW}
J. Kaminker, I. F. Putnam, M. F. Whittaker; $\Kt$-theoretic duality for hyperbolic dynamical systems, \textit{J. Reine Angew. Math.} \textbf{730} (2017), 263\texttt{-}299.

\bibitem{KS}
J. Kaminker, C. L. Schochet; Spanier-Whitehead $\Kt$-duality for $C^*$-algebras, \textit{J. Topol. Anal.} \textbf{11} (2019), 21\texttt{-}52. 

\bibitem{Kasparov5}
G. G. Kasparov; The operator $\Kt$-functor and extensions of $C^*$-algebras., \textit{Math. USSR, Izv.} \textbf{16} (1981), 513\texttt{-}572.

\bibitem{Kasparov2}
G. G. Kasparov; Equivariant KK-theory and the Novikov Conjecture, \textit{Invent. Math.} \textbf{91} (1988), 147\texttt{-}201.

\bibitem{LacaSp}
M. Laca, J. Spielberg; Purely Infinite $C^*$-algebras from Boundary Actions of Discrete Groups, \textit{J. Reine Angew. Math.} \textbf{480} (1996), 125\texttt{-}139.

\bibitem{Lance}
E. C. Lance; \textit{Hilbert $C^*$-modules \texttt{-} a toolkit for operator algebraists}, Cambridge Univ. Press, Cambridge, 1995.

\bibitem{MRW}
P. S. Muhly, J. N. Renault, D. P. Williams; Equivalence and isomorphism for groupoid $C^*$-algebras, \textit{J. Operator Th.} \textbf{17} (1987), 3\texttt{-}22.

\bibitem{NP}
S. Nishikawa, V. Proietti; Groups with Spanier-Whitehead duality, \textit{Ann. K-theory} \textbf{5} (2020), 465\texttt{-} 500.

\bibitem{KirP}
N. C. Phillips; A Classification Theorem for Nuclear Purely Infinite Simple $C^*$-algebras, \textit{Doc. Math.} \textbf{5} (2000), 49\texttt{-}114.

\bibitem{PZ}
I. Popescu, J. Zacharias; E-theoretic duality for higher rank graph algebras, \textit{K-Theory} \textbf{34} (2005), 265\texttt{-}282.

\bibitem{PY}
V. Proietti, M. Yamashita; Homology and K-theory of dynamical systems III. Beyond stably disconnected Smale spaces, arXiv:2207.03118 (2022), preprint.

\bibitem{Putnam_Book}
I. F. Putnam; A Homology Theory for Smale Spaces, \textit{Mem. Amer. Math. Soc.} \textbf{232} (2014), viii+122.

\bibitem{Putnam_algebras}
I. F. Putnam; $C^*$-algebras from Smale spaces, \textit{Canad. J. Math.} \textbf{48} (1996), 175\texttt{-}195.

\bibitem{PS}
I. F. Putnam, J. Spielberg; The Structure of $C^*$-Algebras Associated with Hyperbolic Dynamical Systems, \textit{J. Func. Analysis} \textbf{163} (1999), 279\texttt{-}299.

\bibitem{RRS}
A. Rennie, D. Robertson, A. Sims; Poincar{\'e} duality for Cuntz-Pimsner algebras of bimodules, \textit{Adv. Math.} \textbf{347} (2019), 1112\texttt{-}1172.

\bibitem{Ruelle} 
D. Ruelle; \textit{Thermodynamic Formalism}, Cambridge Univ. Press, Cambridge, 2004.

\bibitem{Ruelle_algebras}
D. Ruelle; Noncommutative algebras for hyperbolic diffeomorphisms, \textit{Invent. Math.} \textbf{93} (1988), 1\texttt{-}13.

\bibitem{Smale}
S. Smale; Differentiable dynamical systems, \textit{Bull. Amer. Math. Soc.} \textbf{73} (1967), 747\texttt{-}817.

\bibitem{Walters}
P. Walters; \textit{An Introduction to Ergodic Theory}, Springer, New York, 1982.

\bibitem{Whitney}
H. Whitney; The self-intersections of a smooth n-manifolds in $2n$-space, \textit{Ann. of Math.} \textbf{45} (1944), 220\texttt{-}246.

\bibitem{Whittaker_PhD}
M. F. Whittaker; \textit{Poincar{\'e} Duality and Spectral Triples for Hyperbolic Dynamical Systems}, Ph.D. Thesis, Univ. of Victoria, 2010.

\bibitem{Wieler}
S. Wieler; \textit{Smale spaces with totally disconnected local stable sets}, Ph.D. Thesis, Univ. of Victoria, 2012.


\end{thebibliography}
\end{document}